\documentclass[onefignum,onetabnum]{siamart171218}

\usepackage[numbers,sort&compress]{natbib}

\usepackage{mathtools}

\usepackage{amsmath}
\usepackage{amssymb}
\usepackage{graphicx}
\usepackage{multirow}
\usepackage{enumerate}

\usepackage{algpseudocode}
\usepackage{algorithm}

\usepackage{soul}

\DeclareMathOperator{\spanop}{span}

\newsiamthm{conjecture}{Conjecture}
\newsiamthm{question}{Question}
\newsiamthm{answer}{Answer}
\newsiamremark{example}{Example}
\newsiamremark{remark}{Remark}

\newcommand{\LL}{\mathsf{S}}
\newcommand{\NullL}{\mathrm{Null}\LL}
\newcommand{\VV}{\mathbb{V}}

\newcommand{\rd}{\mathrm{d}}

\newcommand{\Amat}{\mathsf{A}}

\newcommand{\Ymat}{\mathsf{Y}}
\newcommand{\Umat}{\mathsf{U}}
\newcommand{\Vmat}{\mathsf{V}}
\newcommand{\Qmat}{\mathsf{Q}}
\newcommand{\Gmat}{\mathsf{G}}

\newcommand{\Mmat}{\mathsf{M}}
\newcommand{\Wmat}{\mathsf{W}}
\newcommand{\Imat}{\mathsf{I}}
\newcommand{\Pmat}{\mathsf{P}}
\newcommand{\Kcal}{\mathcal{K}}

\newcommand{\vep}{\varepsilon}
\newcommand{\ran}{\text{ran\,}}
\newcommand{\diag}{\text{diag\,}}

\newcommand{\wt}[1]{\widetilde{#1}}

\newcommand{\kcresolved}[1]{}

\algdef{SE}[SUBALG]{Indent}{EndIndent}{}{\algorithmicend\ }%
\algtext*{Indent}
\algtext*{EndIndent}

\headers{Random sampling and efficient algorithms for multiscale PDEs}{K. Chen, Q. Li, J. Lu, and S.J. Wright}

\title{Random sampling and efficient algorithms for multiscale PDEs\thanks{
\funding{The work of Q.L.~is supported in part by a start-up fund from
	UW-Madison and National Science Foundation under the grant
	DMS-1619778. The work of J.L.~is supported in part by the National
	Science Foundation under award DMS-1454939 and KI-Net
	RNMS-1107444. The work of S.W.~is supported in part by NSF awards
	IIS-1447449, 1628384, and 1634597 and AFOSR Award FA9550-13-1-0138,
	and Subcontract 3F-30222 from Argonne National Laboratory.  K.C.,
	Q.L., and S.W. are supported by NSF award 1740707.}}}

\author{Ke Chen\thanks{Department of Mathematics, University of Texas at Austin, 2515 Speedway, Austin, TX 78712 USA.
		\email{kechen@math.utexas.edu}}
\and Qin Li\thanks{Mathematics Department, University of Wisconsin-Madison, 480 Lincoln Dr., Madison, WI 53706 USA.
		\email{qinli@math.wisc.edu}}
\and Jianfeng Lu\thanks{Department of Mathematics, Department of Physics and Department of Chemistry, Duke University, Box 90320, Durham, NC 27708 USA.
		\email{jianfeng@math.duke.edu}}
\and Stephen J. Wright\thanks{Computer Sciences Department, University of Wisconsin-Madison, 1210 W Dayton St, Madison, WI 53706 USA.
		\email{swright@cs.wisc.edu}} }
	
\begin{document}
	\maketitle
\begin{abstract}
We describe a numerical framework that uses random sampling to efficiently capture low-rank local solution spaces of multiscale PDE problems arising in domain decomposition. In contrast to existing techniques, our method does not rely on detailed analytical understanding of specific multiscale PDEs, in particular, their asymptotic limits. We present the application of the framework on two examples --- a linear kinetic equation and an elliptic equation with rough media. On these two examples, this framework achieves the asymptotic preserving property for the kinetic equations and numerical homogenization for the elliptic equations.
\end{abstract}



\section{Introduction}
Partial differential equations (PDEs) that involve multiple temporal
and spatial scales are numerically challenging to solve. The current
generation of efficient solvers exploits the analytic solution
structures that are intrinsic to each specific multiscale problem.
In this work, we exploit instead the ``low-rank'' property of the
solution spaces that is common to many multiscale problems that are ``homogenizable"
, and design
a general framework in which analytic structures of solutions are
discovered automatically by the algorithms without the need for any
problem-specific analysis.

We consider the following boundary value problem:
\begin{equation}\label{eqn:general_u_ep}
\mathcal{L}^\vep u^\vep = 0\,,\quad\text{with}\quad\mathcal{B}u^\vep = f\,,
\end{equation}
where $\mathcal{L}^\varepsilon$ is a linear PDE operator with
multiscale structure, with $\varepsilon$ representing the small
scale. $\mathcal{B}$ is the boundary operator and $f$ is the boundary
condition. The solution $u^\varepsilon$ contains information at both
coarse scale $x$ and fine scale ${x}/{\varepsilon}$.  A naive
numerical scheme for \eqref{eqn:general_u_ep} would require a fine
discretization: The mesh size $h$ must resolve $\vep$ (that is,
$h\ll\vep$), and thus the number of grid points (the degrees of
freedom) $N_\vep$ is of the order of $\vep^{-d}$, with $d$ being the
dimension of the problem. For small $\vep$, the computational cost is
prohibitive. These observations have motivated research into
algorithms for multiscale PDE problems that are much more efficient
than such naive schemes.

One strategy commonly used by efficient algorithms is to exploit the
asymptotic behavior of the multiscale problems as $\vep \to 0$. In
particular, the ``effective equations'' that capture the behavior of
the solution as $\vep$ approaches zero have been derived for several
specific multiscale problems. More specifically, we seek a homogenized
operator $\mathcal{L}^\ast$, with no dependence on $\varepsilon$, such
that the solution $u^\ast$ of the ``effective equation''
\begin{equation}\label{eqn:general_u_ast}
\mathcal{L}^\ast u^\ast = 0\,,\quad\text{with}\quad\mathcal{B}u^\ast = f\,
\end{equation}
satisfies
\begin{equation}\label{eqn:convergence}
\|u^\varepsilon-u^\ast\| \to 0 \quad  \mbox{as $\varepsilon \to 0$}
\end{equation}
in a proper norm. Since $u^\ast$ is asymptotically equivalent to $u^\vep$
\eqref{eqn:convergence} and no small-scale oscillation is present,
solving \eqref{eqn:general_u_ast} can typically be done in a much more
efficient manner than directly solving \eqref{eqn:general_u_ep} with small $\vep$.


Identifying the effective operator $\mathcal{L}^\ast$, however, is mostly nontrivial. Different techniques are needed for
different equations. The hydrodynamic limit of kinetic equations is
based on moment expansions and entropic closures; 
the homogenization of elliptic equations with oscillatory media is
based on corrector equations and two-scale convergence analysis; and
the semiclassical limit of Schr\"odinger equations is based on WKB
expansion and Wigner transformations. Each of these analytical tools
leads to a different algorithmic approach, so there is a wide variation
in algorithms for different multiscale problems.

We describe in this paper a general approach to
  designing efficient algorithms for multiscale PDE problems that does
  not rely on detailed analytical knowledge of the PDE and applies to
  a wide variety of problems. Our approach not only has the advantage
  of a unified treatment, but also applies to cases in which the
  asymptotic limit is not known, or is too complicated to
  derive. (See, for example, an application in~\cite{chen2019low}.)
  While the proposed approach might not be the most effective approach
  for every multiscale problem (for example, many numerical approaches
  have been developed over the years for elliptic PDEs with rough
  coefficients), its numerical performance compares favorably with
  known approaches for particular problems. We believe that the broad
  applicability of our generic approach is a significant advantage.

Our framework is based on domain decomposition together with random
sampling to characterize the local solution space on each patch in the
decomposition.  We make use of the fact that most multiscale PDEs that
have asymptotic limits independent of small scales also have local
solution spaces of low dimension.

\begin{figure}[tb]
	\centering
	\includegraphics[width = 0.9\textwidth]{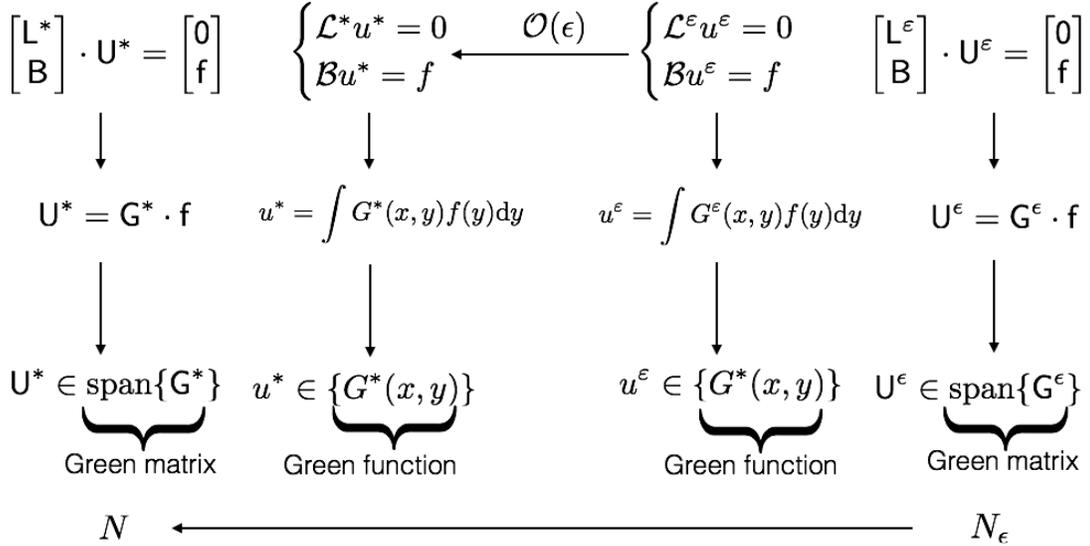}
	\vspace{-4em}
	\caption{Green's function superposition in discrete and
		continuous setting, for both $\mathcal{L}^\vep$ and the asymptotic
		limit $\mathcal{L}^\ast$. The column space of $\mathsf{G}^\vep$,
		which typically has high dimension, can be well approximated well by
		the (lower-dimensional) column space of $\mathsf{G}^\ast$.}
	\label{fig:low_rank}
\end{figure}

We illustrate the relationships between the multiscale PDE, its
discretization, and its asymptotic limits in \cref{fig:low_rank}.  The
key points of this diagram are as follows.
\begin{itemize}
	\item[1.] Both $u^\ast$ and $u^\vep$, the solutions to
	$\mathcal{L}^\ast$ and $\mathcal{L}^\vep$, respectively, are
	convolutions of Green's functions $G^{\vep/\ast}$ with the boundary
	conditions.
	\footnote{With slight danger of confusion, we adopt a generalized
		notion of ``Green's function'' in this work, which might vary from
		conventional terminology for specific PDEs. For example, for
		elliptic PDEs with Dirichlet boundary condition, the ``Green's
		function'' would be given by the Poisson kernel, that is, the
		derivative of the usual Green's function (Newtonian kernel).}
	
	\item[2.] In the discrete setting, with $\mathsf{L}^\ast$ and
	$\mathsf{L}^\vep$ denoting the discrete operators and
	$\mathsf{G}^{\ast}$ and $\mathsf{G}^{\vep}$ the corresponding
	Green's matrices, the numerical solutions $U^{\ast}$ and $U^{\vep}$
	are in the column space spanned by the respective Green's matrices.
	\item[3.] As discussed above, accurate discretization of
	$\mathcal{L}^\vep$ requires $N_\vep\sim{\vep^{-d}}$ degrees of
	freedom, while discretization of $\mathcal{L}^\ast$ usually requires
	a modest number $N$ of degrees of freedom, independent of $\vep$,
	with $N \ll N_\vep$ for interesting values of $\vep$.
\end{itemize}

\cref{fig:low_rank} suggests that if $\Umat^\ast$ and
$\Umat^\vep$ are good numerical approximations to $u^\ast$ and
$u^\vep$, respectively, and since $u^\ast$ and $u^\vep$ are close when
$\vep$ is small, then $\Umat^\ast$ and $\Umat^\vep$ should also be
close to each other. Since $\Umat^\vep$ and $\Umat^\ast$ lie in the
column spaces of $\Gmat^\vep$ and $\Gmat^\ast$ respectively, the two
matrices should therefore have similar column spaces. Without knowing
the effective equations, it may not be possible to identify
$\Gmat^\ast$ explictly, but we can still obtain essential information
contained in $\Gmat^\ast$ from $\Gmat^{\vep}$.  For this task, we need
to determine, first, how much column-space information is contained in
$\Gmat^\vep$ and, second, how to extract this information.

Regarding the first question, we define ``numerical rank'' to be the
minimum number of degrees of freedom required to capture the solution
space of a PDE to within a preset error tolerance. The concept is
closely connected to Kolmogorov $N$-width.
To address the second question, we employ random sampling: The range
of a matrix with low numerical rank can be captured by multiplying the
matrix by a set of random vectors. We adapt this strategy to sketch
the local solution space of the PDE via random sampling.

  Random sampling for numerical PDEs has been explored
  in previous works, mainly for multiscale {\em elliptic}
  equations. In particular, it has been used to construct local basis
  functions for the generalized finite element method; see
  \cite{calo2016randomized,Lipton16,Smetana_random} and our previous
  work \cite{ChenLiLuWright}, in which we report on numerical
  experiments to determine optimal sampling strategies. In
  \cite{Owhadi_bayesian,owhadi_siam_review}, numerical homogenization
  is reformulated as a Bayesian inference problem through observation
  of random samplings, where orthogonal basis functions in
  $H_0^1(\Omega)$ could be obtained by nested measurements of
  solutions or source terms. This approach is consistent with
  randomized linear algebra approaches that use random projections of
  a matrix to provide good approximations to the left/right singular-vector space
  corresponding to the largest singular values of that matrix.  Similar
  connections to randomized linear algebra have been made in
  \cite{Smetana_random, ChenLiLuWright} for numerical homogenization
  of elliptic equations.  From another perspective
  \cite{Martinsson_HSS,LLY:11}, randomized linear algebra algorithms
  are used to compress the Green's matrix of elliptic equations based
  on the framework of hierarchical matrices~\cite{Hackbusch2015}.  
Most of these works that exploit randomized sampling
consider only elliptic equations with oscillatory media, while the
method we propose in this paper applies to more general situations.
(The authors learnt about \cite{Smetana_random}
  while drafting the current paper. That work also seeks low-rank
  representations based on domain decomposition for elliptic
  equations, but it does not utilize homogenization theory, nor does
  it extend to general multiscale PDEs. As mentioned previously, the
  main value of the proposed approach is that it brings a unified
  framework for various PDE problems exhibiting multiscale features.)

The remainder of the paper is organized as follows. In~\cref{sec:AP_NH}, we review two representative case studies of
multiscale PDEs: the linear kinetic equation with small Knudsen number and
an elliptic PDE with oscillatory media. Motivated by the essential
similarity of these multiscale problems, we define in~\cref{sec:framework} the notion of numerical rank and design a
general framework for efficient algorithms based on domain
decomposition and random sampling of local solution
space.~\cref{sec:rte} and \cref{sec:elliptic} describe details
of the application of our framework to the two problems introduced in~\cref{sec:AP_NH}. Numerical results demonstrate that the
general methodology yields competitive algorithms, without the need
for detailed analytical knowledge of the specific structure of the
multiscale problems at hand.

\section{Asymptotic preserving scheme and numerical homogenization}\label{sec:AP_NH}

In this section we briefly summarize the asymptotic preserving scheme
and numerical homogenization. These approaches were developed for two
rather different multiscale problems, but they share the similar
philosophy of finding a set of ``effective equations'' that are
numerically simpler than the original PDE in some sense, and utilizing
these equations in efficient numerical solvers. These approaches are
closely related to our randomized methodology and will serve to
motivate our approach.

\subsection{Asymptotic preserving scheme for kinetic equations}\label{sec:ap}
The asymptotic preserving (AP) scheme was developed originally in the
context of numerical methods for kinetic theory.
We will explain the idea using the radiative transfer equation, a
particular linear Boltzmann equation that is a model problem in
kinetic theory.

In radiative transfer, we seek a function $u^\varepsilon(x,v)$,
defined on the phase space $(x,v)\in\Kcal\times\VV$, that represents
the density of photons at location $x$ with speed $v$.  The
equation is
\begin{equation}\label{eqn:rte_general}
-v\cdot\nabla_xu^\varepsilon + \frac{1}{\varepsilon}\mathsf{S}[u^\varepsilon] = g(x)\,,\quad (x,v)\in\Kcal\times\VV\,, 
\end{equation}
where the linear collision operator $\LL$ is defined as follows:
\begin{equation} \label{eqn:def.LL}
\LL u(x,v) = \int_\VV k(x,v,v')u(x,v') \rd v' - \int_\VV k(x,v',v)\rd{v'}u(x,v)\,.
\end{equation}
In Eq.~\eqref{eqn:rte_general}, the evolution of photon density is
governed by the transport term $v\cdot\nabla_xu^\vep$, that describes
the photons free streaming with speed $v$ in direction $x$, and the
collision term $\LL$ that characterizes the interaction of the
particles with the background media. The first term in $\LL$
represents particles with velocity $v'$ that are scattered off to
obtain $v$, while the second term indicates the particles whose
velocity changes from $v$ to $v'$.
The specific form of $k(x,v,v')$ depends on the media.  When the
scattering is homogeneous in velocity, we can write $k(x,v,v') =
\sigma(x)$ for some $\sigma$, so that \eqref{eqn:def.LL} becomes
\begin{align}\label{eqn:def.LL.hom}
\LL u(x,v) &= \sigma(x)\int_\VV \left(u(x,v') - u(x,v)\right)\rd{v'}\,.
\end{align}


In the radiative transfer equation \eqref{eqn:rte_general}, the quantity
$\varepsilon$, which captures the strength of the collision term, is
called the {\em Knudsen number}. When $\vep$ is small, the collision
term dominates the transport term, and we have $\LL[u^\vep]=0$, to
leading order. In this case, the solution is close to lying in the
null space of $\LL$, that is, the solution profile nearly achieves
local equilibrium for every $x$. Via asymptotic expansion, we have
that $u^\varepsilon(x,v)\to M(v)u^\ast(x)$, where $u^\ast(x)$ solves
the heat equation and $M(v)$ (called the {\em local equilibrium} or
the {\em Maxwellian}) spans $\NullL$. More specifically, we
have the following result~\cite{BSS84, LK74, Papa75} for homogeneous collision~\eqref{eqn:def.LL.hom}. 
\begin{theorem}\label{thm:hom_transport}
	Suppose that $u^\vep$ solves~\eqref{eqn:rte_general} with collision term~\eqref{eqn:def.LL.hom} in $\mathcal{K}$, which is a bounded domain in $\mathbb{R}^3$ with $C^1$ boundary,
	with $\VV = \mathbb{S}^{2}$, and with boundary condition
	\begin{equation}\label{eqn:transportBdy}
	u^\vep(x,v) = \phi(x,v) \quad\text{on}\quad x\in\partial\mathcal{K}\,, \;\; v\cdot n_x<0\,.
	\end{equation}
	Then $M(v)=1$ and
	\begin{equation}\label{eqn:conv_rte}
	\|u^\vep(x,v) - u^\ast(x)\|_{L_2(\rd x\rd v)} \to 0\,,
	\end{equation}
	where $u^\ast(x)$ solves
	\begin{equation}\label{eqn:diff}
	\frac{1}{3}\nabla_x\cdot \left(\frac{1}{\sigma}\nabla_xu^\ast \right) =g(x)\,,\quad x\in\mathcal{K}\,,
	\end{equation}
	with the boundary condition
	\begin{equation*}
	u^\ast(x) = \xi_\phi(x)\,,\quad\text{on}\quad x\in\partial\mathcal{K}\,,
	\end{equation*}
	where $\xi_\phi(x)$ is obtained by solving the boundary layer equation~\cite{Papa75}.
\end{theorem}

This result indicates that the limiting operator as $\vep \to 0$ is
$\mathcal{L}^\ast := (1/3)
\nabla_x\cdot\left(({1}/{\sigma})\nabla_x\right)$, which is
independent of the velocity variable. The constant changes with the
dimension of $v$; $1/3$ is the appropriate value for $\Kcal \subset
\mathbb{R}^3$. 

\begin{remark}
	Here we only present the least complicated case, in which the collision is homogeneous~\eqref{eqn:def.LL.hom}, and we do not specify the convergence rate in~\eqref{eqn:conv_rte}. If the collision operator $\LL$ is not homogeneous in $v$, the Maxwellian $M(v)$ could
	have a complicated form, and the theorem must be modified accordingly. It was long believed that with the correct boundary-layer equation introduced in~\cite{Papa75} to translate the boundary conditions from that of $u^\vep$ to that of $u^\ast$, the convergence rate is first order (that is, $u^\vep-u^\ast=\mathcal{O}(\vep)$). Recently, however, this was shown not to be the case; see \cite{WG2014,LiLuSun_DD_AP,Li2017}, which show that the boundary layer corrector can reduce the convergence order to less than $1$. The sharpest bound is still unknown.
\end{remark}

AP, both as a term and a concept, was coined in~\cite{Jin_AP}, although the
development of AP in the context of the radiative transfer equation
dates back to earlier works~\cite{JL93, LM89}. The fundamental idea is
that a good numerical method, besides being consistent and stable,
should also (for fixed discretization) preserve the asymptotic limit
of the original equation. As shown in \cref{fig:AP}, one
designs a method $\mathcal{F}^h_\vep$ for a system $\mathcal{F}_\vep$,
and asks (1) whether the discrete system, with fixed $h$, converges
when $\vep$ shrinks; and (2) if it does converge, whether the limit as
$h \to 0$ correctly discretizes $\mathcal{F}_\ast$, the limiting
system on the continuous level.  If $\mathcal{F}^h_\vep$ satisfies
both properties, it is said to be asymptotic preserving.
\begin{figure}[h]
	\centering
	\includegraphics[width = 0.6\textwidth,height = 0.2\textheight]{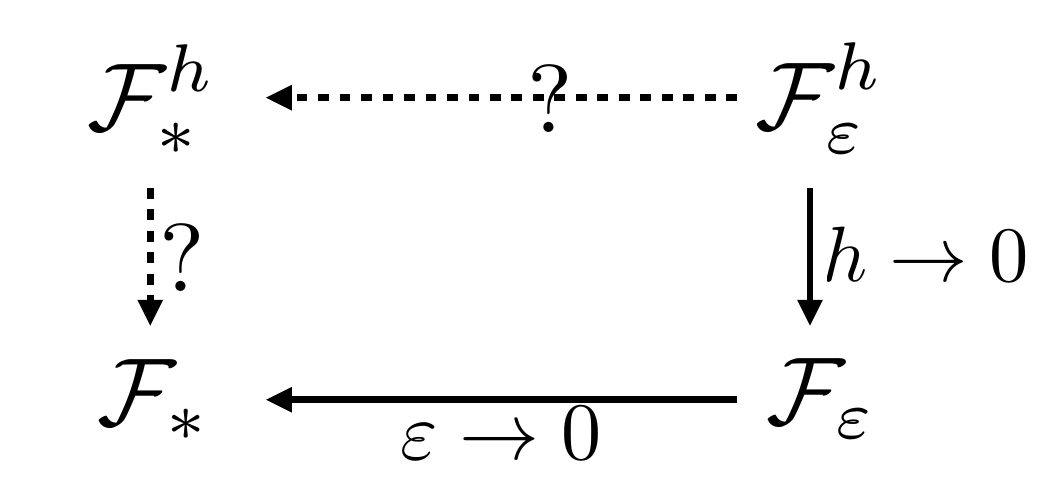}
	\caption{Commuting diagram of asymptotic preserving (AP)
		schemes. An AP solver $\mathcal{F}^h_\vep$ should, in the zero limit of $\vep$ for fixed $h$, capture the solution to $\mathcal{F}_\ast$.}\label{fig:AP}
\end{figure}

This AP property is not easy to satisfy in general. For conventional
schemes, we need $h\ll \vep$ for accuracy, so we cannot in practice
fix $h$ as $\vep \to 0$. AP schemes have to be designed carefully by
using analytic knowledge about the limiting operator
$\mathcal{F}_{\ast}$. Much progress has been made in the past decade.
For linear equations, an even-odd decomposition approach has
been designed, with the even part capturing the limit and the odd part
capturing the second order
expansion~\cite{Klar98,JPT00,LM08,GK:10}. Another approach uses a
preconditioned conjugate gradient that exploits the structure of the
discrete matrix ~\cite{Azmy02,li_wang_2017}.  In the nonlinear
setting, the BGK penalization method was developed
in~\cite{FJ10,BLM_MM} and methods based on the Wild sum~\cite{Wild}
were described in~\cite{DP,expo2} (see also
\cite{Degond-Rev13,HuJinLi16}). Most of these methods are designed for
time-dependent problems. 
Because of  the
limited analytic knowledge about kinetic boundary layers, there are
very few AP solvers for time-independent problems (see~\cite{LiLuSun_half_space,GolseKlar:95,LiLuSun_DD_AP}).

\subsection{Numerical homogenization}
Consider elliptic equations in divergence form with highly oscillatory media:
\begin{equation}\label{eqn:diff_ep}
\begin{cases}
\nabla_x\cdot\left(a\left(\frac{x}{\varepsilon}\right)\nabla_xu^\vep\right) = g & \; \text{in }\Kcal\,, \\
u^\vep = f & \; \text{on } \partial\Kcal\,,
\end{cases}
\end{equation}
where $\Kcal$ is a bounded Lipchitz domain and $0<\vep \ll 1$ characterizes the small scale in the problem. We assume $a(\cdot )$ is bounded below and above by positive constants. We also assume $a(\cdot)$  is H\"older continuous and $1$-periodic, so problem \eqref{eqn:diff_ep} is elliptic with highly-oscillatory media.
As $\vep$ goes to zero, the
solution $u^\vep$ converges to that of a homogenized equation
\begin{equation}\label{eqn:diff_limit}
\begin{cases}
\nabla_x\cdot\left(a^\ast(x)\nabla_xu^\ast\right)=g & \; \text{in }\Kcal\,, \\
u^\ast = f \,, & \; \text{on } \partial\Kcal \,,
\end{cases}
\end{equation}
in the sense that
\[
\|u^\vep - u^\ast\|_2 = \mathcal{O}(\vep)\,.
\]
Here $a^\ast$ is termed the {\em effective media}~\cite{BLP,Allaire_hom,moskow_vogelius_1997}.

\begin{theorem}[Theorem 1.1 in ~\cite{Kenig2012}]\label{thm:hom_diff}
	Denote $u^\vep$ is the solution to Equation \eqref{eqn:diff_ep} and $u^\ast$ the solution to the effective equation \eqref{eqn:diff_limit}. With same assumptions on $\Kcal$ and $a(\cdot)$ as above, then for any $g\in L^2(\Kcal)$ and $f\in H^1(\partial \Kcal)$,
	if $u^\ast \in H^2(\Kcal)$, then we have the strong convergence in $L_2$:
	\begin{equation}\label{eqn:conv_de}
	\| u^\vep - u^\ast \|_{L^2(\Kcal)} \leq C_\sigma \vep | \ln(\vep) |^{\frac{1}{2}+\sigma} \left[ \| g \|_{L^2(\Kcal)} + \| f \|_{H^1(\partial \Kcal )}\right]
	\end{equation}
	for any $\sigma>0$.
\end{theorem}

The aim of numerical homogenization, or numerical treatment for elliptic equations with rough media to a larger extent, which has a long history, is to develop efficient solvers with two key properties:
\begin{enumerate}
	\item[1.] the discretization is independent of $\vep$;
	\item[2.] the numerical solutions capture the correct limiting
	solutions on the discrete level.
\end{enumerate}
Many methods have been developed for elliptic equations,
including the multiscale Finite Element Method
(MsFEM)~\cite{EHW:00,HW97,HWC99}, the heterogeneous multiscale method
(HMM)~\cite{EMingZhang:2005,EEngquist:03,ABDULLE_HM}, the reduced basis type method~\cite{ABDULLE_RB,abdulle_rb_nonlinear}, local orthogonal
decomposition~\cite{MP14}, subspace decomposition methods~\cite{Kornhuber16,Kornhuber18}, local basis
construction methods~\cite{babuska_partition_1997,BL11,OZ14,Owhadi_bayesian}, and the global-local approach~\cite{OdenVemaganti:2000, BabuskaLiptonStuebner:2008, HuangLuMing:18}, to name just a few. Many of them have been extended to treat a large class of other equations as well ~\cite{peterseim2017eliminating,Aarnes06,engquist2005heterogeneous}. The focus for these methods are slightly different. For example, MsFEM intends to capture the fine scale oscillation while HMM mainly targets at finding the solution to the effective equation. The comparison of these methods is tangential to the goal of the current work. Interested readers are referred to review papers and books~\cite{EEngquist:03,efendiev2009multiscale,peterseimnumerical}.

\section{General solution framework based on domain decomposition and random sampling}
\label{sec:framework}

The asymptotic preserving and numerical homogenization schemes
reviewed in the previous section are two efficient schemes for solving
multiscale problems with highly oscillatory solutions.  Although these
schemes tackle different problems in different ways, both schemes
achieve efficiency by exploiting the fact that the solutions are close
to their asymptotic limits, which lie in a low-dimensional subspace.
The design of these schemes relies heavily on a sophisticated
understanding of the equation and its asymptotic limit. For many PDEs,
this level of understanding is not
available~\cite{chen2019low}. Our goal of
this work is to propose a general numerical framework that can be
applied to various multiscale problems, capturing the efficient
representation of the solutions without an explicit reliance on the
analytical understanding.

A first step in developing our framework is to relate the AP and
numerical homogenization schemes to the numerical linear algebra
concept of {\em low rank}. When the matrix operator in a linear
algebra problem has low rank, the solution lies in a subspace of low
dimension; there are efficient numerical schemes, based on random
sampling, that exploit this property. Drawing on these ideas from
linear algebra, we propose a method under the domain decomposition framework, that utilizes random sampling to search representative modes in the solution space.

\subsection{Numerical rank}\label{sec:numerical_rank}

In this section, we tackle the questions of low-rankness of a PDE
operator and low dimensionality of the solution space in a general
setting, and estimate the rank and dimension for several problems of
interest. In this way, we aim to unify the AP and numerical
homogenization schemes, and develop numerical schemes for more general
multiscale problems.

We consider a bounded linear operator $\mathcal{A}$:
\begin{equation}\label{eqn:Acal}
\begin{aligned}
\mathcal{A}: & \quad \mathcal{X} &\rightarrow &\quad \mathcal{Y} \\
& \quad f 				&\mapsto & \quad u
\end{aligned}
\end{equation}
that maps $f\in\mathcal{X}$ to a Hilbert space $\mathcal{Y}$.  In the
PDE setting, $\mathcal{A}$ maps the boundary conditions and/or source
term to the solution of the problem. We define the following
neighborhood of $\mathcal{A}$ that is parametrized by a positive
scalar $\tau$:
\begin{equation*}
S_\tau := \{ \tilde{\mathcal{A}}\in
\mathcal{L}(\mathcal{X},\mathcal{Y}):
\|\mathcal{A}-\tilde{\mathcal{A}}
\|_{\mathcal{X}\rightarrow\mathcal{Y}} \leq \tau \}\,.
\end{equation*}
The set $S_{\tau}$ is the collection of all operators whose operator
norm is within distance $\tau$ of $\mathcal{A}$.
When the context is clear, we suppress the subscript in the operator
norm $\|\cdot\|_{\mathcal{X}\rightarrow\mathcal{Y}} $.

\begin{definition}[Numerical rank]\label{def:num_rank}
	The numerical $\tau$-rank of $\mathcal{A}$ is the rank of the
	lowest-rank operator in $S_\tau$, that is,
	\[
	k_\tau(\mathcal{A}) := \dim \ran \mathcal{A}_\tau;
	\quad \mathcal{A}_\tau := \arg \min\{\dim \ran \tilde{\mathcal{A}}: \tilde{\mathcal{A}}\in S_\tau  \}\,.
	\]
	That is, $\mathcal{A}_\tau$ is the operator within distance $\tau$ of
	$\mathcal{A}$ whose range space has the smallest dimension, and
	$k_\tau(\mathcal{A})$ is this dimension. We set $k_\tau(\mathcal{A})$
	to $\infty$ if all $\tilde{\mathcal{A}} \in S_{\tau}$ have range
	spaces of infinite dimension.
\end{definition}

The definition of numerical rank is closely related to Kolmogorov
$N$-width, which we define here.
\begin{definition}[Kolmogorov $N$-width]\label{def:Kol_N}
	Given the linear operator in \eqref{eqn:Acal}, the Kolmogorov
	$N$-width $d_N(\mathcal{A})$ is the shortest distance to an
	$N$-dimensional space, that is,
	\[
	d_N(\mathcal{A}) := \min_{S:\dim S = N} d(\mathcal{A},S)=\min_{S:\dim S = N} \sup_f \min_{v \in S} \frac{\| \mathcal{A}f - v\|_\mathcal{Y} }{\| f \|_{\mathcal{X}}} \,.
	\]
\end{definition}
~\cref{def:num_rank} and \cref{def:Kol_N} are connected
through the following proposition:
\begin{proposition}
	For the operator $\mathcal{A}$ specified in \eqref{eqn:Acal}, the
	following are true.
	\begin{enumerate}[(a)]
		\item If the numerical $\tau$-rank is $N$, then $d_N(\mathcal{A})\leq
		\tau$.
		\item If $d_{N}(\mathcal{A})\leq \tau < d_{N-1}(\mathcal{A})$, then
		the numerical $\tau$-rank is $N$.
	\end{enumerate}
\end{proposition}
\begin{proof}
	We use $\mathbf{P}_S$ to denote the projection operator onto a finite
	dimensional subspace $S$. Note that the Kolmogorov $N$-width is a
	non-increasing function of $N$.
	
	For (a), let ${\mathcal{A}}_\tau\in S_\tau$ be the operator that
	achieves the numerical $\tau$-rank of $N$, and denote by $S$ the range
	of ${\mathcal{A}}_\tau$. We then have
	\begin{equation*}
	\tau \geq \| \mathcal{A} - {\mathcal{A}}_\tau\| = \sup_f \frac{\| \mathcal{A}f - {\mathcal{A}}_\tau f \|_\mathcal{Y}}{\| f\|_\mathcal{X}} \geq  \sup_f \min_{v\in S}\frac{\| \mathcal{A}f - v \|_\mathcal{Y}}{\| f \|_\mathcal{X}} \geq d_N(\mathcal{A})\,,
	\end{equation*}
	where the last inequality is from ~\cref{def:Kol_N}.
	
	For (b), suppose that $d_N\leq\tau< d_{N-1}(\mathcal{A})$.  First, for
	an arbitrary $(N-1)$-dimensional subspace $S$, we have
	\[
	\tau < d_{N-1}(\mathcal{A}) \leq \sup_f \min_{v\in S}\frac{\| \mathcal{A}f - v \|_\mathcal{Y}}{\| f \|_\mathcal{X}}\leq \sup_f \frac{\| \mathcal{A}f - \mathbf{P}_S\mathcal{A}f \|_\mathcal{Y}}{\| f \|_\mathcal{X}}= \|\mathcal{A} -\mathbf{P}_S\mathcal{A}\|\,,
	\]
	then according to ~\cref{def:num_rank}, there is no
	$(N-1)$-dimensional operator that achieves $\tau$ accuracy, so we must
	have $k_\tau(\mathcal{A})\geq N$. Second, since $d_N(\mathcal{A})\leq
	\tau$, then there exists a $N$-dimensional subspace $S$ such that
	\[
	d_N(\mathcal{A}) = \sup_f \min_{v\in S}\frac{\| \mathcal{A}f - v \|_\mathcal{Y}}{\| f \|_\mathcal{X}} = \sup_f \frac{\| \mathcal{A}f - \mathbf{P}_S\mathcal{A}f \|_\mathcal{Y}}{\| f \|_\mathcal{X}}=\| \mathcal{A} -\mathbf{P}_S\mathcal{A} \| \leq \tau\,.
	\]
	Defining $\mathcal{A}_\tau = \mathbf{P}_S\mathcal{A}$, we see that the
	numerical $\tau$-rank is $N$.
\end{proof}


The numerical rank and the Kolmogorov $N$-width both depend on
\emph{optimal approximations}, which typically require basis set
construction that is adaptive to the given problem.
The pre-defined basis sets conventionally
used in numerical discretization, such as local polynomials and global
Fourier functions (as used in finite difference/element methods and
spectral methods), are not optimal, except in very special cases (heat equation, for example). In
fact, there are counterexamples that show them to be arbitrarily bad;
see \cite{Tadmor_spectral_CL} for the spectral method and
\cite{babuska_bad} for finite elements.

It is important to distinguish between numerical rank and degrees of
freedom (DOF). The DOF is the number of variables needed to represent
the solutions (to a certain specified error tolerance), when the basis
functions are given. Each numerical method utilizes a certain set of
pre-specified basis functions, and the DOF changes according to the
method used. The numerical rank, however, depends on the {\em optimal}
representation, so is the minimum DOF across all possible methods.
We study two examples and give rough computation of DOF using standard
finite element methods, thus yielding upper bound of the respective
numerical ranks. 
Numerical rank, as a concept, was explicitly explored in several papers on numerical homogenization, including \cite{BL11,MP14,calo2016randomized,Grasedyck12}. In \cite{BL11}, it was proved that the singular values of a confinement map decay almost exponentially. This concept, however, was not as developed in other sub-areas of multiscale computation. We  compare numerical rank and DOF explicitly below.

\subsubsection{Numerical rank of the radiative transfer equation}\label{sec:rank_Boltzmann}

To estimate the numerical rank of the solution
  operator $\mathcal{A}$ for the radiative transfer equation
  \eqref{eqn:rte_general}, \eqref{eqn:def.LL} and its diffusion limit,
  we consider the following cases. We assume in this section that the
  boundary condition $\phi$ in \eqref{eqn:transportBdy} satisfies $\phi\in W^{2,\infty}$, so that the solution
  $u$ and $u^\vep$ attain the same regularity~\cite{egger14}. The
  boundary-to-solution map $\mathcal{A}$ thus maps $W^{2,\infty}$ to
  $W^{2,\infty}$. For simplicity, we study the numerical rank of
  $\mathcal{A}$ associated with $L^2$ norm.
\begin{enumerate}[a)]
	\item Let $\vep = 1$ in~\eqref{eqn:rte_general}. If we use the
upwind method for $\partial_x$ and the trapezoidal rule for $\LL$, the
method is first-order in $x$ and second-order in $v$. By equating the
numerical error estimate to the accuracy required, we have
	\begin{equation*} \mathcal{O}(N_x^{-1}+N_v^{-2}) = \tau
\quad\Rightarrow\quad N_x = \mathcal{O}(1/{\tau})\,, \quad N_v =
\mathcal{O}(1/\sqrt{\tau})\,.
	\end{equation*} For $\tau$-accuracy, we thus obtain the
following DOF:
	\begin{equation*} N_{\vep=1} = N_xN_v =
\mathcal{O}(1/\tau^{3/2})\,.
	\end{equation*}
	\item Suppose that $\vep$ is extremely small
in~\eqref{eqn:rte_general} and we use the same method as shown
above. Then, defining $C_\vep = \|\partial^2_x
u^\vep\|_\infty=\mathcal{O}(\frac{1}{\vep^2})$, we have that
	\begin{equation*} \mathcal{O}(C_\vep N_x^{-1}+N_v^{-2}) = \tau
\quad\Rightarrow\quad N_x = \mathcal{O}(C_\vep/{\tau})=
\mathcal{O}\left( \frac{1}{\tau\vep^2} \right)\,,N_v =
\mathcal{O}(1/\sqrt{\tau})\,.
	\end{equation*} Note that $C_\vep$ blows up for small $\vep$,
since $u^\vep$ has sharp transitions. For $\tau$-accuracy, the DOF is
	\begin{equation}\label{eqn:rte_rank_vep} N_{\vep} = N_xN_v =
\mathcal{O}\left(\frac{1}{\vep^2\tau^{3/2}} \right)\,.
	\end{equation}
	\item If hat functions are used to construct the finite
element basis for the limiting Poisson equation~\eqref{eqn:diff}, the
method is second-order convergent in $x$, and we obtain
	\begin{equation*} \mathcal{O}(N_x^{-2}) = \tau
\quad\Rightarrow\quad N_x = \mathcal{O}\left(1/\sqrt{\tau} \right)\,.
	\end{equation*} The DOF in this case is thus:
	\begin{equation*} N_\ast = N_x = \mathcal{O}(1/\sqrt{\tau})\,.
	\end{equation*}
	\item If we make use of the diffusion limit, the triangle
inequality yields
	\begin{equation*} \|u^\vep-U\| \leq \|u^\vep-u^\ast\| +
\|u^\ast-U\|\leq \mathcal{O}(\vep) + \mathcal{O}(N_x^{-2})\,,
	\end{equation*} with $U$ being the numerical solution to
$u^\ast$. By comparing with the tolerance $\tau$ and taking the zero
limit of $\vep$, we obtain for the DOF that
	\begin{equation}\label{eqn:rte_rank_hom} N^\text{ap}_\vep =
\mathcal{O}\left(\frac{1}{\sqrt{|\tau-\vep|}}\right) = \mathcal{O} (
1/\sqrt{\tau})\,\quad\text{as}\quad\vep\to 0\,.
	\end{equation} This is the approximation used by the AP
method, hence our notation $N^{\text{ap}}_{\vep}$.
\end{enumerate}

We see by comparing \eqref{eqn:rte_rank_hom} and
\eqref{eqn:rte_rank_vep} that different schemes produce vastly
different DOF. Numerical rank of $\mathcal{A}$, bounded by the
smallest DOF, is thus controlled by $N^\text{ap}_\vep$.  The
homogenization scheme gives a much sharper bound on numerical rank
than the brute-force finite difference method.

\subsubsection{Numerical rank of elliptic equation with oscillatory
coefficients}\label{sec:rank_elliptic}

A similar analysis to the previous subsection can be conducted for the
diffusion equation \eqref{eqn:diff_ep} with rough
media. Again, we assume $H^{3/2}$ regularity for the
  boundary condition $g$, so that the solution $u$ and $u^\vep$ gain
  $H^2$ regularity. We thus consider the solution operator
  $\mathcal{A}$ to be a mapping from $H^{3/2}$ to $H^2$, and study the
  numerical rank of $\mathcal{A}$ associated with $L^2$ norm.
  
\begin{enumerate}[a)]
\item Let $\vep = 1$ in~\eqref{eqn:diff_ep}. If one uses the
  classical finite element method with piecewise hat functions as basis
  functions for $\nabla_x\cdot(a(x,x/\vep)\nabla_x)$, the method is
  second-order convergent. By equating the numerical error to the
  required accuracy $\tau$, we obtain
  \begin{equation*} \mathcal{O}(N_x^{-2}) = \tau
    \quad\Rightarrow\quad N_x = \mathcal{O}(1/\sqrt{\tau})\,,
  \end{equation*} so that the DOF within $\tau$-accuracy is $
N_{\vep=1} = N_x = \mathcal{O}(1/\tau^{1/2}) $.
\item Suppose that $0<\vep \ll 1$ in~\eqref{eqn:diff_ep}. If
  we use the classical finite element method with hat functions, as
  above, the discretization needs to resolve the oscillations, leading
  to the estimate 
    \begin{equation*} \mathcal{O}\left(\frac{1}{\vep^2}
        N_x^{-2}\right) = \tau \quad\Rightarrow\quad N_x =
      \mathcal{O}\left(\frac{1}{\vep\sqrt{\tau}} \right)\,,
    \end{equation*} where the factor ${1}/{\vep^2}$ arises from
    Theorem 4.4 in~\cite{hou1999convergence}. We thus have
    \begin{equation}\label{eqn:elliptic_rank_vep} N_{\vep} = N_x =
      \mathcal{O}\left(\frac{1}{\vep\sqrt{\tau}}\right)\,.
    \end{equation} 
\item If the finite element method with hat-function basis is
  applied to the limiting effective equation with smooth
  media~\eqref{eqn:diff_limit}, the solution is smooth and the
  derivative is order one. Since the method is second-order, we obtain
  \begin{equation*} \mathcal{O}(N_x^{-2}) = \tau
    \quad\Rightarrow\quad N_x = \mathcal{O}(1/\sqrt{\tau})\,,
  \end{equation*} which leads to a DOF of $N_\ast = N_x =
  \mathcal{O}(1/\sqrt{\tau})$.
\item The homogenization route and the triangle inequality
  leads to
  \begin{equation*} \|u^\vep-U\| \leq \|u^\vep-u^\ast\| +
    \|u^\ast-U\|\leq \mathcal{O}(\vep) + \mathcal{O}(N_x^{-2})\leq \tau\,,
  \end{equation*}
  so that
  \begin{equation}\label{eqn:elliptic_rank_hom}
    N_\vep^\text{hom}=N_x \geq
    \mathcal{O}\left(\frac{1}{\sqrt{|\tau-\vep|}}\right)\to\mathcal{O}(1/\sqrt{\tau})\quad
    \text{as}\quad \vep\to 0\,.
  \end{equation}
\end{enumerate} By comparing~\eqref{eqn:elliptic_rank_vep}
and~\eqref{eqn:elliptic_rank_hom}, we see that the DOF obtained from
homogenization gives a much sharper bound on the numerical
rank. Moreover, the numerical rank is finite, even in the zero limit
of $\vep$.


The discussions above show that the DOF depends on
  both the approximate solution space and the choice of basis
  functions, while numerical rank, by contrast, reflects the size of
  the basis required to approximate the solution up to a certain given
  accuracy. Heuristically, it also implies that the singular values of
  stiffness matrix decay rapidly, while the size of this matrix
  explodes as $\vep \to 0$. When DOF is significantly higher than the
  numerical rank, fast matrix-vector multiplication methods, which may
  exploit the sparsity of the stiffness matrix, may accelerate the
  computation.  However, this topic is beyond the focus of this
  paper. We take the alternative route here of identifying
  lower-dimensional spaces that approximate the solution space well
  and economically, using techniques that are motivated by randomized
  algorithms in numerical linear algebra.

\begin{remark}
  The discussion above has been justified rigorously in \cite{BL11}
  for elliptic equation with rough media. This paper shows the optimal
  local basis functions are indeed the singular vectors of a
  restriction operator $P$, and that the Kolmogorov $N$-width of $P$
  is exponentially decaying, that is,
  \[
    d_N(P) \lessapprox e^{-n^{1/(d+1)}},
  \]
  where $d$ is dimension of physical space. Therefore, the numerical
  $\tau$-rank of $P$ is small and the optimal representation of
  solution of elliptic equation has small DOF. The work~\cite{BL11}
  constructed optimal basis via an eigenvalue problem, whereas our work
  proposes to use a randomized algorithm.
\end{remark}

\subsection{Random sampling in numerical linear
	algebra}\label{sec:rsvd} Random sampling algorithms have a long
history in numerical linear algebra \cite{Hutchinson:90, Goreinov:97,
	Frieze:98, Stewart:99, Mahoney:11, Tropp_rSVD, KannanVempala:17}; we
will focus here on those related to low-rank approximations of a
matrix.  Given a matrix
$\Amat\in\mathbb{R}^{m\times n}$ that is known to be approximately low
rank, a standard way to obtain the most important modes in its range
is via the singular value decomposition (SVD). Without loss of generality, we assume $m\geq n$ and  
write the singular value decomposition as
\begin{equation}
\label{eqn:svd}
\Amat = \Umat\Sigma \Vmat^\top = \sum_{i=1}^n\sigma_iu_iv^\top_i\,,
\end{equation}
where $\Umat=\left[u_1\,,u_2\,,\dotsc,u_n\right]\in
\mathbb{R}^{m\times n}$ contains the left singular vectors,
$\Vmat=\left[v_1\,,v_2\,,\dotsc,v_n \right]\in \mathbb{R}^{n\times n}$
contains the right singular vectors and $\Sigma = \diag
(\sigma_1,\sigma_2,\dotsc,\sigma_n)$ contains the singular values in
descending order: $\sigma_1\geq\sigma_2\geq\dotsc\sigma_n\geq
0$. $\Umat$ and $\Vmat$ are orthogonal matrices. It is well known that
the best $k$-rank approximation to $\Amat$ (in spectral norm) is given
by thresholding the singular value decomposition at $k$-th order,
termed $\Amat_k$ here:
\begin{equation*}
\Amat_k = \Umat_k\Sigma_k V_k=\sum_{i=1}^k\sigma_iu_iv^\top_i\,,
\end{equation*}
where $\Umat_k$ and $\Vmat_k$ contain the first $k$ columns in $\Umat$
and $\Vmat$. We say the matrix is approximately rank-$k$ if
$\|\Amat-\Amat_k\|=\sigma_{k+1}\ll \sigma_1$. In this case, we have
\begin{equation*}
\|\Amat-\Amat_k\| = \|\Amat - \Umat_k\Umat_k^\top\Amat\| = \sigma_{k+1}\ll\sigma_1 = \|\Amat\|\,.
\end{equation*}
In terms of the discussion in the previous subsection, the range space
of $\Amat$ is approximately the same as the range space of $\Amat_k$,
which equals the span of the columns of $\Umat_k$, which is the
subspace we seek. Computation of the SVD \eqref{eqn:svd} is a
classical problem in numerical linear algebra, requiring
$\mathcal{O}(mn^2)$ operations.


Randomized SVD efficiently computes the low-rank approximation of a
given matrix by means of random sampling of its column space.
The particular version of the algorithm we describe here was developed
in~\cite{Liberty:07, Woolfe:08}; see \cite{Tropp_rSVD} for a review.


The idea behind the algorithm is simple: if an $m\times n$ matrix
$\Amat$ is of approximate low rank $k$, the matrix maps an
$n$-dimensional sphere to an $m$-dimensional ellipsoid that is ``skinny:'' $k$ of its
axes are significantly larger than the rest. With high probability,
vectors that are randomly sampled vector on the $n$-dimensional sphere
are mapped by $\Amat$ to vectors that lie mostly in a $k$-dimensional
subspace of $\mathbb{R}^m$, which is the range of $\Amat$. An
approximation to $\Amat_k$ can be obtained by projecting onto this
subspace.

The precise statement of the randomized SVD algorithm and its error
estimates are recalled in the following theorem.
\begin{theorem}[Theorems~10.6 and 10.8 of \cite{Tropp_rSVD}]
	\label{thm:average_spectral}
	Let $\Amat$ be defined as in \eqref{eqn:svd} and let the target rank
	$k$ be at least $2$. Define
	\begin{equation}\label{eqn:define_Y}
	\Ymat = \Amat\Omega\,,
	\end{equation}
	where $\Omega=\left[\omega_1\,,\dotsc,\omega_{k+p}\right]$ is a matrix of size
	$n\times (k+p)$ with its entries randomly drawn from i.i.d.~normal
	distribution, where $p$ is an oversampling parameter. If $\Amat$ is
	approximately $k$-rank, then with large probability,
	$\Pmat_\Ymat(\Amat)$, the projection of $\Amat$ onto the space
	spanned by $\Ymat$, defined by
	\begin{equation*}
	\Pmat_\Ymat(\Amat) = \Ymat(\Ymat\Ymat^\top)^{-1}\Ymat^\top\Amat\,,
	\end{equation*}
	yields the following error bounds.
	\begin{enumerate}[a)]
		\item Average spectral error:
		\begin{equation*}
		\mathbb{E}\,\|\Amat - \Pmat_\Ymat(\Amat)\|\leq \biggl(1+\frac{k}{p-1}\biggr)\sigma_{k+1} + \frac{e\sqrt{k+p}}{p}\biggl(\sum_{j>k}\sigma_j^2\biggr)^{1/2}\ll\sigma_1\,.
		\end{equation*}
		\item Deviation bound:
		\begin{equation*}
		\|\Amat - \Pmat_\Ymat(\Amat)\|\leq\left[\biggl(1+t\sqrt{\frac{3k}{p+1}}\biggr)\sigma_{k+1}+t\frac{e\sqrt{k+p}}{p+1}\biggl( \sum_{j>k}\sigma_j^2 \biggr)^{1/2}\right] + ut\frac{e\sqrt{k+p}}{p+1}\sigma_{k+1}\ll\sigma_1,
		\end{equation*}
		with failure probability at most $2t^{-p} + e^{-u^2/2}$, for all
		$u,t>1$.
	\end{enumerate}
\end{theorem}

We emphasize two advantages of the algorithm: It captures the
approximate range within $n(k+p)$ operations ($p$ is fixed and small), and it does not require full
knowledge of $\Amat$, only the ability to evaluate the matrix-vector
product $\Amat\Omega$.


\subsection{General solution framework for multiscale
	problems} \label{sec:offon}


Finding a low-rank representation of solution space is the key to
reducing complexity.
In this section, we adapt the low-rank approximation scheme from
numerical linear algebra into a general methodology for solving
multiscale PDEs. The method requires limited knowledge on the specific
structure of the solution spaces, so the solvers are expected to be
applicable to a large class of multiscale problems. Our framework uses
domain decomposition to sketch the local solution space via randomized
sampling, in an offline step. This is followed by an online step, in
which the solution is patched together by imposing continuity
conditions across the domains.

We wish to solve the problem \eqref{eqn:general_u_ep}, that is,
\begin{equation}
\begin{cases}
(\mathcal{L}^\vep u^\vep)(x) = 0\,, \quad x\in \Kcal \,, \\
\mathcal{B}u(x) = \phi(x) \,, \quad x\in \Gamma \,,
\end{cases}
\end{equation}
where $\mathcal{B}$ is the boundary condition operator, $\Gamma$ the
boundary associated with domain $\Kcal$ and $f$ the boundary data.  We
adopt the domain decomposition approach, partitioning $\Kcal$ into $M$
non-overlapping subdomains, as follows:
\[
\Kcal = \bigcup_{m=1}^M \Kcal_m\,, \quad\text{with}\quad \Kcal_m^{\circ} \cap \Kcal_n^{\circ} = \emptyset \quad ( m\neq n)\,,
\]
where $\Kcal_m$ denotes the $m$-th local patch. Accordingly, we denote
by $\Gamma_m$ the boundary associated with $\Kcal_{m}$. Different types
of equations require various kinds of boundary conditions, as we will
make explicit in~\cref{sec:rte} and \cref{sec:elliptic}.
Each subdomain is further discretized with a conformal mesh. We denote $h$ as the largest meshsize and assume that it is fine enough such that $h\ll \vep$. The number of subdomains $M$ does not depend on $\vep$.


Domain decomposition approach consists of two stages, as follows.
\begin{enumerate}[(1)]
	\item \textbf{Offline stage}: Prepare local solution space. Denote by
	$\Gmat_{m}$ the collection of local solutions in each local patch
	$\Kcal_m$, $m=1,2,\dotsc,M$, that is,
	\begin{equation*}
	\Gmat_m = \left[b_{m,1}\,,b_{m,2}\dotsc \right]\,,
	\end{equation*}
	where each local function $b_{m,n}$ is one solution to the equation on
	the subdomain $\Kcal_m$, that is,
	\begin{equation*}
	\mathcal{L}^\vep b_{m,n} = 0\,, \quad x\in \Kcal_m\,,
	\end{equation*}
	with boundary condition on $\Gamma_m$. These solutions are computed on
	fine grids with discretization $h$.
	\item \textbf{Online stage}: The global solution is written as
	\[
	u = \sum_{m=1}^M u_m = \sum_{m=1}^M \Gmat_{m}{c}_m,
	\]
	with $u_m$ being $u$ confined on $\Kcal_m$. $c_m$ is a vector of coefficients determined by the boundary conditions $\phi$ and conditions
	that enforce continuity across patches.
\end{enumerate}
The online stage is a standard step in domain decomposition. Its cost
is governed by the number of basis functions chosen in the offline
step. In the offline stage, there are many ways to construct the local
solution space $\Gmat_m$. Since this space contains all possible
local solutions, it can be regarded as a full library of all Green's
functions. One possible way to define $\Gmat_m$ is to define the
boundary conditions on the $m$th patch to be delta functions defined
over a grid on the boundary $\Gamma_m$, that is,
\begin{equation*}
\begin{cases}
\mathcal{L}^\vep b_{m,n} = 0\,, \quad x\in \Kcal_m \\
b_{m,n} = \delta_{m,n}\,,\quad  x\in \Gamma_m\,,
\end{cases}
\end{equation*}
where $\delta_{m,n}$ is the Kronecker delta function that takes the
value $1$ at the $n$-th grid point on $\Gamma_m$ and zero on the other
grid points on $\partial\Kcal_m$. Since $h\ll \vep$, the number of
functions $n_m$ in $\Gmat_m$ grows as $\vep$ shrinks. This strategy,
summarized in Algorithm~\textproc{DetLocalSolu}, is referred to as the
{\em full-basis approach}.

An alternative way to construct basis functions for each patch also
makes use of a grid defined on the boundary $\Gamma_m$, but
takes the boundary conditions for each function $b_{m,n}$ to be a set
of random values on the grid points, rather than a $\delta$ function.
Specifically, we have
\begin{equation*}
\begin{cases}
\mathcal{L}^\vep r_{m,n} = 0\,, \quad x\in \Kcal_m, \\ r_{m,n} =
\omega_{m,n}\,,\quad x\in \Gamma_m,
\end{cases}
\end{equation*}
where $\omega_{m,n}$ is defined to have a random value drawn
i.i.d. from a normal distribution at each grid point in
$\Gamma_m$.  Since the local solution space is homogenizable
and low rank, we expect that the number of basis functions $k_m$
required to represent it adequately will be much smaller than $n_m$
defined above, and independent of $\vep$. This strategy, which we
refer to as the {\em randomized reduced-basis approach}, is summarized in
Algorithm~\textproc{RandLocalSolu}. 
In practice, one could add a QR-decomposition at the end of algorithm~\textproc{RandLocalSolu} to return basis functions that are orthonormal. This would improve the condition number of the global online problems (for example, \eqref{eqn:rte_cond_reduced} and \eqref{eqn:elliptic_cond_reduced}).

Denote by $\Gmat_m^b$ the collection of full basis $\{b_{m,n}\}$ and
$\Gmat_m^r$ the collection of random reduced basis $\{r_{m,n}\}$, we have
the following relationship:
\begin{equation*}
\Gmat_m^r = \Gmat_m^b\Omega\,,
\end{equation*}
where $\Omega$ is a random i.i.d. matrix with entries $\omega_{m,n}$.

The complete scheme, which includes the two alternative
implementations of the offline stage described above, is specified as~\cref{alg:summary}.

\begin{algorithm}
	\caption{Multiscale solver for $\mathcal{L}^\vep
		u^\vep = 0$ over $\Kcal$ with $\mathcal{B}u = f$ on $\Gamma$}\label{alg:summary}
	\begin{algorithmic}[1]
		\State \textbf{Domain Decomposition}
		\Indent
		\State Partition domain into non-overlapping patches $\Kcal = \bigcup_{m=1}^M \Kcal_m$. 
		\State Form the ansatz $u = \sum_{m=1}^M u_m =\sum_{m=1}^{M}\Gmat_{m}\vec{c}_m$.
		\EndIndent
		\State \textbf{Offline Stage:}
		\Indent
		\State Call function $\Gmat_{m}$=\textproc{DetLocalSolu}($\Kcal_m$) or $\Gmat_{m}$=\textproc{RanLocalSolu}($\Kcal_m$).
		\EndIndent
		\State \textbf{Online Stage:}
		\Indent
		\State Use continuity condition and global boundary data $f$ to determine $\left[\vec{c}_1,\ldots,\vec{c}_M\right]$.
		\EndIndent
		\State \textbf{Return:} approximated global solution $\hat{u} = \sum_{i=1}^M \Gmat_m \vec{c}_m$.
	\end{algorithmic}
	\hrulefill
	\begin{algorithmic}[1]
		\Function{DetLocalSolu}{$\Kcal_m$}
		\State	Prepare full list of numerical delta functions $\delta_{m,i},i=1,\ldots,n_m$ on $\Gamma_m$.
		\State	Call function $u_{m,i}$=\textproc{LocalPDESolver}($\Kcal_m$,$\delta_{m,i}$) for $i=1,2,\dotsc,n_m$.
		\State  \textbf{Return:} Local solution space span $\Gmat_{m}=\left[u_{m,1},\ldots,u_{m,n_m} \right]$.
		\EndFunction
	\end{algorithmic}
	\hrulefill
	\begin{algorithmic}[1]
		\Function{RanLocalSolu}{$\Kcal_m$}
		\State	Prepare $k_m$ random i.i.d.~Gaussian vector $\omega_{m,i},i=1,\ldots,k_m$ on $\Gamma_m$.
		\State	Call function $u_{m,i}$=\textproc{LocalPDESolver}($\Kcal_m$,$\omega_{m,i}$) for $i=1,2,\dotsc,k_m$.
		\State	\textbf{Return:} Approximated local solution space span $\Gmat_{m}=\left[u_{m,1},\ldots,u_{m,k_m} \right]$.
		\EndFunction
	\end{algorithmic}
	\hrulefill
	\begin{algorithmic}[1]
		\Function{LocalPDESolver}{Local domain $\Kcal_m$, Boundary
			condition $\phi$} 
		\State Use standard Finite Element/Difference Methods to
		solve PDE $\mathcal{L}^\vep u^\vep_m = 0$ over $\Kcal_m$ with
		$u^\vep_m = \phi$ over $\Gamma_m$, for solution
		$u^\vep_m$. 
		\State \textbf{Return:} Local solution $u^\vep_m$.  \EndFunction
	\end{algorithmic}
\end{algorithm}

In practice, for \textproc{RandLocalSolu}, we often use a slightly
larger patch $\wt{\Kcal}_m\supset\supset \Kcal_m$ 
that augments $\Kcal_m$ by a buffer zone. The local solution is obtained
on $\wt{\Kcal}_m$, with random boundary conditions on its associated
boundary $\wt{\Gamma}_m$, and then restricted on $\Kcal_m$, as
follows:
\begin{equation*}
\begin{cases}
\mathcal{L}^\vep \tilde b_{m,n} = 0\,, \quad x\in \wt{\Kcal}_m \\
\tilde b_{m,n} = \omega_{m,n}\,,\quad  x\in \wt{\Gamma}_m\,.
\end{cases}
\end{equation*}
Use of the buffer zone helps to remove boundary layer effects and the
effect of the singularity at the boundary. This technique will be
discussed further for the particular PDEs considered in the next two
sections.

\begin{remark}
We emphasize that such connection between PDE and linear algebra has been observed by several previous works, including~\cite{Smetana_random,Owhadi_bayesian,OZ14}. Our proposed method especially coincides with that of~\cite{OZ14}, in which the author explicitly connects the random sampling in $H^{-1}$ (seen in the source) to the representative basis functions in $H^1$ (seen in the solution space). In our case the random sampling is done on the boundary condition, but the method shares the same spirit as reported in~\cite{OZ14}.
\end{remark}

\section{Example 1: Radiative transfer equation}
\label{sec:rte}

We now describe the application of our framework to the radiative
transfer equation with zero source, which is
\begin{equation}\label{eqn:RTE}
\mathcal{L}^\vep u^\vep=v\partial_x u^\varepsilon (x,v) -\frac{1}{\varepsilon} \LL [u^\varepsilon] = 0\,,\quad (x,v) \in \Kcal= \Omega\times\mathbb{V}=[0,1]\times [-1,1]\,,
\end{equation}
where
the collision term $\LL$ is given by
\begin{align*}
\LL u(x,v) &= \int_{-1}^1 k(x,v,v')u(x,v') \rd v' - \int_{-1}^1 k(x,v',v)\rd{v'}u(x,v)\,.
\end{align*}
We use the Henyey-Greenstein model, in which the scattering
coefficient is defined by
\begin{equation} \label{eqn:HG}
k(x,v,v') = \frac{1}{2} \frac{1-g^2}{1+g^2 + 2g (vv')}\,,
\end{equation}
where $g\in (-1,1)$ is a specified constant. To impose boundary
conditions properly for radiative transfer equations, we denote by
$\Gamma_\pm$ the outgoing / incoming part of the boundary:
\begin{equation*}
\Gamma_\pm = \{(x,v): x\in\partial\Omega, \; \pm v\cdot n_x >0\}\,,
\end{equation*}
where $n_x$ is the exterior normal direction at $x\in\partial\Omega$.
In particular, for the problem \eqref{eqn:RTE} on the spatial domain
$\Omega = [0, 1]$,  we have
\begin{equation*}
\Gamma_- = \{(x=0,v>0)\}\cup\{(x=1,v<0)\}\,,\quad \Gamma_+ = \{(x=0,v<0)\}\cup\{(x=1,v>0)\}\,.
\end{equation*}
The equation \eqref{eqn:RTE} is well-posed if a Dirichlet boundary
condition is imposed on the incoming boundary, also known as the
incoming boundary condition: $u^\varepsilon|_{\Gamma_-} = \phi$.

To implement domain decomposition, we partition the domain as follows:
\begin{equation}\label{eq:domainK}
\Kcal = [0,1]\times [-1,1] = \bigcup_{m=1}^M \Kcal_m\,,\quad\text{with}\quad\Kcal_m = [x_{m-1}\,,x_{m}]\times [-1,1]\,,
\end{equation}
where $x_m = {m}/{M}$ forms a set of $(M+1)$ equi-spaced grid
points on $[0,1]$ and $\Kcal_m$ is the $m$-th patch of the domain.
The incoming / outgoing parts of the boundary of each patch are 
\begin{equation*}
\Gamma_{m,-} =
\{(x_{m-1},v>0)\}\cup\{(x_m,v<0)\}\,,\quad\text{and}\quad \Gamma_{m,+}
= \{(x_{m-1},v<0)\}\cup\{(x_m,v>0)\}\,.
\end{equation*}
We denote by $L_{m,m+1} = \Kcal_m \cap \Kcal_{m+1} = \{(x_{m},v): v
\in [-1, 1]\}$ the line segment that separates $\Kcal_m$ and
$\Kcal_{m+1}$. The geometry of the domain and the patches is plotted
in~\cref{fig:dd_stencils}.

\begin{figure}[hbt]
	\centering
	\includegraphics[scale = 0.7]{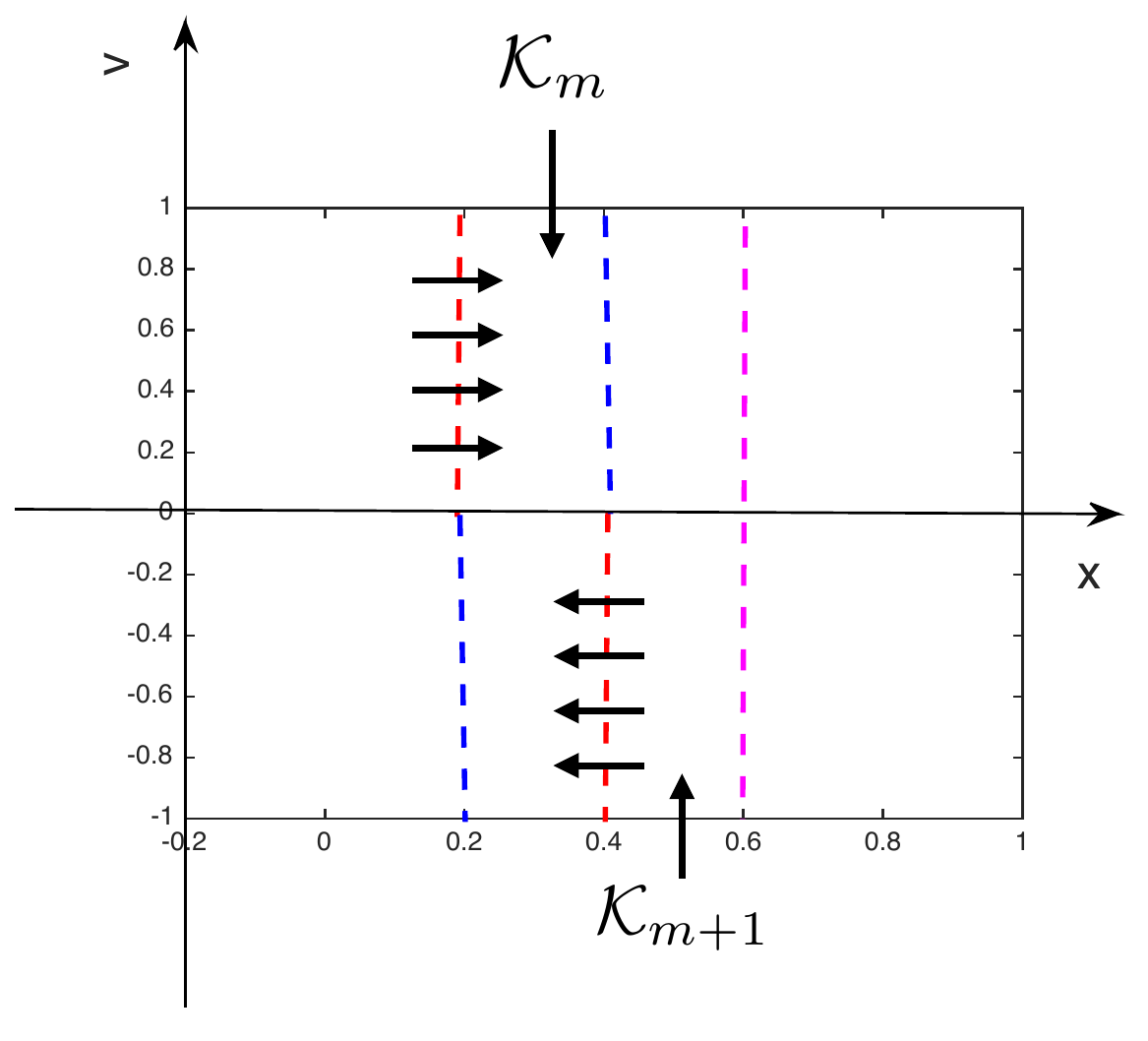}
	\includegraphics[scale = 0.25]{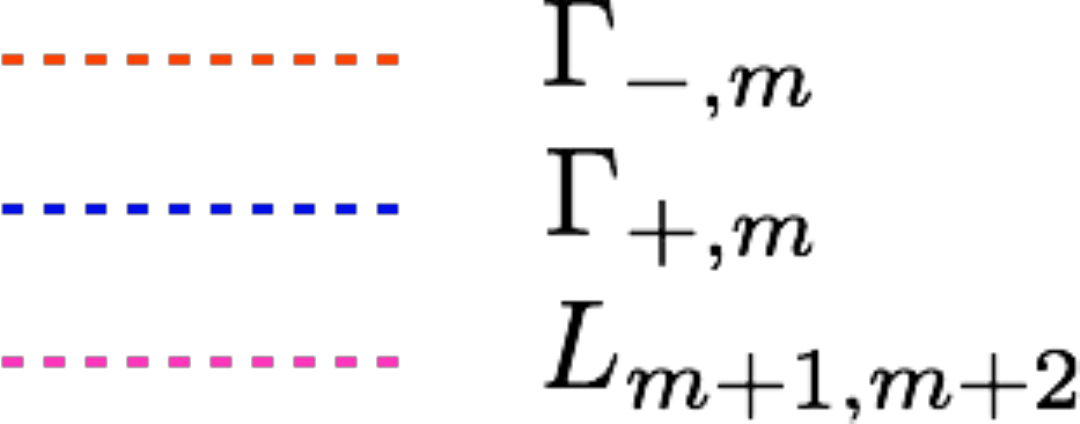}
	\caption{Domain decomposition for RTE and the boundaries of the local
		patch.}\label{fig:dd_stencils}
\end{figure}

As described in~\cref{sec:framework}, the domain decomposition
approach prepares the local solution space in the offline step and
patches together solutions via continuity and boundary conditions in
the online step. We describe the two options for constructing the
basis functions --- the full-basis approach and the randomized
reduced-basis approach --- in the following two subsections.


%
%

\subsection{Full basis approach}

\subsubsection*{Offline step}
We prepare a full basis of the local solution space by enumerating all
possible boundary conditions, up to a discretization. Since the
problem \eqref{eqn:RTE} is linear, we can obtain each basis function
by solving a problem over a patch with a Dirichlet boundary condition
that is nonzero at only one grid point. Specifically, for the patch
$\Kcal_m$, each basis function $b_{m,i}$ is obtained by solving
\begin{equation}\label{eqn:rte_full}
\begin{cases}
v\partial_x b_{m,i} -\frac{1}{\varepsilon} \LL [b_{m,i}] = 0\,,\quad (x,v) \in \Kcal_m\,,\\
b_{m,i}|_{\Gamma_{m,-}} = \delta_{m,i},
\end{cases}
\end{equation}
where $\delta_{m,i}$ is a numerical delta function supported on a grid
point on $\Gamma_{m,-}$ and the index $i$ enumerates all grid points
on the incoming boundary. The full basis for the local solution space
is then given by
\begin{equation}\label{eq:GmatTrans}
\Gmat^\vep_m = \left[b_{m,1}\,,\dotsc,b_{m,n_m} \right]\,,
\end{equation}
where $\Gmat^\vep_m$ is a Green's matrix whose columns are the basis
functions $b_{m, i}$.  Here, $n_m$ is the total number of grid points
on the incoming boundary $\Gamma_{m,-}$ of $\Kcal_m$. 
In other words, the Green's matrix $\Gmat^\vep_m$ is the analog of the operator $\mathcal{A}^\vep_m: f \mapsto b$ defined by
\[
\begin{cases}
v\partial_x b -\frac{1}{\varepsilon} \LL [b] = 0\,, & \; (x,v) \in \Kcal_m\,,\\
b|_{\Gamma_{m,-}} = f. &
\end{cases}
\]
\subsubsection*{Online step}

The online step obtains the global solution as a linear combination of
all local basis functions,  as follows:
\begin{equation}\label{eqn:dd_soln}
u^\vep = \sum_mu^\vep_m = \sum_m\sum_{i}c_{m,i}b_{m,i}\,,
\end{equation}
where the coefficients $c_{m,i}$ are chosen to satisfy the following
conditions:
\begin{itemize}
	\item[$\ast$]{Continuity:} $u_m({L_{m,m+1}}) =
	u_{m+1}(L_{m,m+1})$, which can be stated in more detail as
	\begin{equation}\label{eqn:rte_cont}
	\begin{cases}
	u_m(\Gamma_{m,+}\cap L_{m,m+1}) = u_{m+1}(\Gamma_{m,+}\cap L_{m,m+1})= u_{m+1}(\Gamma_{m+1,-}\cap L_{m,m+1})\,,\\
	u_m(\Gamma_{m,-}\cap L_{m,m+1}) = u_{m+1}(\Gamma_{m,-}\cap L_{m,m+1})=u_{m+1}(\Gamma_{m+1,+}\cap L_{m,m+1})\,.
	\end{cases}
	\end{equation}
	In both equations, the first equality comes from the continuity
	condition and the second equality follows from
	\[
	\Gamma_{m,\pm}\cap L_{m,m+1} = \Gamma_{m+1,\mp}\cap L_{m,m+1},
	\]
	as illustrated in~\cref{fig:dd_stencils}.
	\item[$\ast$]{Boundary condition:}
	\begin{equation}\label{eqn:rte_bdry}
	u|_{\Gamma_-} = \phi\,.
	\end{equation}
\end{itemize}

Algebraically, we denote by $\Mmat_m$ the matrix that maps inflow
boundary condition $c_m = u_m(\Gamma_{m, -})$ to outflow data
$u_m(\Gamma_{m,+})$, and denote by $\Imat_m^{l}$ (\textit{resp.}
$\Imat_m^r$) the restriction operator on the left edge $L_{m-1,m}$
(\textit{resp.} the right edge $L_{m,m+1}$) of patch $\Kcal_m$. Using
this notation, \eqref{eqn:rte_cont} and~\eqref{eqn:rte_bdry} can
be written as follows:
\begin{equation*}
\begin{bmatrix}
\Imat_m^r \Mmat_m & -\Imat_{m+1}^l \\
-\Imat_m^r &  \Imat_{m+1}^l \Mmat_{m+1}
\end{bmatrix}
\begin{bmatrix}
c_m \\
c_{m+1}
\end{bmatrix}
=
\begin{bmatrix}
0 \\
0
\end{bmatrix},
\quad  
\begin{bmatrix}
\Imat_1^l & 0 \\
0	& \Imat_M^r 
\end{bmatrix}
\begin{bmatrix}
c_1 \\
c_M
\end{bmatrix}
=
\begin{bmatrix}
\Imat_1^l \phi\\
\Imat_M^r \phi
\end{bmatrix} \,.
\end{equation*}
Assembling these conditions over all patches, we obtain
\begin{equation}\label{eqn:rte_cond}
\Pmat c = d\,,
\end{equation}
where
\begin{equation*}
\Pmat = \begin{bmatrix}
\Imat_1^l & 0 & 0 & \ldots & 0 \\
\Imat_1^r\Mmat_1  & -\Imat_2^l & 0 & \ldots & 0 \\
-\Imat_1^r & \Imat_2^l\Mmat_2 &  0 & \ldots & 0 \\
& 	\ddots	& \ddots &	\ddots & 		\\
0 & \ldots & 0 &\Imat_{M-1}^r \Mmat_{M-1}& -\Imat_M^l\\
0 & \ldots & 0 & -\Imat_{M-1}^r & \Imat_M^l\Mmat_{M} \\
0 & \ldots & 0 & 0		&\Imat_M^r
\end{bmatrix}\,, \quad
c =
\begin{bmatrix}
c_1 \\
c_2 \\
\vdots \\
c_M
\end{bmatrix}\,, \quad d = 
\begin{bmatrix}
\Imat_1^l \phi \\
0	\\
\vdots \\
0 \\
\Imat_M^r \phi 
\end{bmatrix} \,.
\end{equation*}
We obtain the solution by substituting the coefficients $\{c_{m,i}: i
= 1,\ldots, n_m\,, m = 1,\ldots,M\}$ from \eqref{eqn:rte_cond} into
\eqref{eqn:dd_soln}.

\subsection{Reduced basis approach}\label{sec:buffer_transport}


An approximation to the local solution space for a patch $\Kcal_m$
starts by defining the larger ``buffered'' patch $\wt{\Kcal}_m
\supset\supset \Kcal_m$.  The buffered patch has boundaries
$\wt{\Gamma}_{\pm,m}$, as illustrated in~\cref{fig:stencil_transport}.  We denote by $\wt{\Gmat}^{\vep}_m$ the Green's
matrix obtained by solving the local equation on the buffered patch
$\wt{\Kcal}_m$ with all possible boundary conditions, as in the
construction of \eqref{eq:GmatTrans}, but restricted to the domain
$\Kcal_m$.  More precisely, we can obtain $\wt{b}_{m, i}$ by solving
\begin{equation*}
\begin{cases}
v\partial_x \wt{b}_{m,i} -\frac{1}{\varepsilon} \LL [\wt{b}_{m,i}] = 0\,,\quad (x,v) \in \wt{\Kcal}_m\,,\\
\wt{b}_{m,i}|_{\wt{\Gamma}_{m,-}} = \delta_{i},
\end{cases}
\end{equation*}
where $\wt{\Gamma}_{m, -}$ is the incoming portion of the boundary of
$\partial \wt{\Kcal}_m$, and then define
\begin{equation*}
\wt{\Gmat}_m^{\vep} = \left[\wt{b}_{m,1} \vert_{\Kcal_m}\,,\dotsc,\wt{b}_{m,\wt{n}_m}\vert_{\Kcal_m} \right]\,,
\end{equation*}
where $\wt{n}_m$ is the number of incoming boundary grid points.  It
is clear that each column of $\wt{\Gmat}_m^{\vep}$ solves \eqref{eqn:rte_full} inside $\Kcal_m$, and thus is in $\spanop
\Gmat_m^{\vep}$ (since the latter consists of all possible local
solutions).
Moreover, the
solution to the global equation restricted to $\Kcal_m$ also lies in
$\spanop \wt{\Gmat}_m^{\vep}$.

Due to the diffusion limit, as discussed in~\cref{sec:ap}, the
Green's matrix $\wt{\Gmat}_m^{\vep}$ is approximately low-rank and can
be compressed through random sampling.\footnote{We do not directly
	approximate $\Gmat_m^{\vep}$, which is not low-rank due to the
	singularity near $\partial \Kcal_m$ caused by the incoming Dirichlet
	boundary condition at $\Gamma_{m, -}$. For $\wt{\Gmat}_m^{\vep}$,
	because of the presence of the buffer, this singularity does not
	appear in $\Kcal_m$, causing $\wt{\Gmat}_m^{\vep}$ to be
	approximately low-rank. The use of a buffer is similar to the
	oversampling approach in the multiscale finite element
	method~\cite{HW97}.}  As in~\cref{sec:offon}, we solve the
following system with randomized boundary conditions to obtain each
basis function $\wt{r}_{m,i}$:
\begin{equation}\label{eqn:rte_random}
\begin{cases}
v\partial_x \wt{r}_{m,i} -\frac{1}{\varepsilon} \LL [\wt{r}_{m,i}] = 0\,,\quad (x,v) \in \wt{\Kcal}_m\,,\\
\wt{r}_{m,i}|_{\wt{\Gamma}_{m,-}} = \omega_{m,i}\,,
\end{cases}
\end{equation}
where $\omega_{m, i}$ takes i.i.d.~standard Gaussian at all grid
points on the boundary $\wt{\Gamma}_{m,-}$ and $i$ is the index of
random samples corresponds to different realizations of the boundary
data.  We then take restrictions $r_{m,i} = \wt{r}_{m,i}|_{\Kcal_m}$ and assemble them into local reduced Green's matrix:
\begin{equation*}
\Gmat^{\vep,r}_m = \left[ r_{m,1}\,,\dotsc,r_{m,k_m} \right] =
\wt{\Gmat}_{m}^{\vep} \left[ \omega_{m,1}\,,\dotsc,\omega_{m,k_m} \right] \,.
\end{equation*}
According to Theorem~\ref{thm:average_spectral}, we
  have with high probability that
\[
\frac{\| \wt{\Gmat}^\vep_m - \Pmat_{\wt{\Gmat}^{\vep,r}_m} (\wt{\Gmat}^\vep_m) \| }{\| \wt{\Gmat}^\vep_m \|} \ll 1.
\]
Because of the approximate low-rank property, we can take $k_m\ll
n_m$, thus reducing significantly the dimension of the local solution
space (and also the dimension of the global linear system in the
online step). 
For $m = 1$ and $m=M$ (for which the patch $\Kcal_m$
is at the boundary of full domain), we use the full basis matrix
$\Gmat_m^{\vep, r} = \Gmat_m^{\vep}$, so that we can capture the
boundary conditions that are imposed on the full domain.


In the online step, we write the solution as
\begin{equation} \label{eq:gb2}
u^\vep = \sum_mu^\vep_m \approx \sum_m\sum_{i}\tilde{c}_{m,i}r_{m,i}\,,
\end{equation}
with $\{\tilde{c}_{m,i},m=1,2,\dotsc,M, \; i = 1,2,\dotsc,k_m\}$ being
the coefficients for the reduced basis. We denote by $\wt{\Mmat}_m$
and $\wt{\Wmat}_m$ the matrix that maps $\tilde{c}_m$ to outflow data
$\sum_i \tilde{c}_{m,i} r_{m,i}(\Gamma_{m, +})$ and inflow data
$\sum_i \wt{c}_{m,i}r_{m,i}(\Gamma_{m,-})$ respectively. Note that the
analogous $\Wmat$ would become identity in the full basis approach. By
imposing the continuity condition and exterior boundary condition, we
obtain
\begin{equation*}
\begin{bmatrix}
\Imat_m^r \wt{\Mmat}_m & -\Imat_{m+1}^l \wt{\Wmat}_{m+1} \\
-\Imat_m^r\wt{\Wmat}_m &  \Imat_{m+1}^l \wt{\Mmat}_{m+1}
\end{bmatrix}
\begin{bmatrix}
\wt{c}_m \\
\wt{c}_{m+1}
\end{bmatrix}
=
\begin{bmatrix}
0 \\
0
\end{bmatrix},
\quad  
\begin{bmatrix}
\Imat_1^l\wt{\Wmat}_1 & 0 \\
0	& \Imat_M^r \wt{\Wmat}_M
\end{bmatrix}
\begin{bmatrix}
\wt{c}_1 \\
\wt{c}_M
\end{bmatrix}
=
\begin{bmatrix}
\Imat_1^l \phi\\
\Imat_M^r \phi
\end{bmatrix} \,.
\end{equation*}
Assembling these equations, we obtain

\begin{equation} \label{eq:gb3}
\wt{\Pmat} \tilde{c} = d \,,
\end{equation}
where
\begin{equation*}
\wt{\Pmat} = \begin{bmatrix}
\Imat_1^l\wt{\Wmat}_1 & 0 & 0 & \ldots & 0 \\
\Imat_1^r\wt{\Mmat}_1  & -\Imat_2^l\wt{\Wmat}_2 & 0 & \ldots & 0 \\
-\Imat_1^r\wt{\Wmat}_1 & \Imat_2^l\wt{\Mmat}_2 &  0 & \ldots & 0 \\
& 	\ddots	& \ddots &	\ddots & 		\\
0 & \ldots & 0 &\Imat_{M-1}^r \wt{\Mmat}_{M-1}& -\Imat_M^l\wt{\Wmat}_M\\
0 & \ldots & 0 & -\Imat_{M-1}^r\wt{\Wmat}_{M-1} & \Imat_M^l\wt{\Mmat}_{M} \\
0 & \ldots & 0 & 0		&\Imat_M^r \wt{\Wmat}_M
\end{bmatrix}\,, \quad
\wt{c} =
\begin{bmatrix}
\wt{c}_1 \\
\wt{c}_2 \\
\vdots \\
\wt{c}_M
\end{bmatrix}\,, \quad d = 
\begin{bmatrix}
\Imat_1^l \phi \\
0	\\
\vdots \\
0 \\
\Imat_M^r \phi 
\end{bmatrix} \,.
\end{equation*}
Since we are working in an approximate local
solution space due to the random sampling, this global linear system
constraint is overdetermined and cannot be solved exactly in general.  Instead,
we use the least-squares solution defined by
\begin{equation}\label{eqn:rte_cond_reduced}
\tilde{c} = \arg\min_{e}\|\wt{\Pmat} e-d\|_2\quad \Rightarrow\quad \tilde{c} = (\wt{\Pmat}^\top\wt{\Pmat})^{-1}\wt{\Pmat}^\top d\,.
\end{equation}

\begin{remark}
	The matrix $\wt{\Pmat}$ is of size $M_p\times N_p$ where $M_p = \sum_{m=1}^M n_m$ and $N_p= \sum_{m=1}^M k_m$.  The typical time complexity for this linear regression problem is of order $\mathcal{O}\left( N_p^2(M_P+N_p) \right)$ whereas for the full basis approach \eqref{eqn:rte_cond}, the  matrix $\Pmat$ is of size $M_p$ by $M_p$ and time complexity is $\mathcal{O}(M_p^3)$. Because   $N_p\ll M_p$, the computation cost of our approach is considerably lower.
\end{remark}

\subsection{Numerical test}
We set $g=1/2$ in \eqref{eqn:HG}, and decompose the domain as in
\eqref{eq:domainK} with $M = 10$.  In the velocity domain, we use
the grid points $v_j = -1+\frac{j}{N_v}$ with $N_v=120$ so that the
mesh size in the velocity domain is $\Delta v = \frac{1}{60}$. We
define the buffered patches $\wt{\Kcal}_m$ to be twice as large as the
original patches $\Kcal_m$, with equal margins on each side.  When
solving the local problems, we use spatial discretization with fine
mesh size $\Delta x = 0.01$. The setup is shown in~\cref{fig:stencil_transport}.

\begin{figure}
	\centering
	\includegraphics[width=0.6\textwidth]{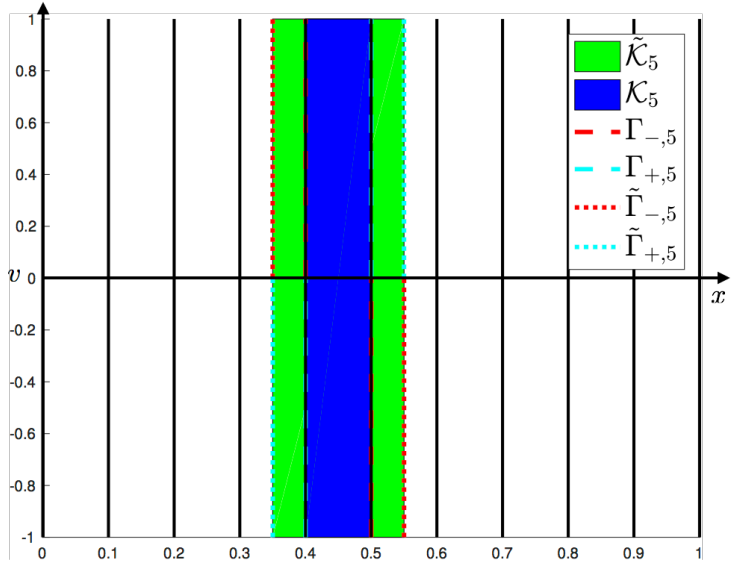}
	\caption{Buffered domain decomposition}
	\label{fig:stencil_transport}
\end{figure}

\subsubsection{Local test}

In~\cref{fig:rte_local}, we show the \textit{normalized}
singular values (that is the ratio $\sigma_j/\sigma_1$ for
$j=1,2,\dotsc$) of Green's matrix $\Gmat^\vep_{2}$ and
$\wt{\Gmat}^\vep_2$ for the second local patch and the buffered patch,
with Knudsen number $\vep=2^{-6}$. 
Note that singular
values enjoy fast decay when $\vep$ is small and that the use of a
buffer induces faster decay.  In~\cref{fig:rte_local_random}, we
plot a measure of relative error for different values of $k_m$ and
$\vep$. The quantity plotted is defined by
\begin{equation*}
\text{error} = \frac{\|\wt{\Gmat}_2^\vep-\Qmat\Qmat^\top\wt{\Gmat}_2^{\vep}\|_2}{\|\wt{\Gmat}_2^\vep\|_2}\,,\quad\text{with}\quad \Gmat_2^{\vep,r} = \Qmat\mathsf{R}\,,
\end{equation*}
that is, $\Qmat$ is obtained from a $QR$ decomposition of
$\Gmat^{\vep,r}_2$, for which the number of columns increases as $k_m$
increases. As $k_m$ increases, the range of $\Gmat_2^{\vep,r}$
captures the range of $\wt{\Gmat}_2^\vep$ more and more accurately,
and that the approximation is satisfactory only for small values of
$\vep$.

In~\cref{fig:rte_local_mode}, we construct random local solution
space span$\{\Gmat_{2}^{\vep,r}\}$ with $k_2= 50$ and show how well
this random solution space can capture the first 3 left singular modes
of $\Gmat_{2}^\vep$ with $\vep = 2^{-6}$.



\begin{figure}
	\centering
	\includegraphics[width=0.6\textwidth]{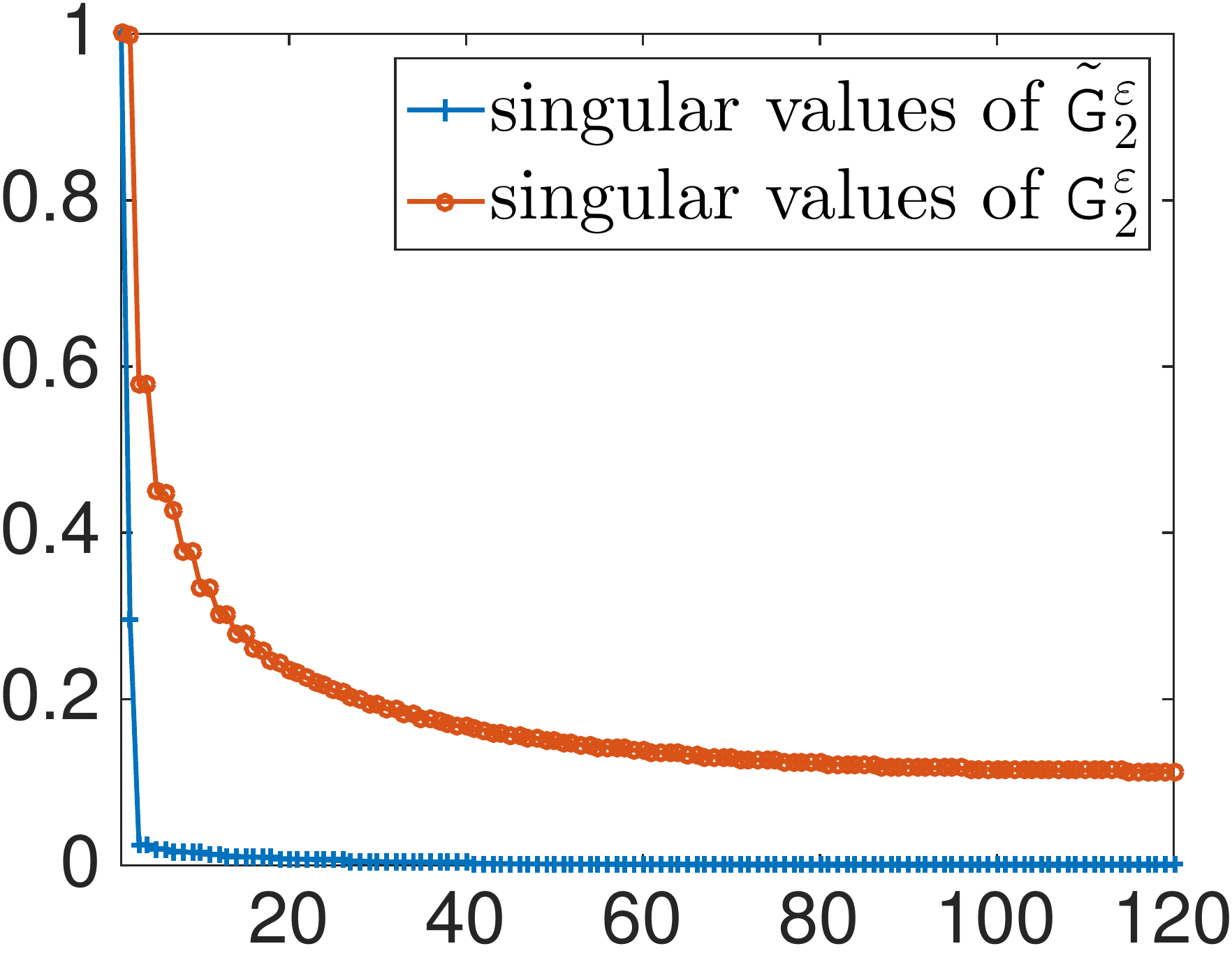}
	\caption{A comparison of normalized singular values of
		$\Gmat_{2}^\vep$ and $\wt{\Gmat}_2^\vep$ when $\vep =
		2^{-6}$. Use of a buffer zone ensures that Green's matrix
		enjoys faster decay in its singular
		values.} \label{fig:rte_local}
\end{figure}

\begin{figure}\label{fig:RTE_LocalError}
	\centering
	\includegraphics[width=0.6\textwidth]{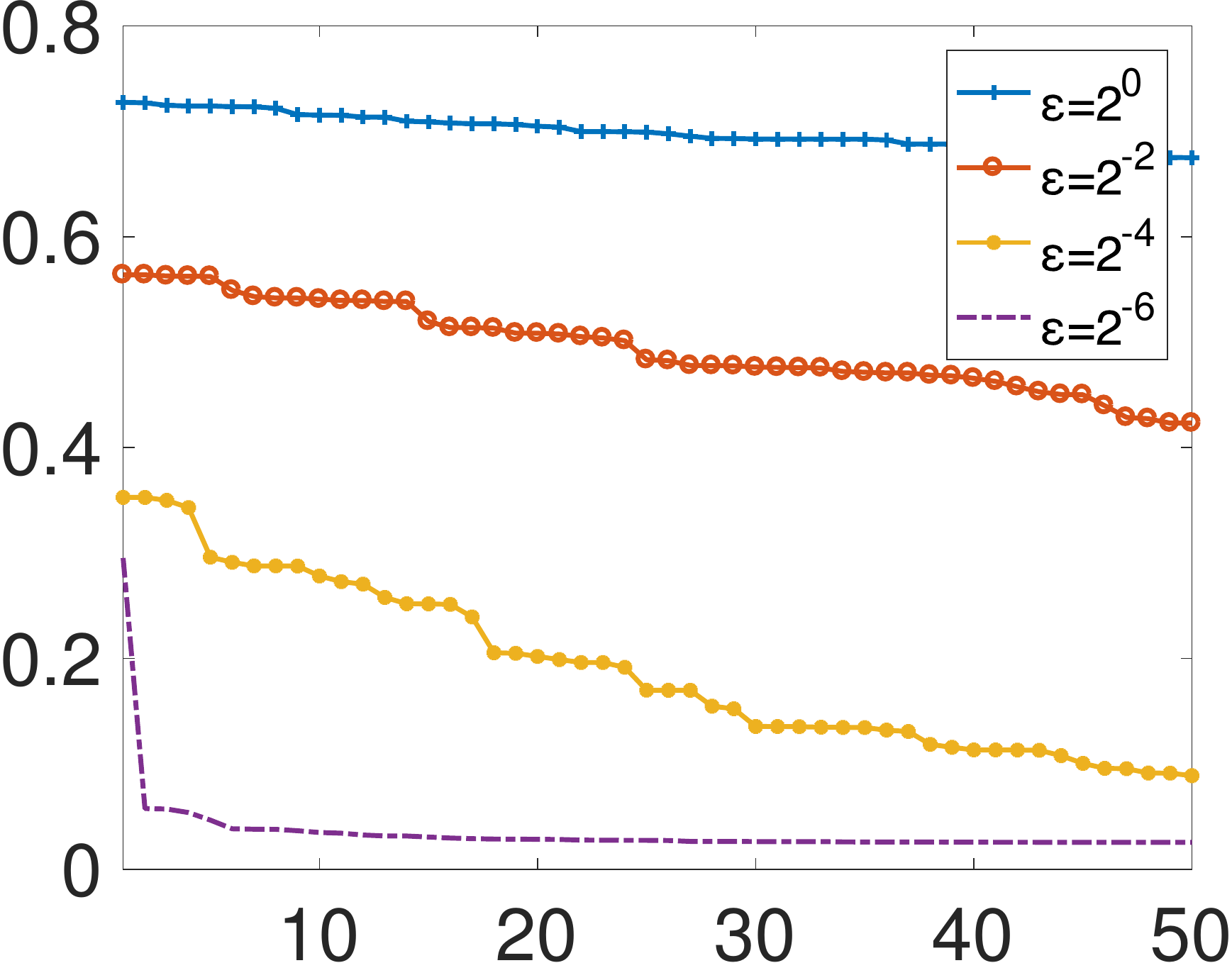}
	\caption{Effectiveness of the random sampling for different
		Knudsen number: $\vep = 2^0,2^{-2},2^{-4},2^{-6}$ on
		buffered domain $\tilde{\mathcal{K}}_2$. For each $\vep$,
		the approximate range captured by $\Gmat_2^{\vep,r}$
		improves as the number of random modes $k_m$ increases. Much
		better approximations are obtained for smaller $\vep$ than
		for larger values.}
	\label{fig:rte_local_random}
\end{figure}

\begin{figure}
	\includegraphics[width=0.3\textwidth]{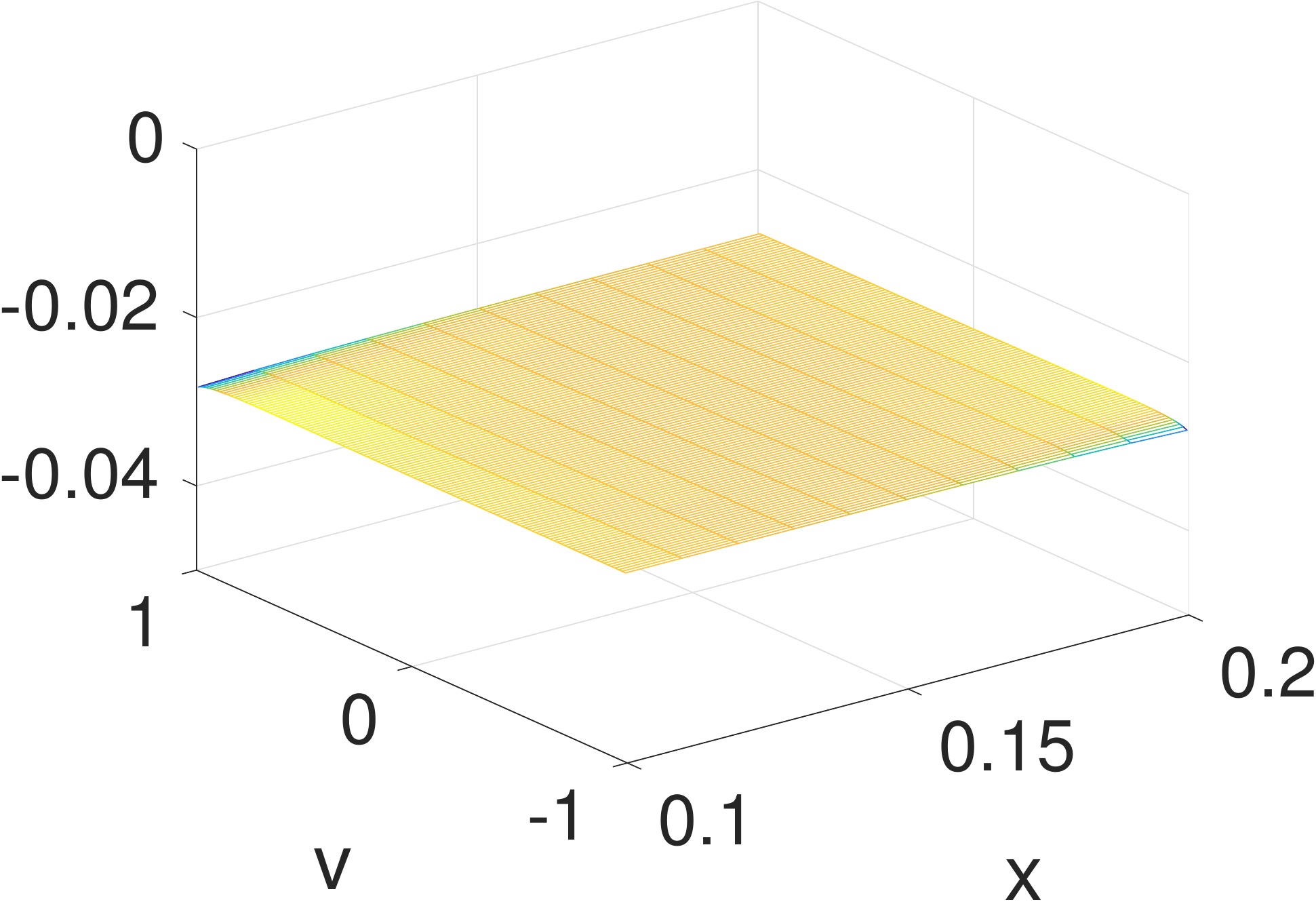}
	\includegraphics[width=0.3\textwidth]{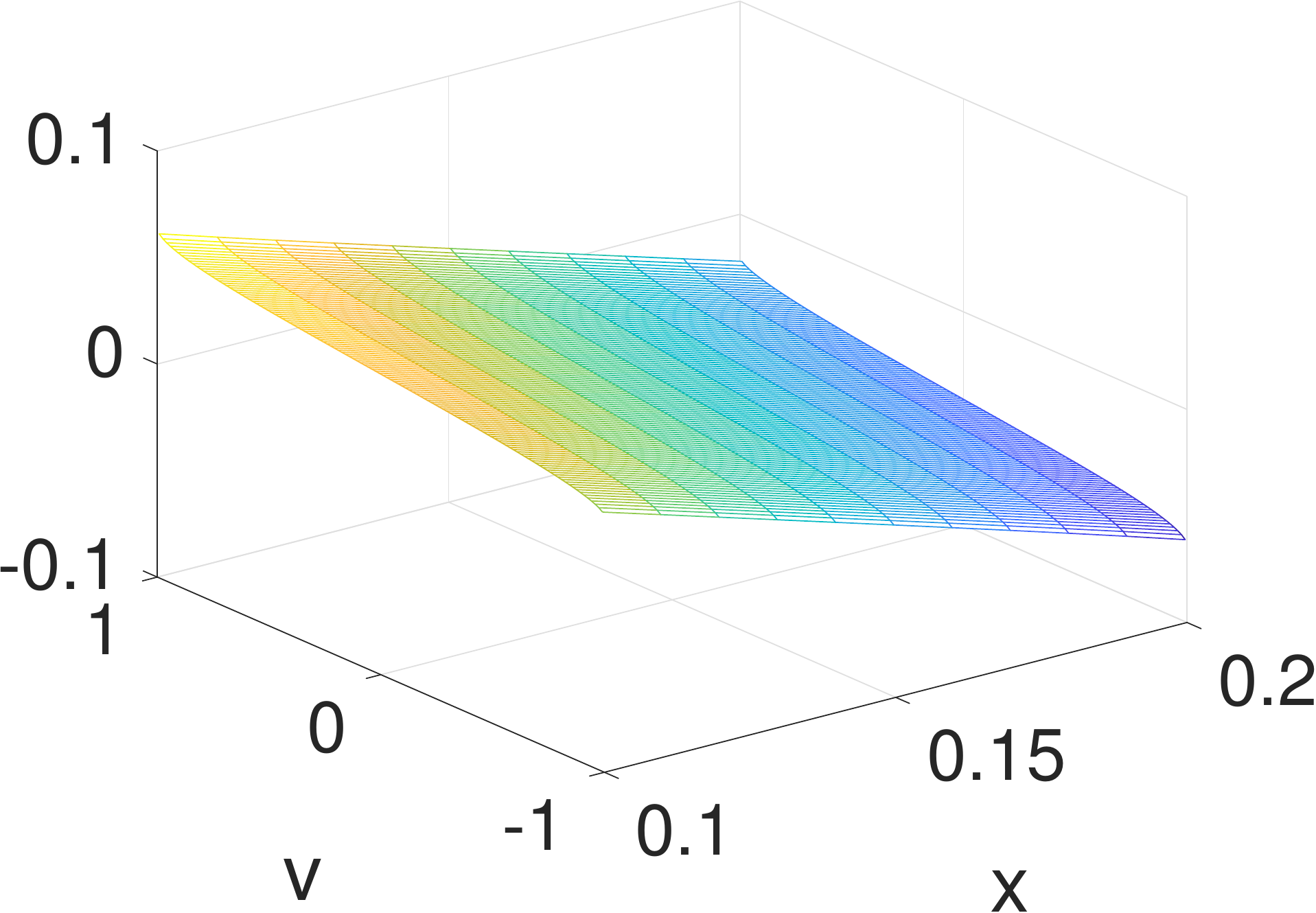}
	\includegraphics[width=0.3\textwidth]{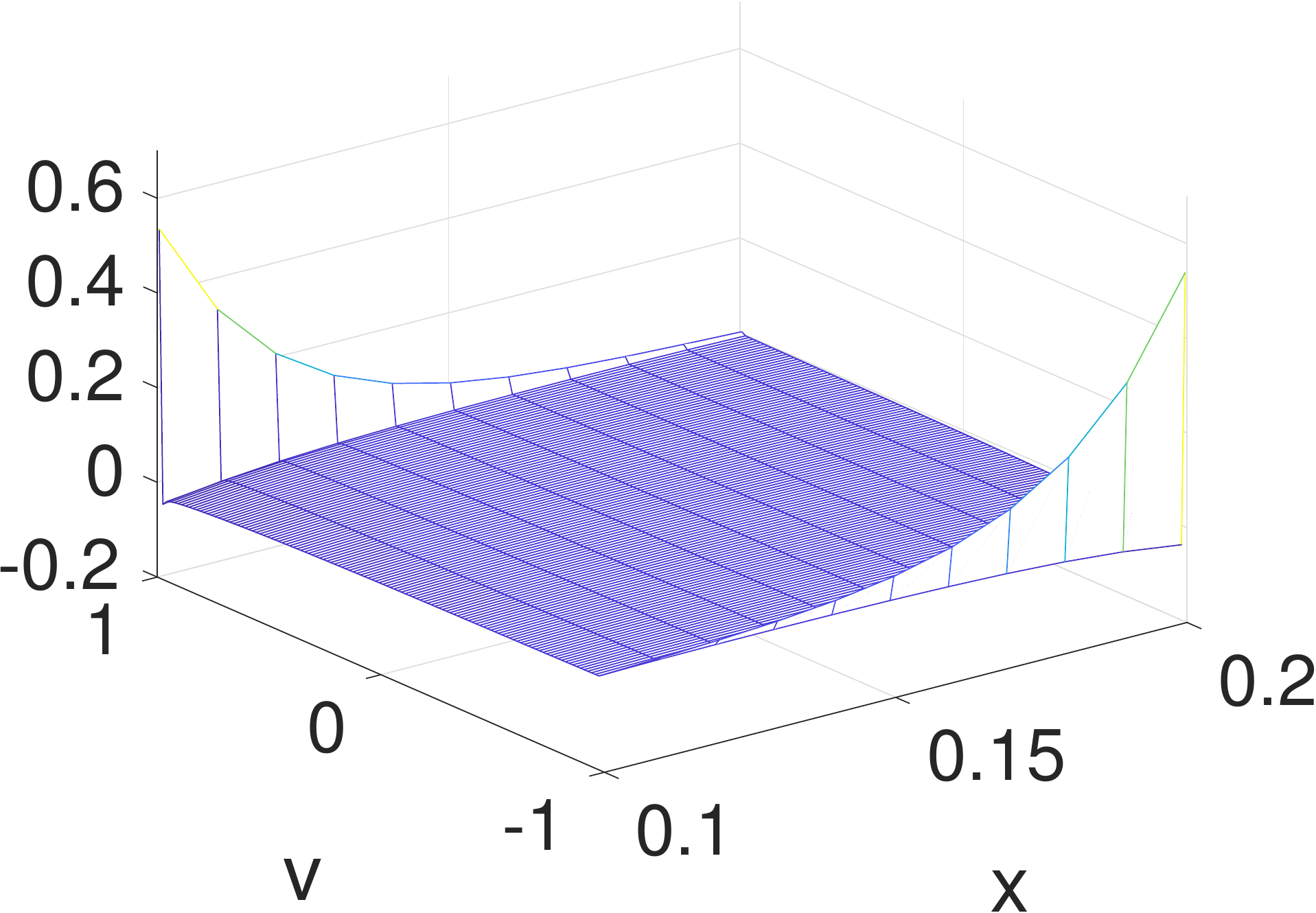}\\
	\includegraphics[width=0.3\textwidth]{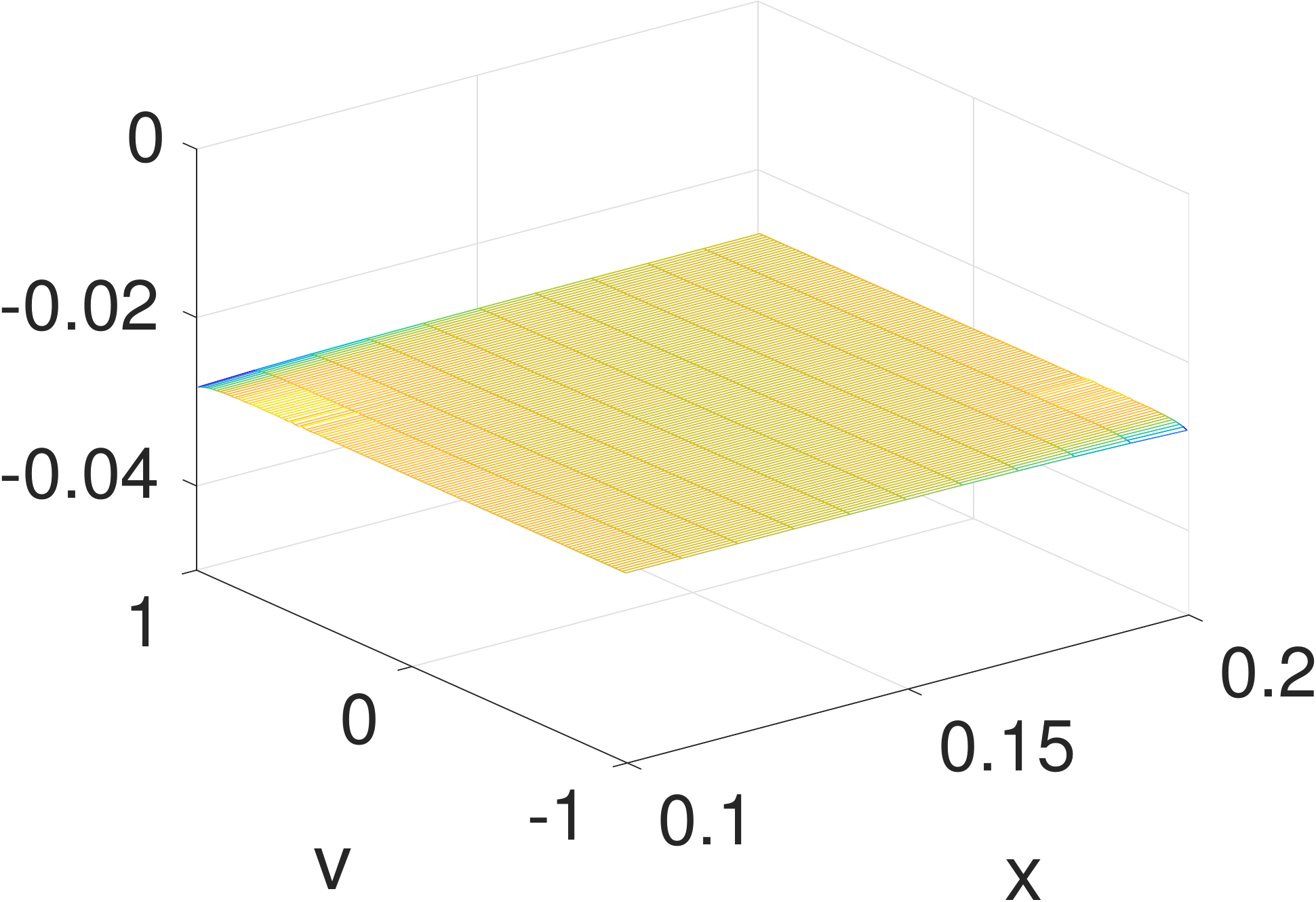}
	\includegraphics[width=0.3\textwidth]{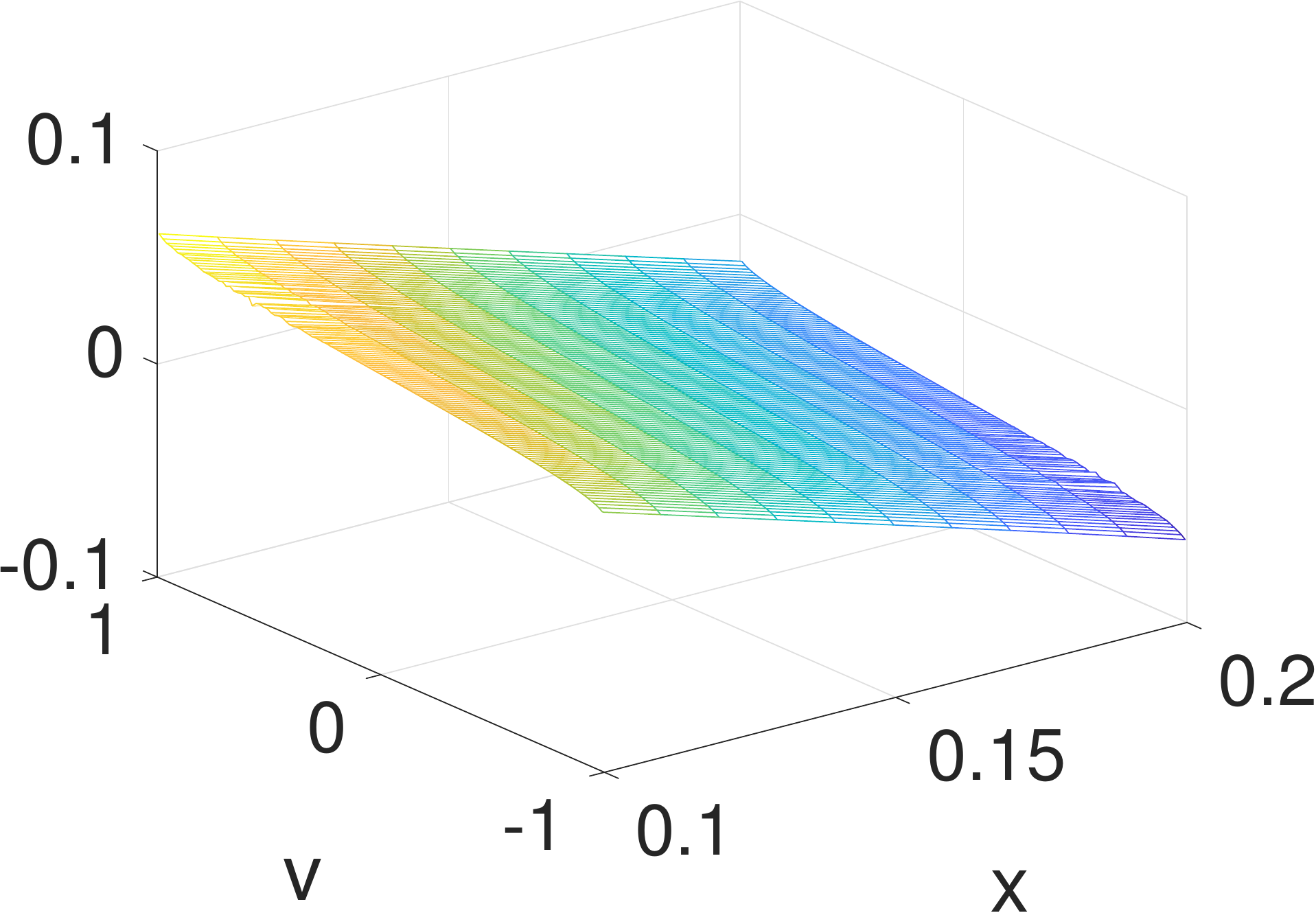}
	\includegraphics[width=0.3\textwidth]{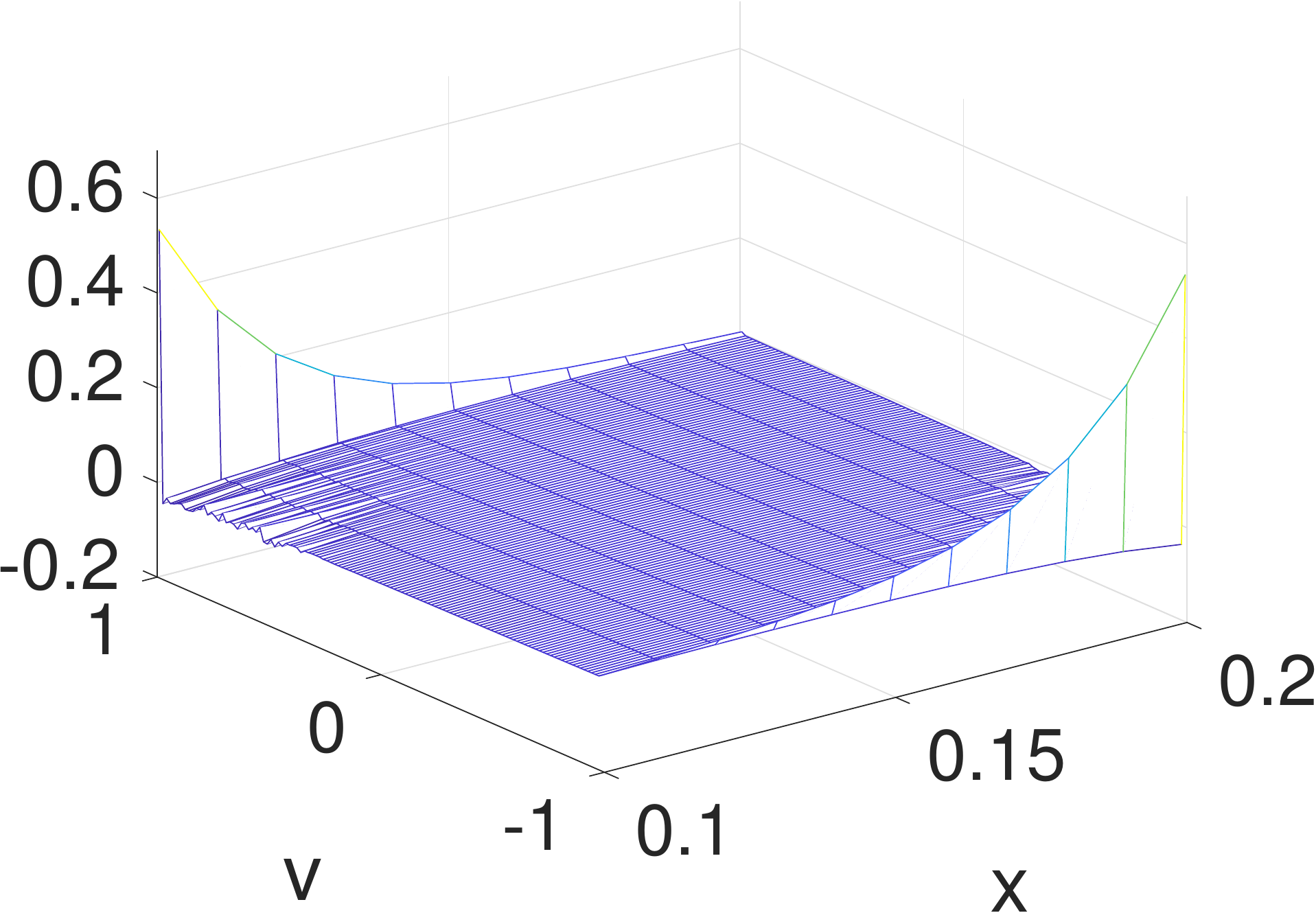}\\
	\caption{For small Knudsen number $\vep= 2^{-6}$, the first row shows the first three singular vectors of $\Gmat^\vep_{2}$ and the second row shows their projection  into  span$\{\Gmat^{\vep,r}_{2}\}$ with $k_2=50$.}
	\label{fig:rte_local_mode}
\end{figure}

\subsubsection{Global test}
In the global test, we consider solving \eqref{eqn:RTE} with boundary
data
\[
\phi(v) = \begin{cases}
3 + \sin(2\pi v)\,,\quad(x=0,v > 0)\\
2 + \sin(2\pi v)\,, \quad(x=1,v<0),
\end{cases}
\]
and compare the numerical solutions of the full-basis and randomized
reduced-basis approaches.~\cref{fig:rte_global} shows three
solutions: reference solution, the solution obtained from the reduced
basis with $k_m=10$, and the solution obtained from the reduced basis
with $k_m=50$. Results are given for $\vep = 2^0$ and $\vep =
2^{-6}$. 
We see that the information contained in $n_m=120$ bases is largely
captured by the random bases with $k_m=10$ (for all $m$) when
$\vep=2^{-6}$, at considerably lower computational cost.
The quantitative error-decay as a function of $k_m$ is
plotted in~\cref{fig:rte_global_error}.
\begin{figure}
	\includegraphics[width=0.3\textwidth]{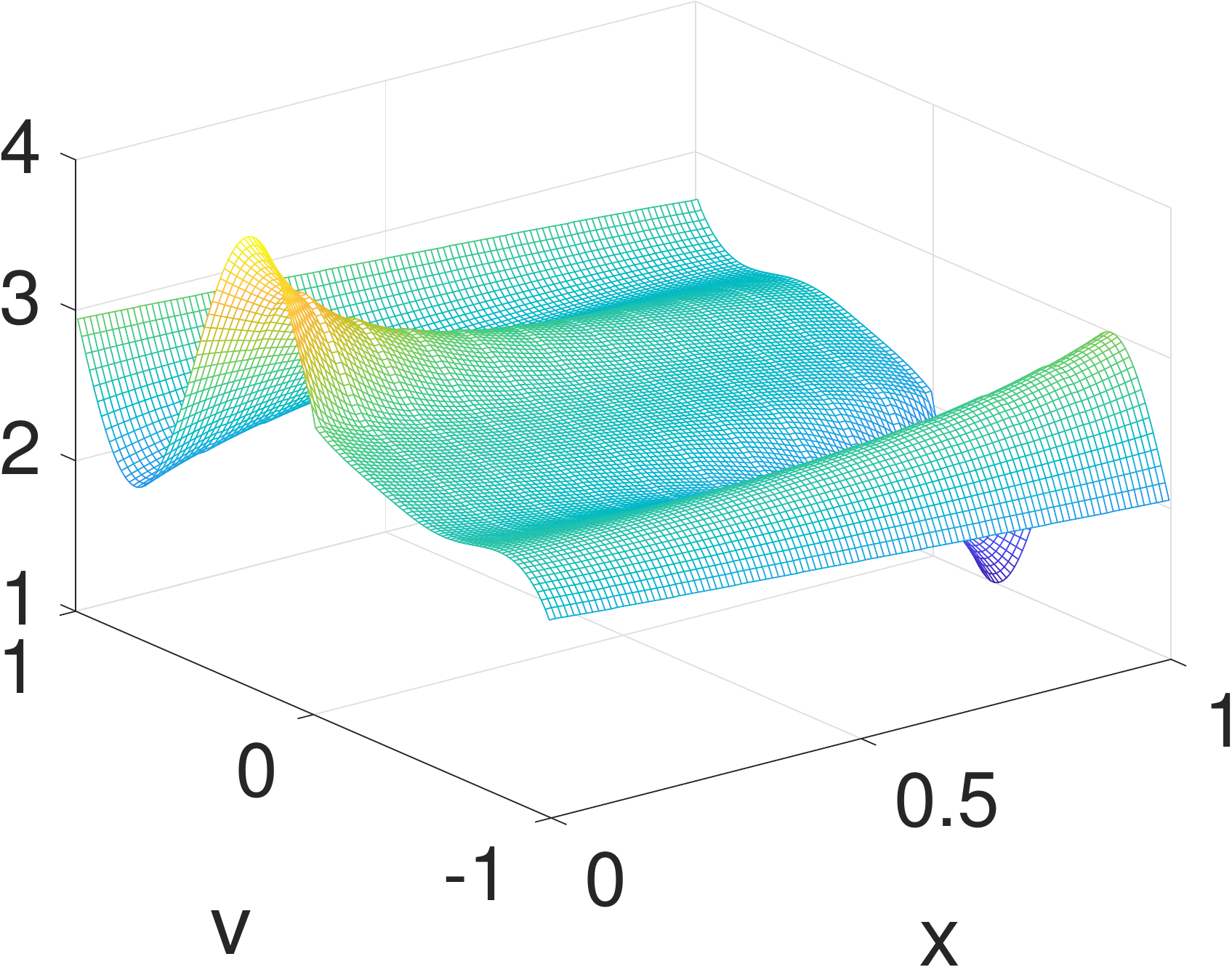}
	\includegraphics[width=0.3\textwidth]{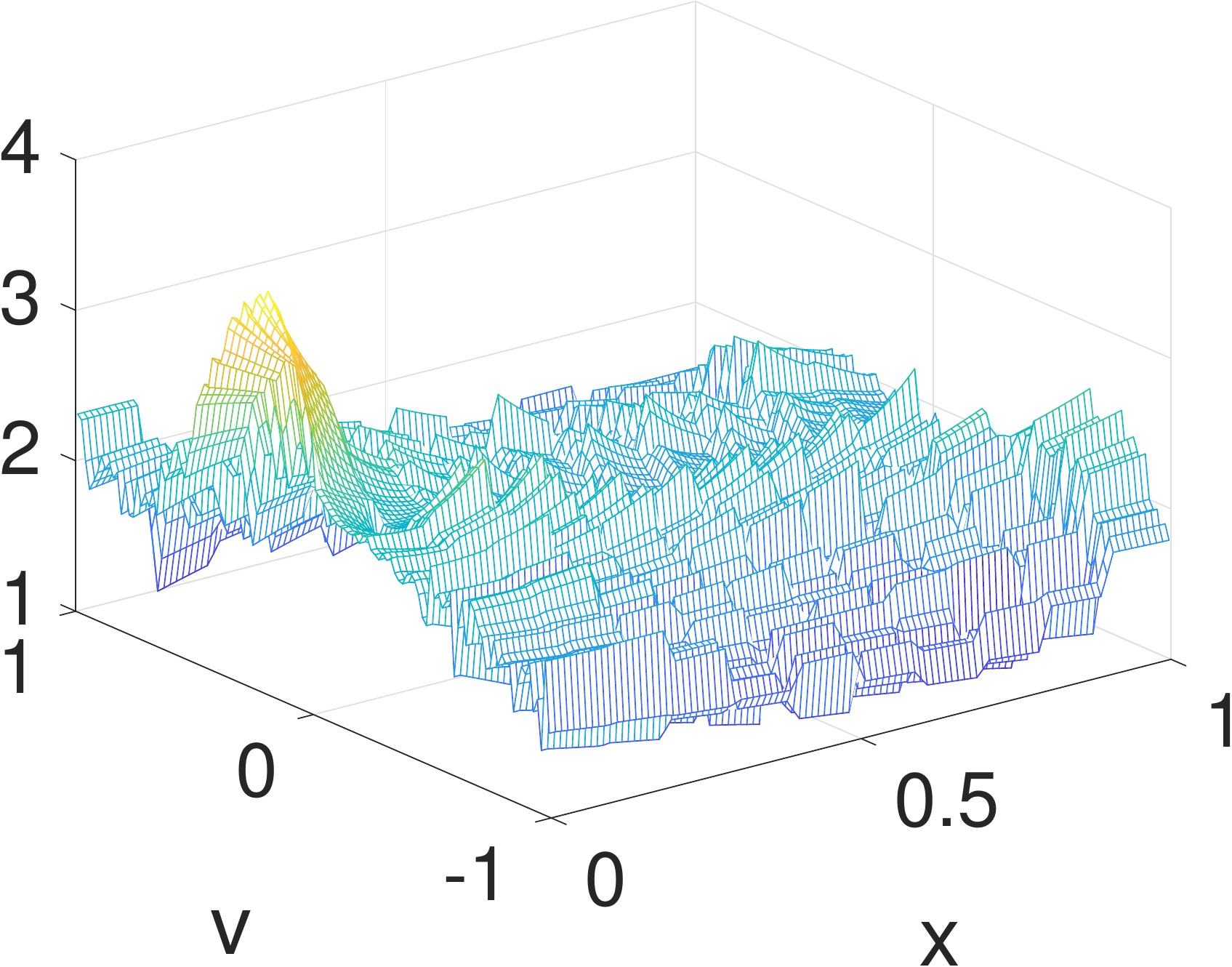}
	\includegraphics[width=0.3\textwidth]{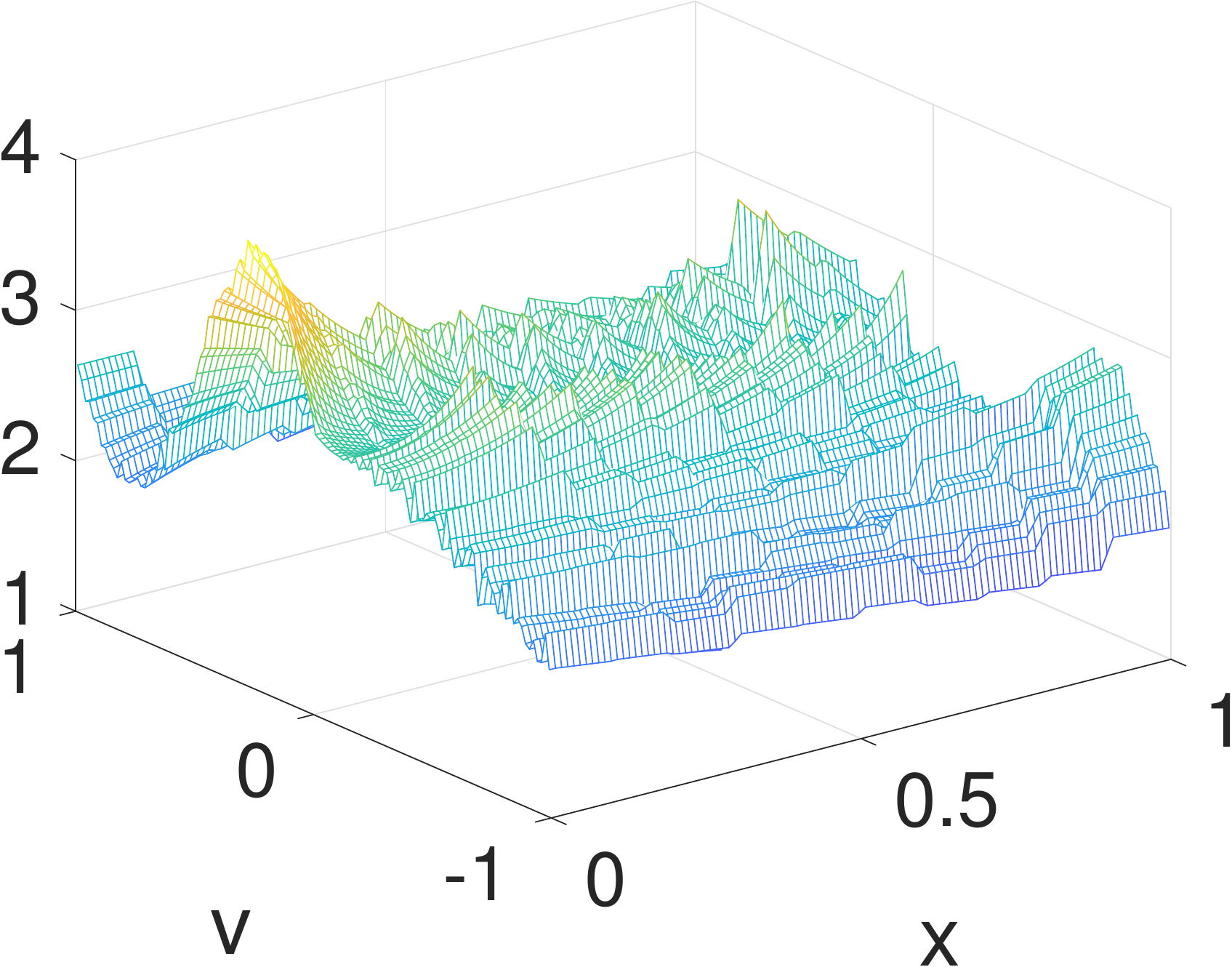}\\
	\includegraphics[width=0.3\textwidth]{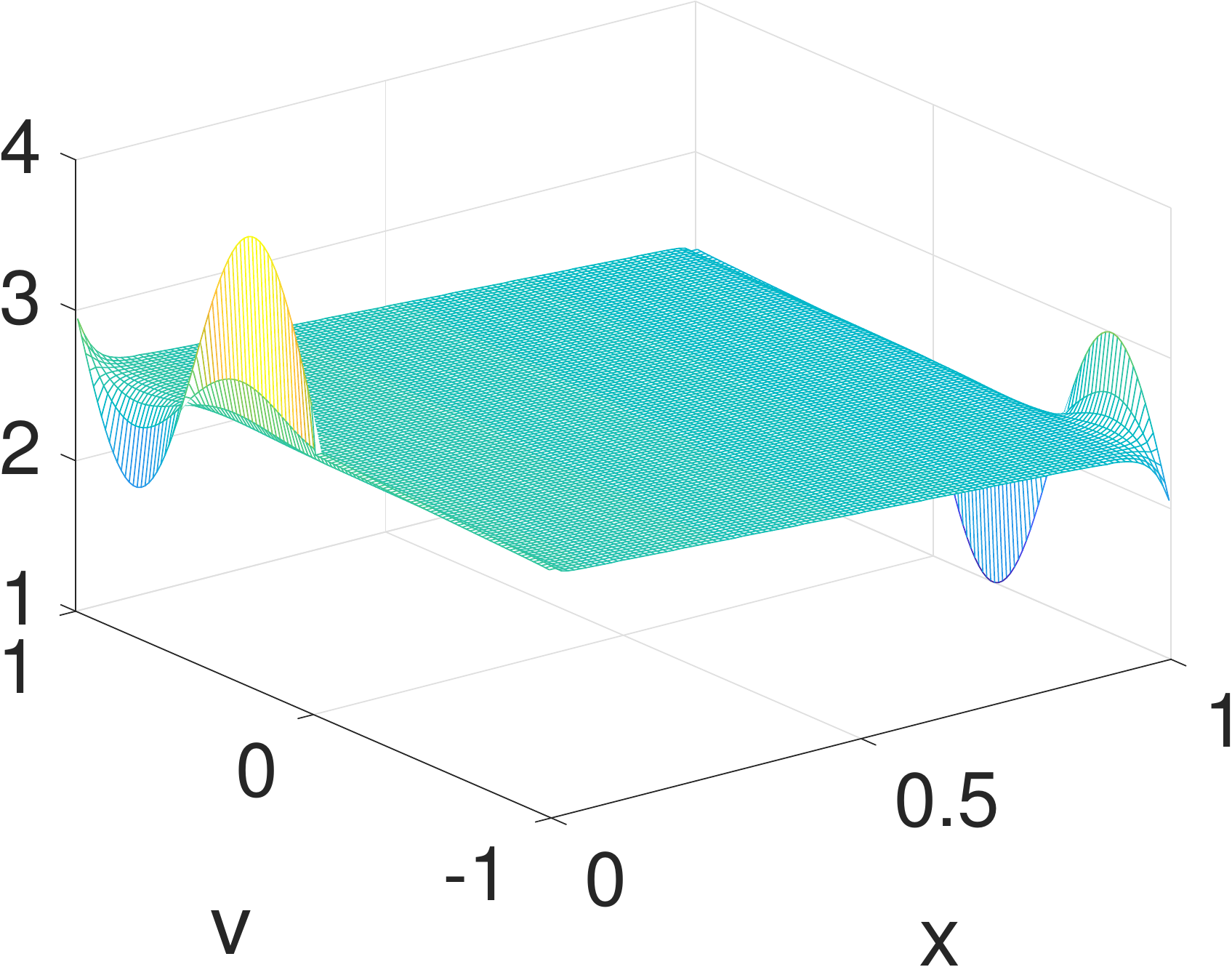}
	\includegraphics[width=0.3\textwidth]{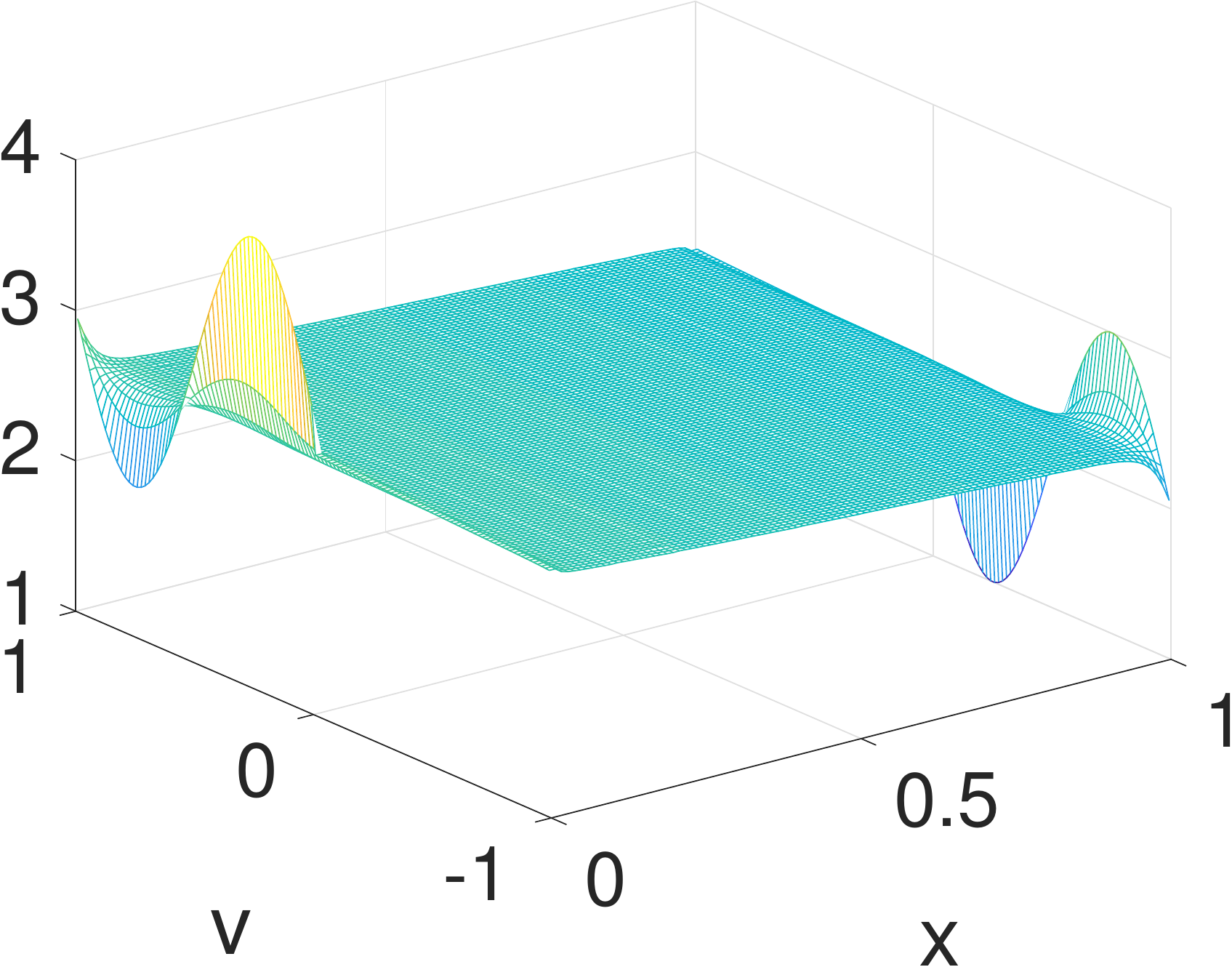}
	\includegraphics[width=0.3\textwidth]{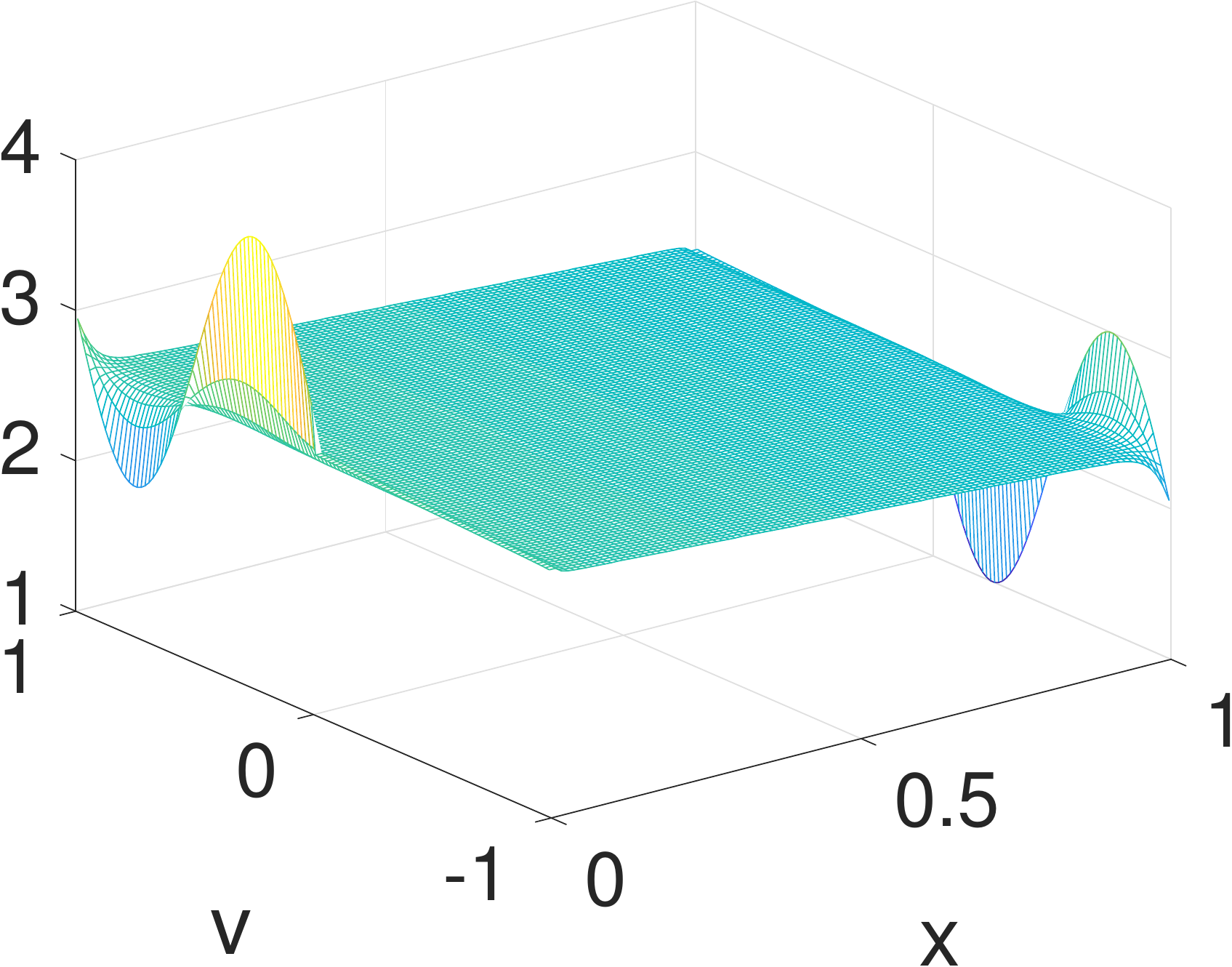}\\
	\caption{The two rows of plots are for $\vep=1$ and $\varepsilon=2^{-6}$ respectively. The leftmost column show a reference solution obtained with fine grids. The middle column and the rightmost column are solutions obtained from randomized reduced bases with $k_m=10$ and $k_m=50$ (for all $m$), respectively.}\label{fig:rte_global}
\end{figure}
\begin{figure}
	\centering
	\includegraphics[width=0.6\textwidth]{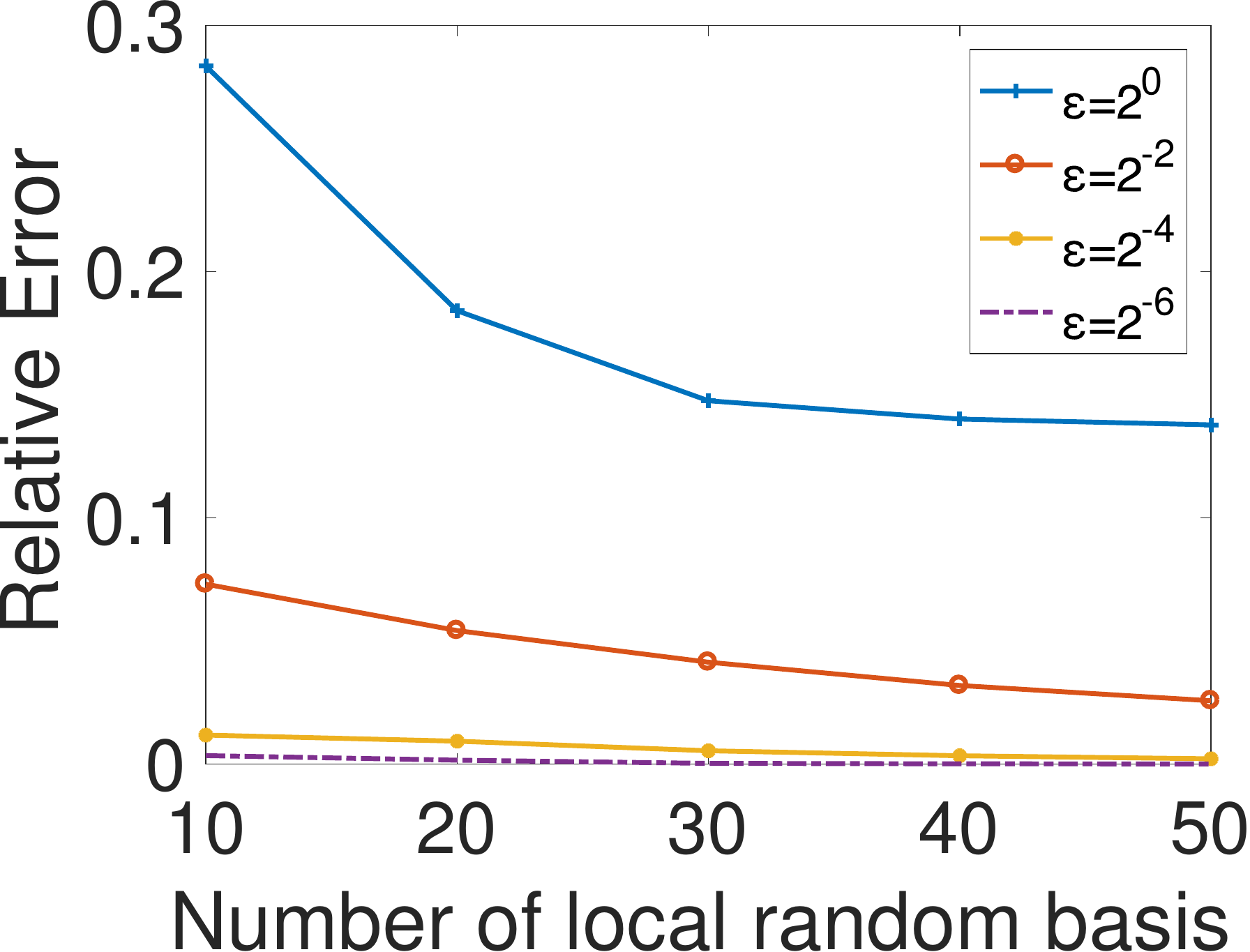}
	\caption{The global error as a function of $k_m$. As $k_m$, the number of random modes per patch increases, the relative error decreases. For fixed number of random modes, the relative error is better for small $\vep$.
	}
	\label{fig:rte_global_error}
\end{figure}

\section{Example 2: elliptic equation with highly oscillatory media}
\label{sec:elliptic}

We now consider elliptic equations with oscillatory media on the
domain $\Kcal = [0, 1]^2$ with Dirichlet boundary conditions. The
problem is
\begin{align}\label{eqn:Elliptic}
\nabla_x\cdot\left(a\left(x,\frac{x}{\varepsilon}\right)\nabla_xu^\vep\right)  = 0\,, \quad & \text{in }\Kcal= [0,1]^2\,, \\
\label{eqn:elliptic_bdy}
u^\vep  = \phi(x)\,, \quad & \text{on } \Gamma = \partial \Kcal,
\end{align}
where the coefficient field $a = a(x, {x}/{\vep})$ is oscillatory
because of its explicit dependence on the fast variable
${x}/{\vep}$. ($\vep$ indicates the scale of oscillation in the
coefficient field.)

We solve \eqref{eqn:Elliptic} on a coarse mesh $\{(x_{m_1},
y_{m_2})\mid x_{m_1} = m_1 H, \;\; y_{m_2} = m_2 H\}$ with
$H=1/M$. The coarse mesh size $H$ is chosen independent of the small
parameter $\vep$.  The domain $\Kcal$ is decomposed into patches defined by
\begin{equation}
\Kcal = \bigcup_m \Kcal_{m}\,,\quad\text{with}\quad\Kcal_m = [ x_{m_1-1},x_{m_1}] \times[y_{m_2-1},y_{m_2}] \,,
\end{equation}
where $m=(m_1,m_2)$ is a multi-index. Two patches $\Kcal_m$ and
$\Kcal_n$ share boundaries if they are adjacent, and we define the
shared edge as follows:
\begin{equation*}
L_{mn} =  \Kcal_m\cap\Kcal_n \,.
\end{equation*}
Thus $L_{mn}$ is nontrivial only if $(m_1,m_2) = (n_1\pm 1,n_2)$ or
$(m_1,m_2) = (n_1,n_2\pm1)$; see~\cref{fig:dd_E}. Note too that
$L_{mn}=L_{nm}$.
\begin{figure}
	\centering
	\includegraphics[width = 0.6\textwidth]{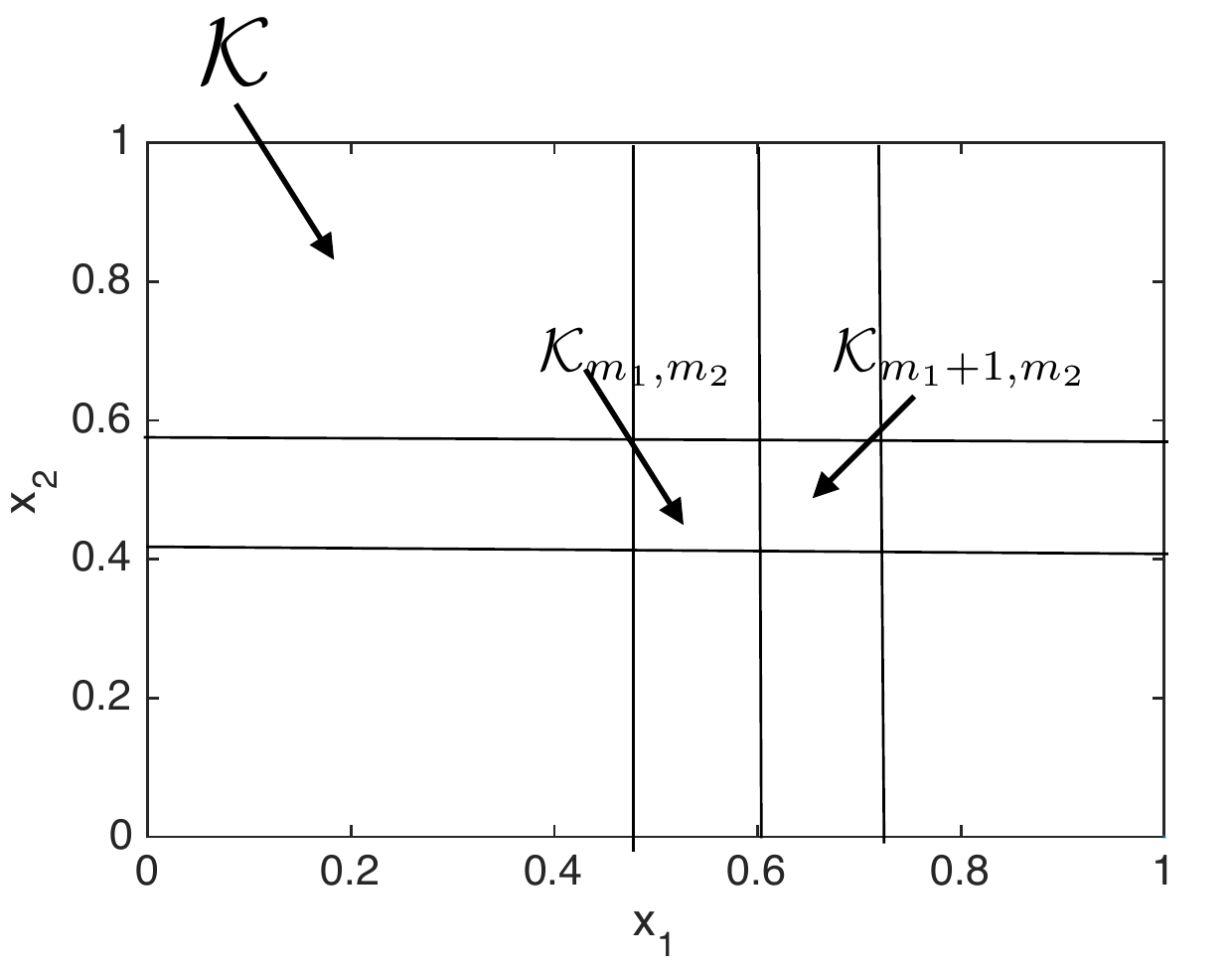}
	\caption{Domain $\Kcal$ is decomposed into patches, each defined by a multi-index $m=(m_1,m_2)$.}\label{fig:dd_E}
\end{figure}

\subsection{Full basis approach} 

\subsubsection*{Offline}
In the full-basis scheme, we prepare the local solution space
functions $b_{m,i}$ in the offline step by solving
\eqref{eqn:Elliptic} in every patch $\Kcal_m$ with boundary conditions
that are non-vanishing at one just grid point $i$ on the boundary of
the patch:
\begin{equation}\label{eqn:elliptic_b}
\begin{cases}
\nabla_x\cdot\left(a\left(x,\frac{x}{\varepsilon} \right)\nabla_xb_{m,i}\right) = 0\,, & \quad x \in \Kcal_m,\\
b_{m,i}  = \delta_i \,,  & \quad x\in \Gamma_m = \partial\Kcal_m \,,
\end{cases}
\end{equation}
where $\delta_i=1$ at the $i$-th boundary grid point of $\Gamma_m$
and is zero at all other grid points in $\Gamma_m$.  The solutions
$\{b_{m,i}\,, i= 1,\ldots,n_m\}$ span the space of local solutions with all possible boundary
conditions, and we assemble them into the local Green's matrix for
$\Kcal_m$:
\begin{equation} \label{eqn:lr3}
\Gmat^\vep_m = \left[b_{m,1}\,,\dotsc,b_{m,n_m} \right]\,.
\end{equation}
Note that to compute the basis functions, we use fine discretization with meshsize $h\ll \vep$, which leads to
$n_m = \mathcal{O}({1}/{\vep})$.
(Details of the fine mesh are discussed in~\cref{sec:num.exp2}.)

\subsubsection*{Online}
We write the global solution as a linear combination of all local
basis functions, with coefficients $c_{m,i}$:
\begin{equation}\label{eqn:dd_soln_elliptic}
u^\vep = \sum_mu^\vep_m = \sum_m\sum_{i=1}^{n_m} c_{m,i}b_{m,i}\,.
\end{equation}
The coefficients are determined by enforcing the following constraints:
\begin{itemize}
	\item[$\ast$]{Continuity across edges $L_{mn}$:} $u_m(L_{mn}) =
	u_{n}(L_{mn})$ and $a\partial_nu_m(L_{mn}) =
	-a\partial_nu_{n}(L_{mn})$ if $L_{mn}\neq\emptyset$, where
	$\partial_n$ denotes the outer normal derivative on the boundary;
	\item[$\ast$]{Boundary condition on $\partial \Kcal$:}
	$u|_{\partial\Kcal} = \phi$.
\end{itemize}
Denote by $\Mmat_{m,n}$ the matrix that maps $c_m$ to $u_m(L_{mn})$,
and by $\Wmat_{m,n}$ the matrix that maps $c_m$ to $a\partial_n
u_n(L_{mn})$, that is,
\begin{equation*}
\Mmat_{m,n}  c_m = u_m(L_{mn})\,,\quad \Wmat_{m,n}  c_m = a\partial_nu_m(L_{mn})\,.
\end{equation*}
(Note that $\Mmat_{m,n}$ is a submatrix of $\Gmat^\vep_m$.)
From the continuity condition, we have
\begin{equation}\label{eqn:elliptic_cont_cond}
\begin{cases}
\Mmat_{m,n} c_m - \Mmat_{n,m} c_{n}=0\,,\quad x\in L_{mn}\\
\Wmat_{m,n} c_{m} + \Wmat_{n,m}c_{n}=0\,,\quad x\in L_{mn}.
\end{cases}
\end{equation}
Similarly, we define by $\Mmat_{m,\text{ext}}$ the matrix that maps $c_m$
to the intersection of $\partial \Kcal_m$ with $\partial \Kcal$. From
the boundary condition, we have
\begin{equation}\label{eqn:elliptic_bdy_cond}
\Mmat_{m,\text{ext}} c_m = \phi \,,\quad x\in\partial\Kcal_m\cap\partial\Kcal\,.
\end{equation}

By assembling the conditions \eqref{eqn:elliptic_cont_cond}
and~\eqref{eqn:elliptic_bdy_cond} for all $m$ and $n$, and solving for
the coefficients $\{c_{m,i}\,,m=1\,,\dotsc,M\,,i = 1\,,\dotsc n_m\}$,
we obtain $u^{\vep}$ from \eqref{eqn:dd_soln_elliptic}. The linear
system has the form
\begin{equation}\label{eqn:elliptic_cond}
\Pmat c = d\,,
\end{equation}
with $d = [\phi,0]$, $c = [c_{m,i}]$ and $\Pmat$ is formed by the
collection of $\Mmat_{m,n}$, $\Wmat_{m,n}$, and $\Mmat_{m,\text{ext}}$.

\subsection{Reduced basis approach}
As in~\cref{sec:buffer_transport}, we define buffered patches
$\wt{\Kcal}_m$ such that $\Kcal_m \subset\subset \wt{\Kcal}_m$, and
solve a local problem on each buffered patch. When we restrict the
local solutions to $\Kcal_m$, we find that (as before) these solutions
lie approximately in a lower-dimensional space.  Similarly to
Equation~\eqref{eqn:elliptic_b}, we define the local problems as
follows:
\begin{equation} \label{eqn:ux7}
\begin{cases}
\nabla_x\cdot\left(a(x,\frac{x}{\varepsilon})\nabla_x\wt{b}_{m,i}\right) = 0\,, & \quad x \in \wt{\Kcal}_m, \\
\wt{b}_{m,i} |_{\partial\wt{\Kcal} _m} = \delta_i, &
\end{cases}
\end{equation}
then define the local solution space via the following Green's matrix:
\begin{equation} \label{eqn:lr2}
\wt{\Gmat}^\vep_m = \left[ \wt{b}_{m,1}|_{\Kcal_m},\ldots,\wt{b}_{m,\wt{n}_m}|_{\Kcal_m} \right]\,,
\end{equation}
Since $\spanop \wt{\Gmat}^\vep_m$ contains all local solutions, we
seek a good approximation to $\spanop \wt{\Gmat}^\vep_m$ for the
interior cells during the offline stage.  As shown in
~\cite{bebendorf2003existence}, and similarly to~\cref{sec:buffer_transport}, the matrix $\wt{\Gmat}^\vep_m$ is
low rank and can be compressed through random sampling.
We
solve \eqref{eqn:ux7} with the boundary condition $\delta_i$ replaced
by a function $w_i$ which takes on random values (specifically,
i.i.d. normal random variables) at the grid points of the boundary
$\partial \wt{\Kcal}_m$, that is,
\begin{equation}\label{eqn:elliptic_random}
\begin{cases}
\nabla_x\cdot\left(a(x,\frac{x}{\varepsilon})\nabla_x\wt{r}_{m,i}\right) = 0\,, \quad x \in \wt{\Kcal}_m,\\
\wt{r}_{m,i} |_{\partial\wt{\Kcal}_m} = \omega_i.
\end{cases}
\end{equation}
We do this for $k_m$ choices of random boundary function $w_i$ and
assemble the local reduced Green's matrix from the restricted solutions
$r_{m,i}= \wt{r}_{m,i}|_{\Kcal_m}$, $i=1,2,\dotsc,k_m$:
\[
\Gmat^{\vep,r}_m = \left[r_{m,1}\,,\dotsc,r_{m,k_m} \right] = \wt{\Gmat}_{m}^{\vep} \left[\omega_{m,1}\,,\dotsc,\omega_{m,k_m} \right]\big|_{\Kcal_m} \,.
\]
As done in the full basis approach, the coefficients are determined in the online step, namely, we express the solution as
\begin{equation*}
u^\vep = \sum_mu^\vep_m \approx \sum_m\sum_{i=1}^{k_m} \tilde{c}_{m,i}r_{m,i}\,.
\end{equation*}
and determine the coefficients $\tilde{c}_{m,i}$ by imposing the continuity conditions in the interior boundaries and and boundary conditions on the exterior boundary. 

Similar to the full basis approach, denote $\wt\Mmat_{m,n}$ and $\wt{\Wmat}_{m,n}$ the matrices that map $\tilde{c}_m$ to $u_m(L_{mn})$ and $a\partial_n u_n(L_{mn})$ respectively, that is,
\begin{equation*}
\wt{\Mmat}_{m,n}  \tilde{c}_m = u_m(L_{mn})\,,\quad \wt{\Wmat}_{m,n}  \tilde{c}_m = a\partial_nu_m(L_{mn})\,.
\end{equation*}
By imposing the continuity condition and the exterior boundary
condition, we obtain
\begin{equation}\label{eqn:elliptic_cond_tilde}
\begin{cases}
\wt\Mmat_{m,n} \tilde{c}_m - \wt\Mmat_{n,m} \tilde{c}_{n}=0\,,\quad x\in L_{mn}\\
\wt\Wmat_{m,n} \tilde{c}_{m} + \wt\Wmat_{n,m}\tilde{c}_{n}=0\,,\quad x\in L_{mn}
\end{cases}\,, \quad\text{and}\quad \wt\Mmat_{m,\text{ext}} c_m = \phi \,,\quad x\in\partial\Kcal_m\cap\partial\Kcal\,.
\end{equation}
Assembling the equations, we obtain:
\[
\wt{\Pmat}\tilde{c} = d\,.
\]

However, since the number of coefficients in the reduced basis
approach is significantly smaller than that in the full basis approach
($k_m\ll n_m$ in every patch $\Kcal_m$), while the number of
continuity condition and the boundary condition is not changed, the
system is overdetermined. We thus consider the least-squares solution,
that is:
\begin{equation}\label{eqn:elliptic_cond_reduced}
\tilde{c} = \text{argmin}_e\|\wt{\Pmat} e-d\|_2\quad \Rightarrow\quad \tilde{c} = (\wt{\Pmat}^\top\wt{\Pmat})^{-1}\wt{\Pmat}^\top d\,.
\end{equation}
Alternatively, we could enforce the boundary conditions exactly and
relax only the continuity condition, as in the following constrained
least-squares formulation: such that:
\begin{equation*}
\min_c\sum_{m,n}\|\wt\Mmat_{m,n} c_m - \wt\Mmat_{n,m} c_{n}\|^2_{2,mn}+\|\wt\Wmat_{m,n} c_{m} + \wt\Wmat_{n,m}c_{n}\|^2_{2,mn}\,,\quad\text{such that}\quad \wt\Mmat_{m,\text{ext}} c_m = \phi\,.
\end{equation*}
Here $\|\cdot\|_{2,mn}$ denotes $L_2$ norm confined on $x\in
L_{mn}$. 
If we assume a uniform mesh with $n_m$  constant for all patches, then matrix $\wt{\Pmat}$ is of size $M_p\times N_p$ where $M_p = \tfrac12 \sqrt{M} (\sqrt{M}+1) n_m$ and  $N_p = \sum_{m=1}^M k_m$. Similar to the case of RTE, the typical time complexity for linear regression problem \eqref{eqn:elliptic_cond_reduced} is of order $\mathcal{O}(N_p^2(M_p+N_p))$. 
Numerically, we obtain satisfactory results from
\eqref{eqn:elliptic_cond_reduced}, which we present  in the
next subsection.

\subsection{Numerical test} \label{sec:num.exp2}

We set the domain to be $\Kcal=[0,1]^2$ and define the media as
follows, for $x =(x_1,x_2) \in \Kcal$:
\begin{equation*}\label{eqn:media}
a\left(x,\frac{x}{\vep} \right)= 2 + \sin(2\pi x_1) \cos(2\pi x_2) + \frac{2 + 1.8\sin(\frac{2\pi x_1}{\vep})}{2 + 1.8\cos( \frac{2\pi x_2}{\vep})} + \frac{2 + \sin(\frac{2\pi x_2}{\vep})}{2 + 1.8\cos(\frac{2\pi x_1}{\vep})}\,.
\end{equation*}
For the domain decomposition we set $M = 5$ (for a total of $25$
patches), and each local patch is further divided into a $20$ by $20$
fine mesh so that the mesh parameter $h= 0.01$ can resolve
the smallest scales $\vep = 2^{-4}$.  A complete basis on each patch is formed from $n_m=80$ basis functions. These  functions are computed from a standard finite element $P_1$ method with bilinear nodal basis. The buffered patch
$\wt{\Kcal}_m$ is set to be a square concentric with $\Kcal_m$ but
with all sides twice as long.~\cref{fig:media} illustrates the
setup, for $\vep = 2^{-4}$.
\begin{figure}
	\includegraphics[width=0.45\textwidth]{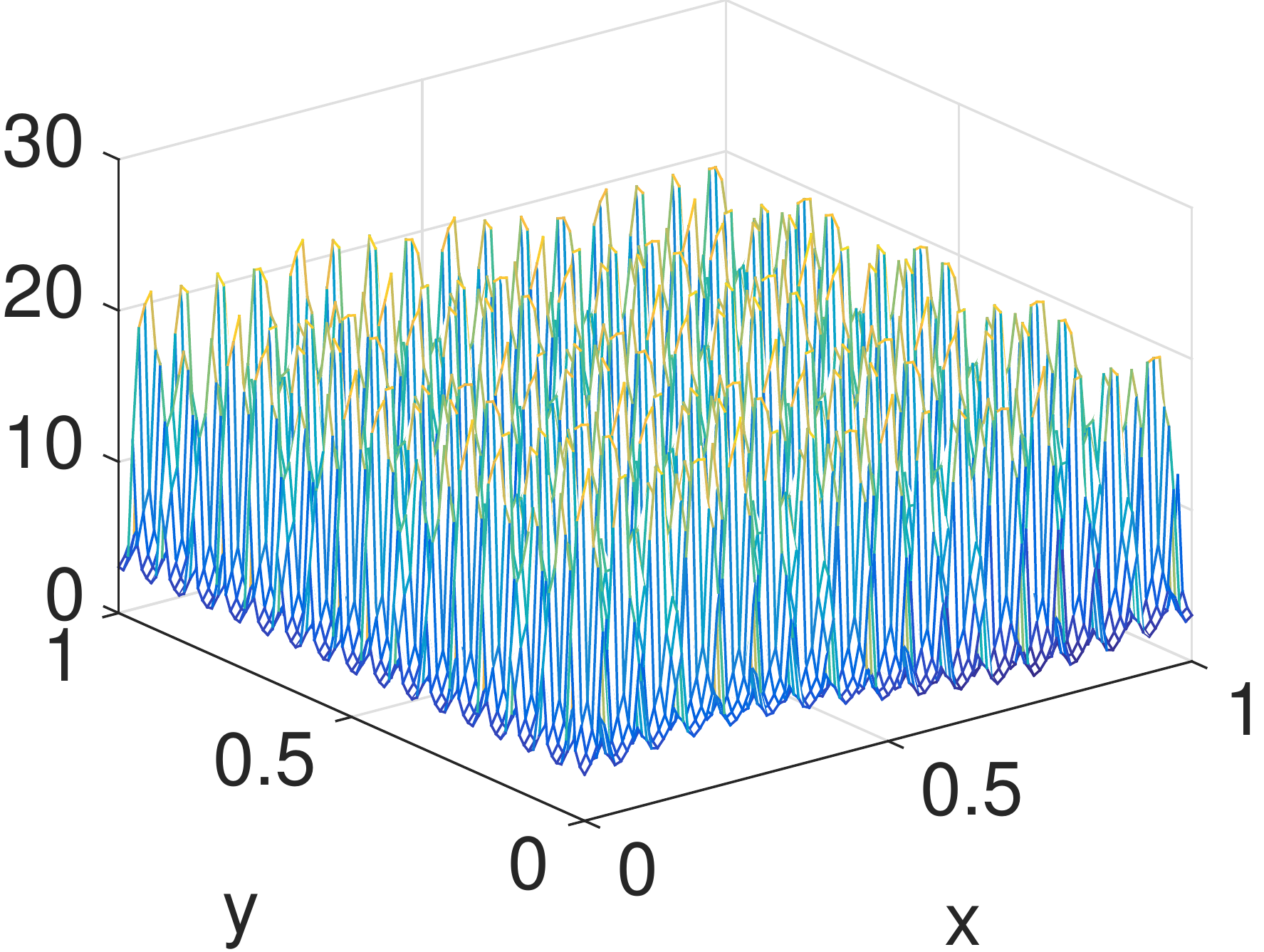}
	\includegraphics[width=0.45\textwidth]{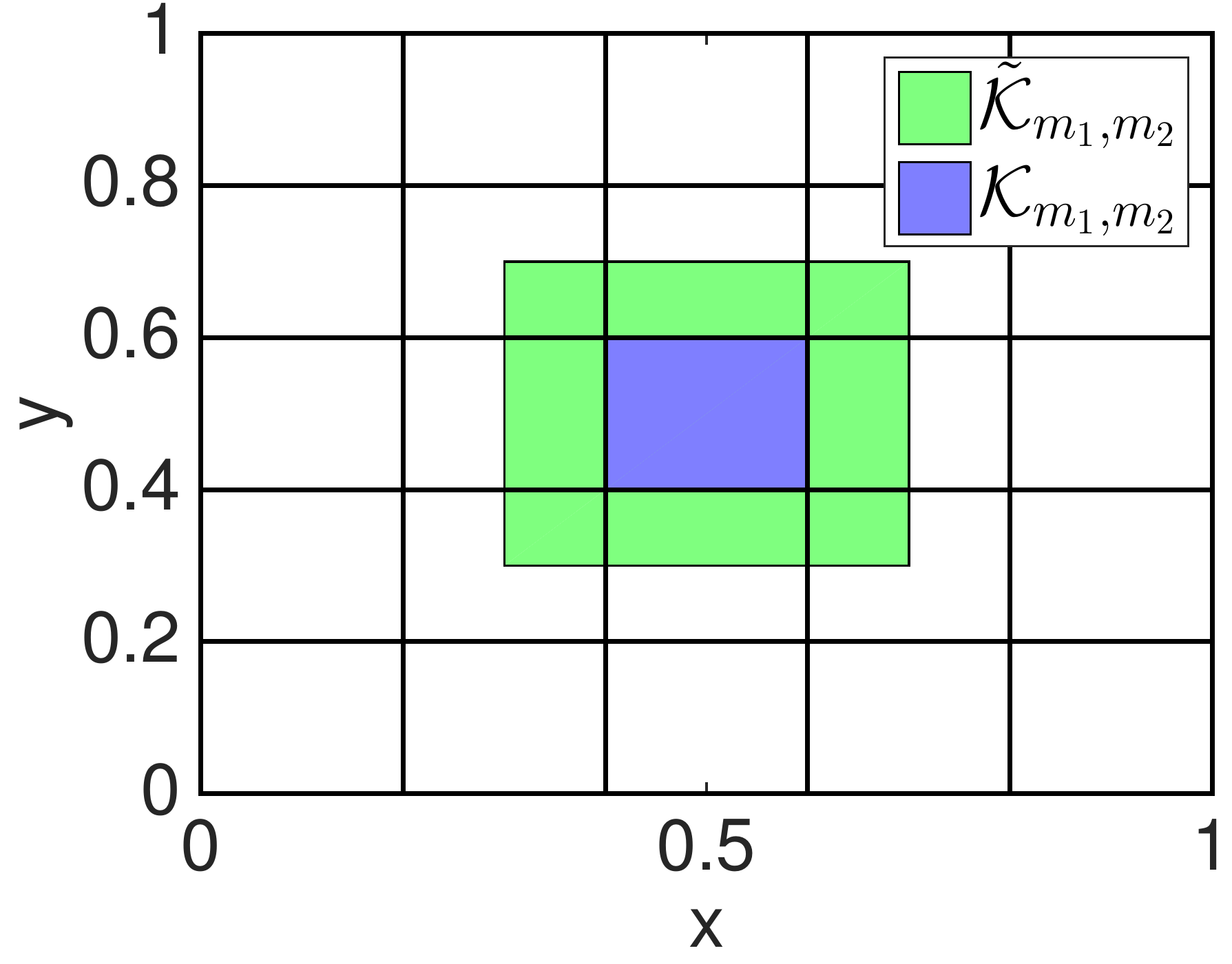}
	\caption{Left: Media used in elliptic equation~\eqref{eqn:Elliptic}. Right: illustration of buffered domain decomposition}\label{fig:media}
\end{figure}

\subsubsection{Local test}
In~\cref{fig:elliptic_local} we show the rank of the Green's
matrices $\Gmat^\vep_{2,2}$ and $\wt{\Gmat}^\vep_{2,2}$ (defined by
\eqref{eqn:lr3} and \eqref{eqn:lr2}, respectively) for the $(2,2)$
patch, with $\vep=2^{-4}$. Use of buffers yields rapid decays in the
singular values of $\wt{\Gmat}^\vep_{2,2}$. We then define the
relative error between $\Gmat^{\vep,r}_{2,2}$ and $\Gmat^{\vep}_{2,2}$
as follows:
\[
\text{error} = \frac{\|\wt{\Gmat}_{2,2}^\vep - \Qmat\Qmat^\top \wt{\Gmat}_{2,2}^\vep \|_2}{\| \wt{\Gmat}_{2,2}^\vep \|_2} \,, \quad \text{with} \quad \Gmat_{2,2}^{\vep,r} = \Qmat\mathsf{R}\,,
\]
where $\Qmat$ is obtained from $QR$ decomposition of
$\Gmat^{\vep,r}_{2,2}$. We see in~\cref{fig:elliptic_local} that
the relative error decays exponentially fast as $k_m$ increases.
In~\cref{fig:elliptic_local_modes}, we plot the first three left
singular vectors of $\wt{\Gmat}^{\vep}_{2,2}$ and their projections
onto $\spanop \{\Gmat^{\vep,r}_{2,2}\}$ with $k_{2,2}=6$. This plot
shows that, visually, $\spanop \{ \Gmat^{\vep,r}_{2,2}\}$ captures
well the leading singular vectors of the full-basis Green's matrix.

\begin{figure}
	\includegraphics[height = 0.19\textheight,width=0.45\textwidth]{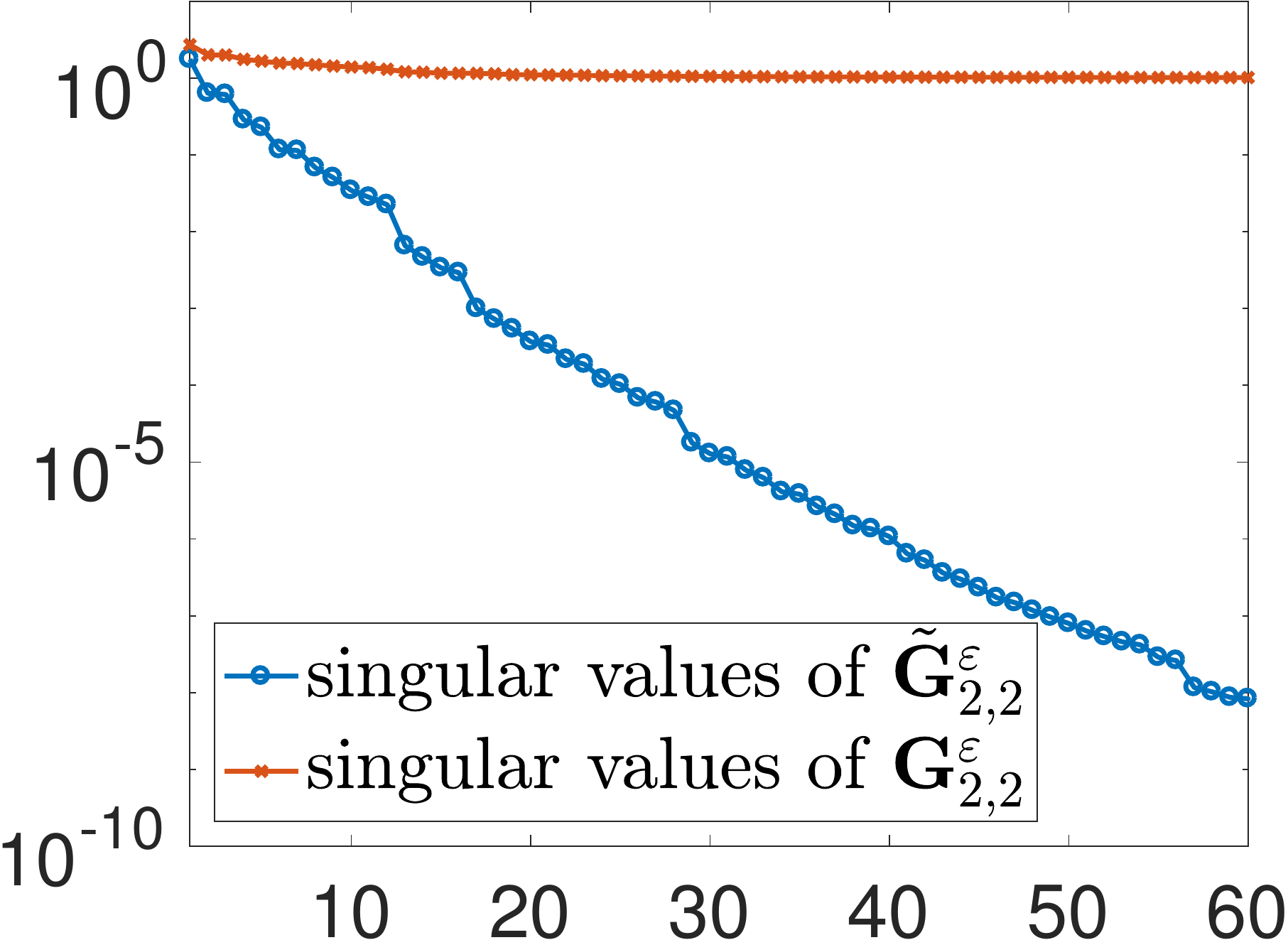}
	\includegraphics[height = 0.2\textheight,width=0.45\textwidth]{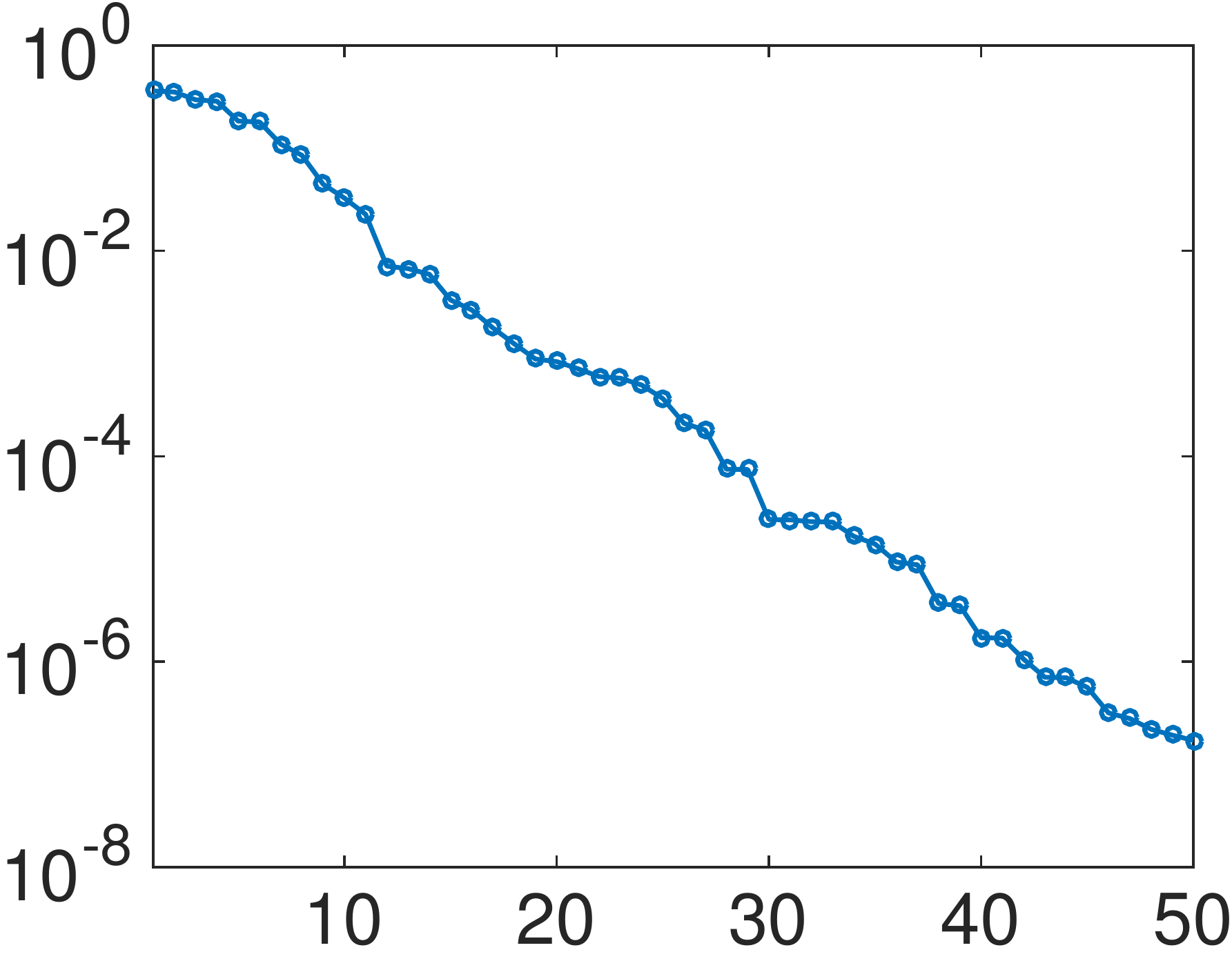}
	\caption{The plot on the left shows the singular values of
		$\Gmat^{\vep}_{2,2}$ (from \eqref{eqn:lr3}) and
		$\wt{\Gmat}^\vep_{2,2}$ (from \eqref{eqn:lr2}), with $\vep =
		2^{-4}$. Use of the buffer zone in the calculation of
		$\wt{\Gmat}^\vep_{2,2}$ causes fast decay of singular
		values, making this matrix approximately low-rank.  The plot
		on the right panel shows relative error between
		$\Gmat_{2,2}^{\vep,r}$ and $\Gmat_{2,2}^\vep$ as we increase
		the number of random modes $k_m$ from 1 to 50.}
	\label{fig:elliptic_local}
\end{figure}

\begin{figure}	
	\includegraphics[width=0.3\textwidth]{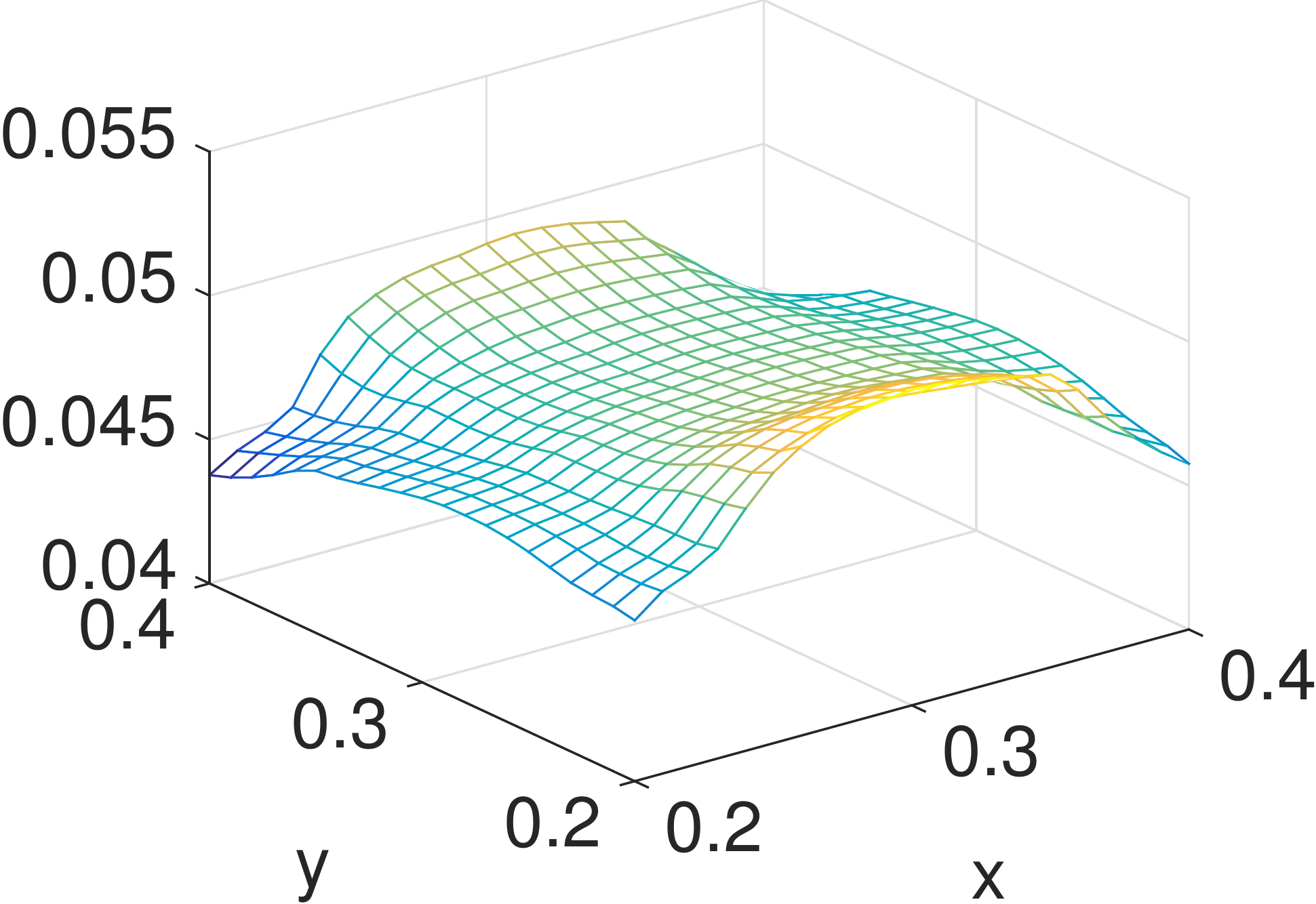}
	\includegraphics[width=0.3\textwidth]{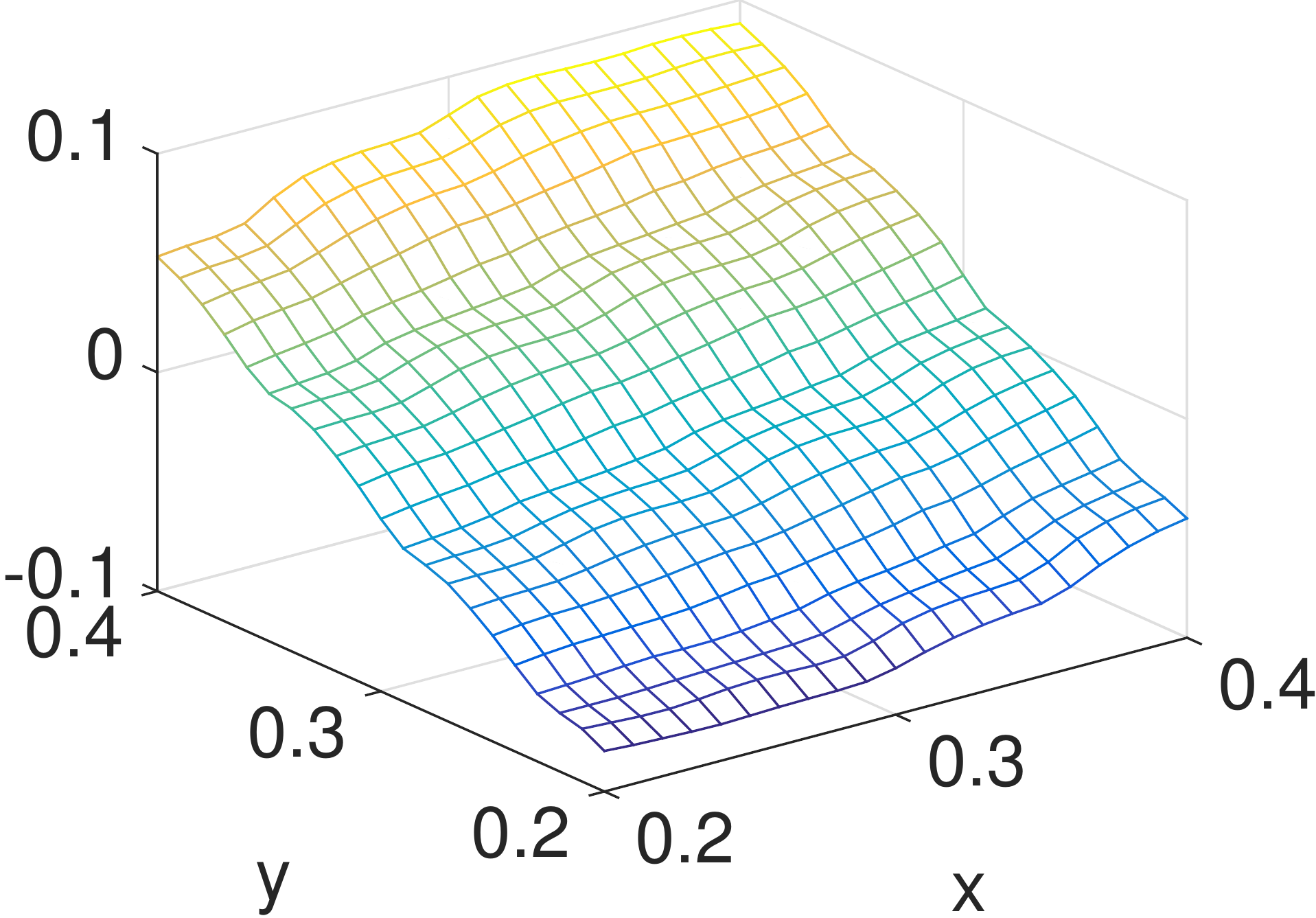}
	\includegraphics[width=0.3\textwidth]{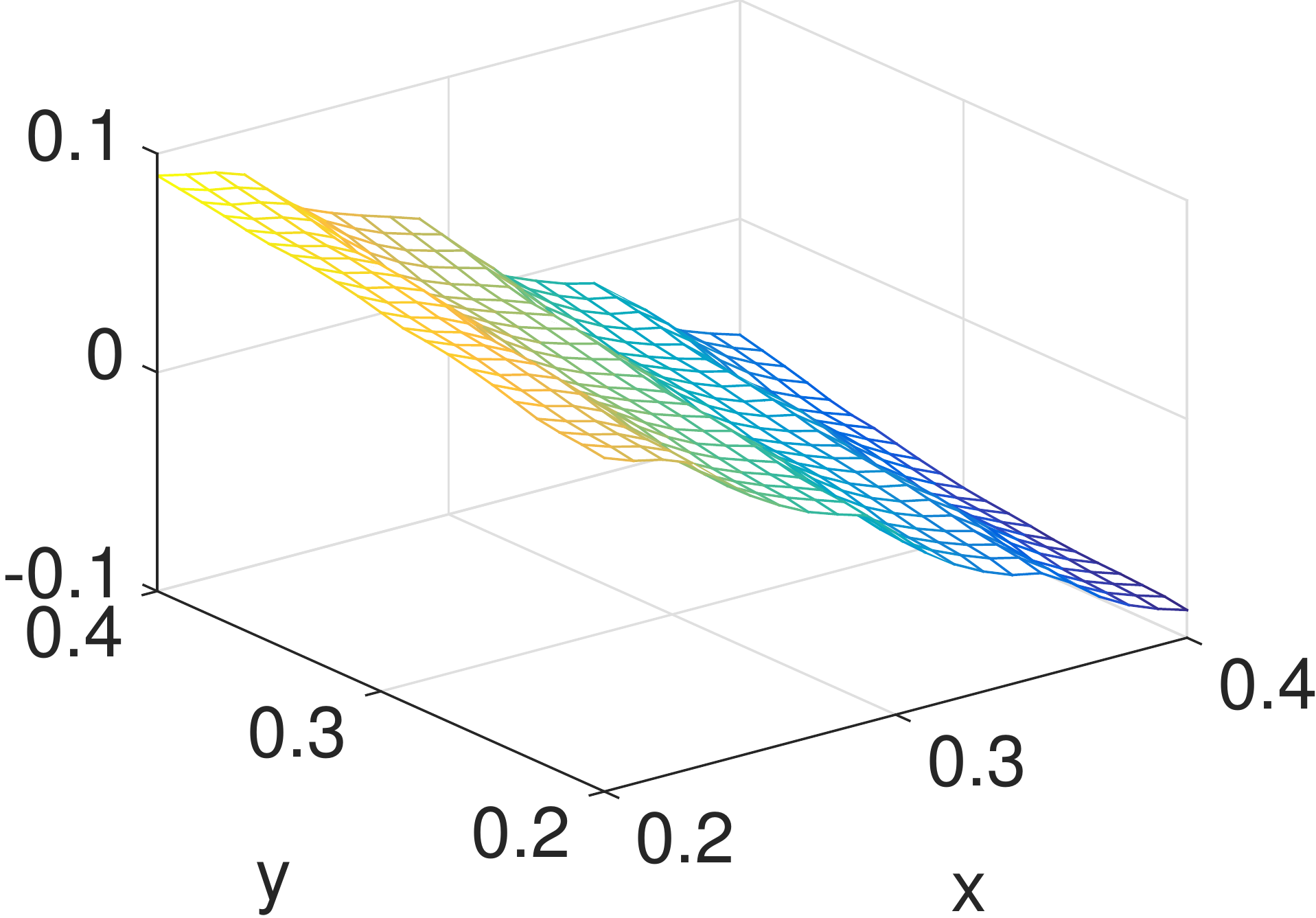}\\
	\includegraphics[width=0.3\textwidth]{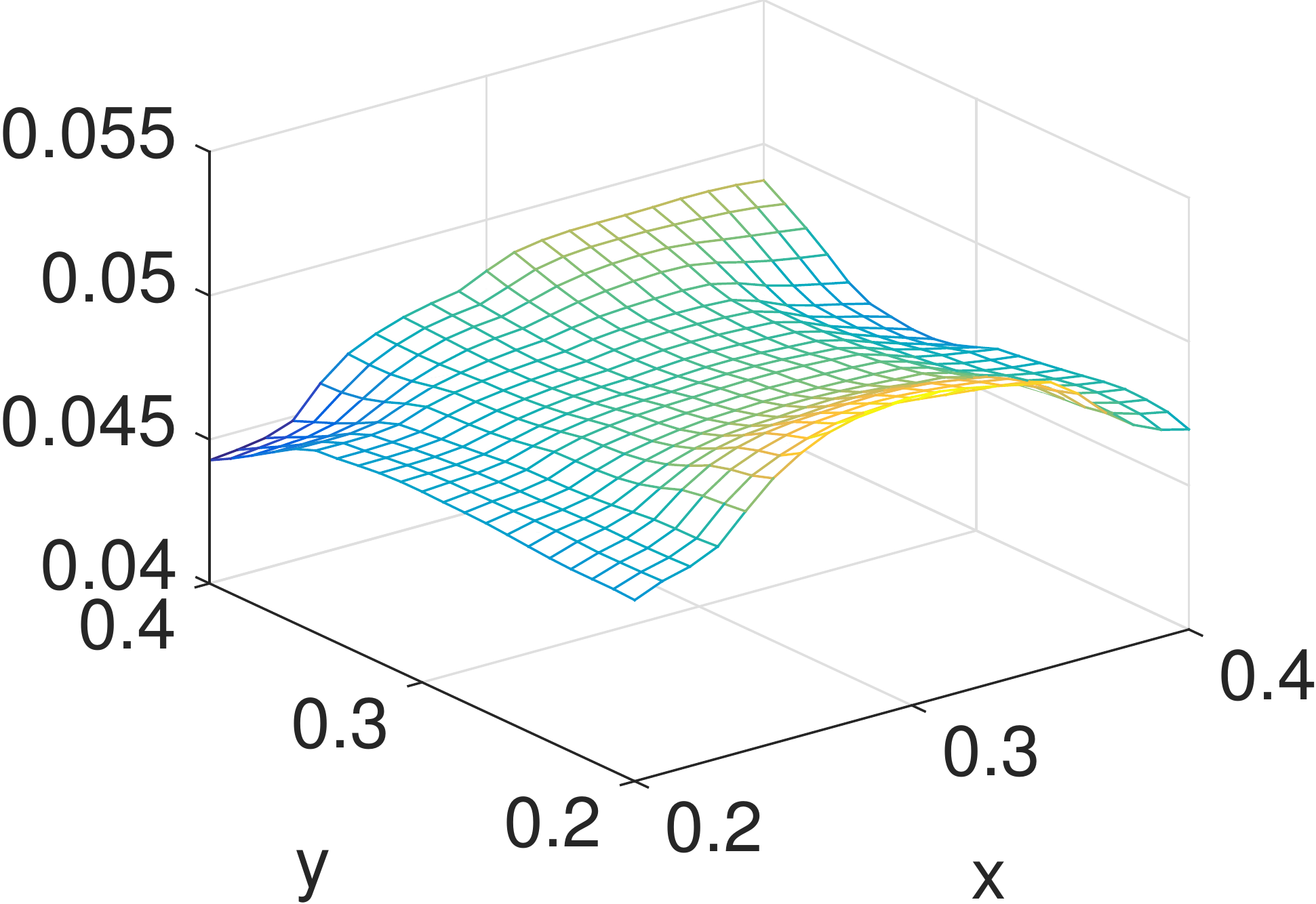}
	\includegraphics[width=0.3\textwidth]{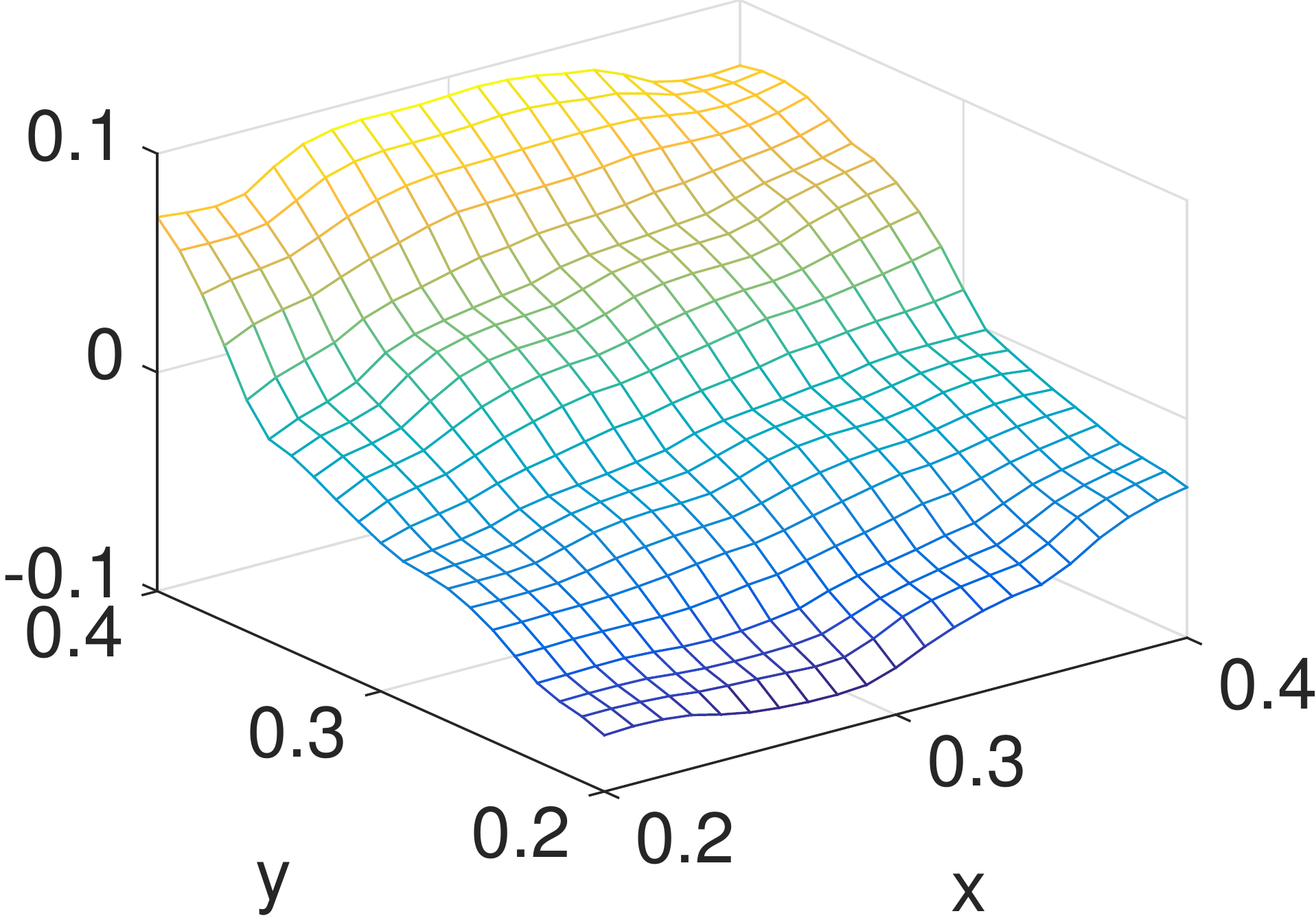}
	\includegraphics[width=0.3\textwidth]{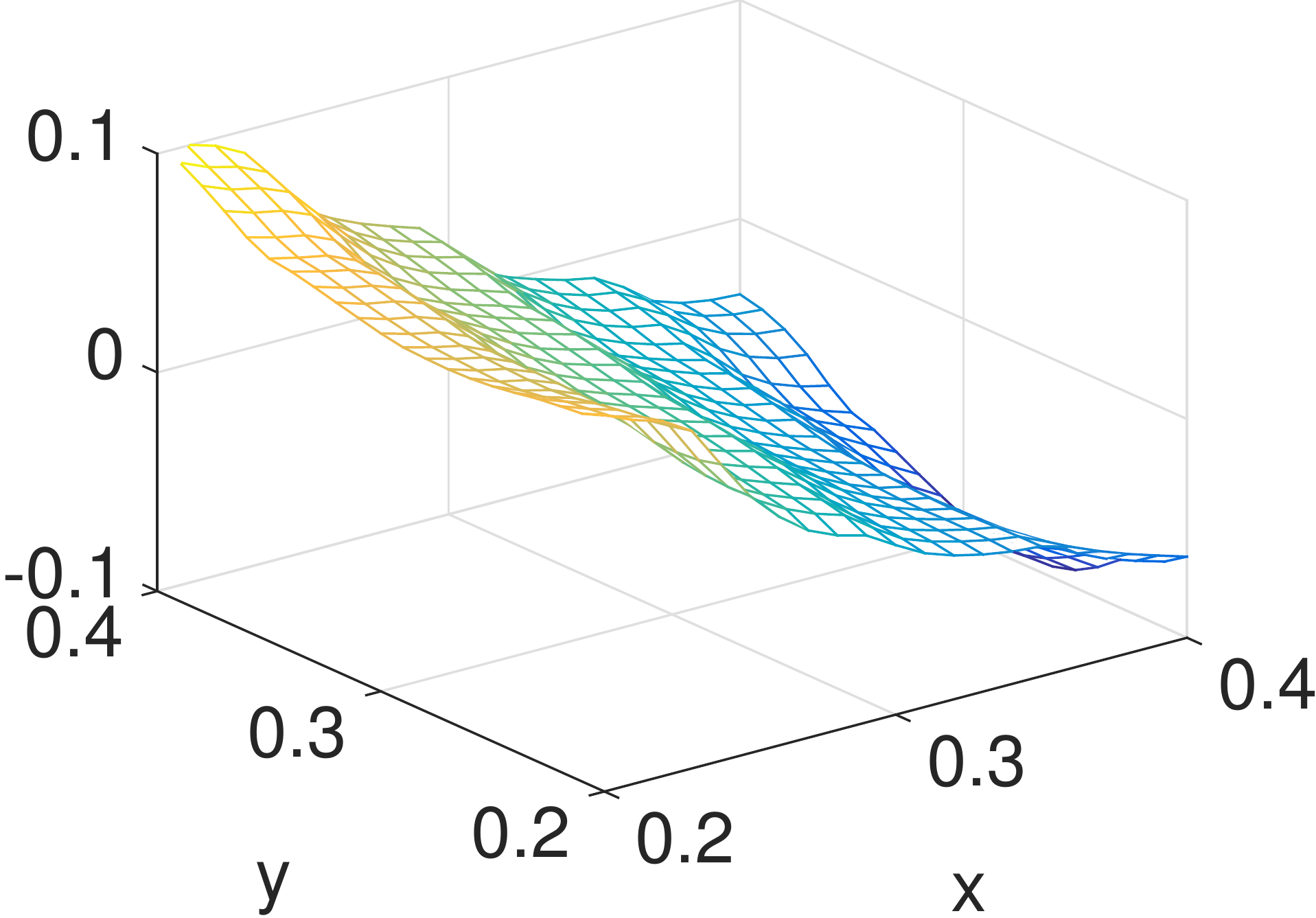}\\
	\caption{The first row shows the first three singular vectors
		of $\wt\Gmat^\vep_{2,2}$ and the second row shows projection
		of them onto $\spanop \{\Gmat^{\vep,r}_{2,2}\}$ with
		$k_{2,2}=6$. Visually, six random sampled basis are enough
		to capture the leading modes from the full-basis Green's
		matrix.
	}
	\label{fig:elliptic_local_modes}
\end{figure}

\subsubsection{Global test}
In the global test, the boundary condition is the sine function over
the boundary $\partial\Kcal$.  Equation~\eqref{eqn:elliptic_cond} is
computed with $\Gmat^{\vep}_m$ for the reference solution $u_{\text{ref}}$,
and~\eqref{eqn:elliptic_cond_reduced} is computed for the approximate
solution $u_{\text{approx}}$.~\cref{fig:elliptic_global} shows the reference
solution $u_{\text{ref}}$ along with the approximated solutions $u_{\text{approx}}$ obtained using $k_m=10$ and
$k_m=50$, respectively.  The decay in relative error 
\[
\text{relative error} = \frac{\| u_\text{ref} - u_{\text{approx}}\|_2}{\| u_{\text{ref}} \|_2}
\]
as a function of $k_m$ is plotted in~\cref{fig:elliptic_global_error}.
\begin{figure}
	\includegraphics[width=0.3\textwidth]{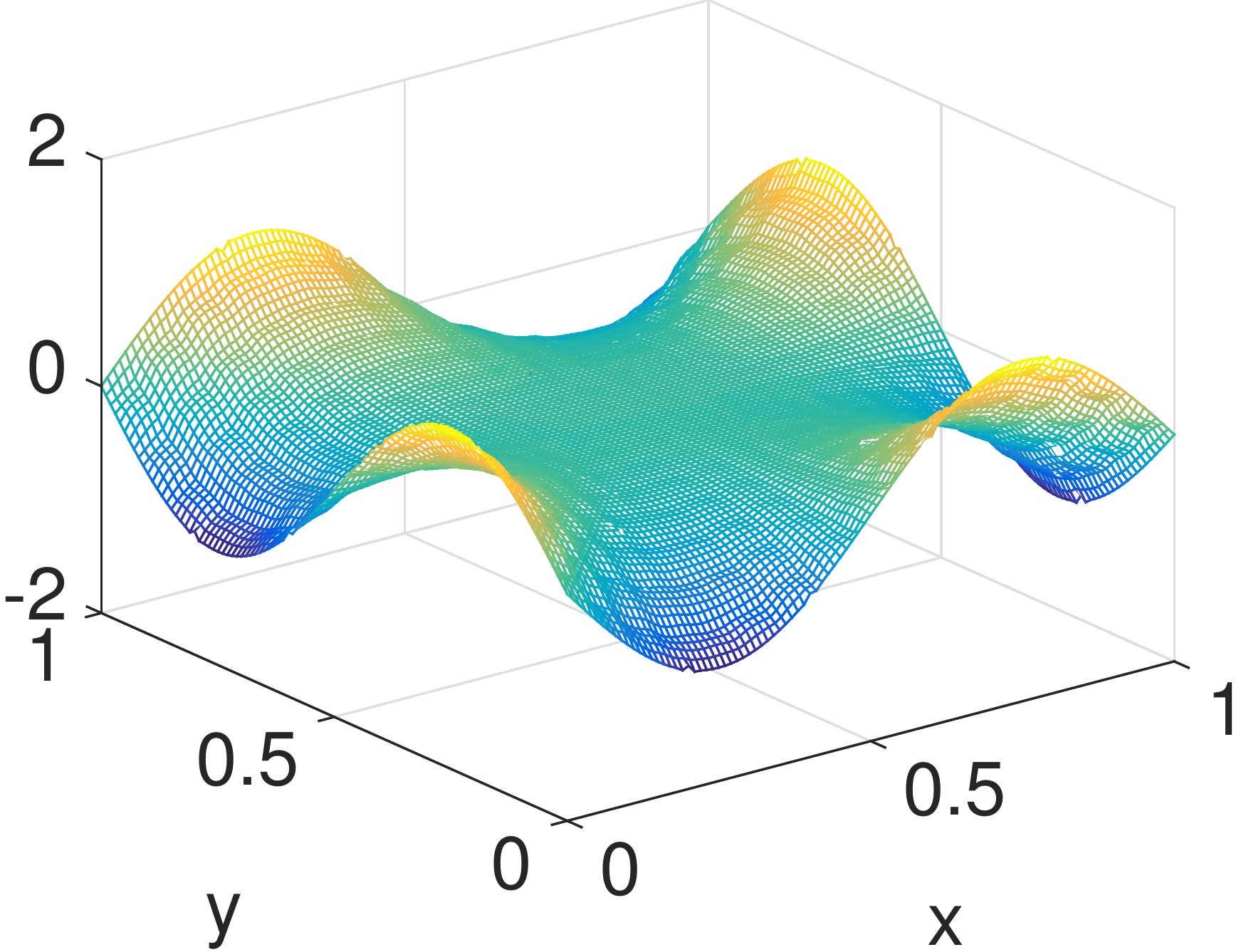}
	\includegraphics[width=0.3\textwidth]{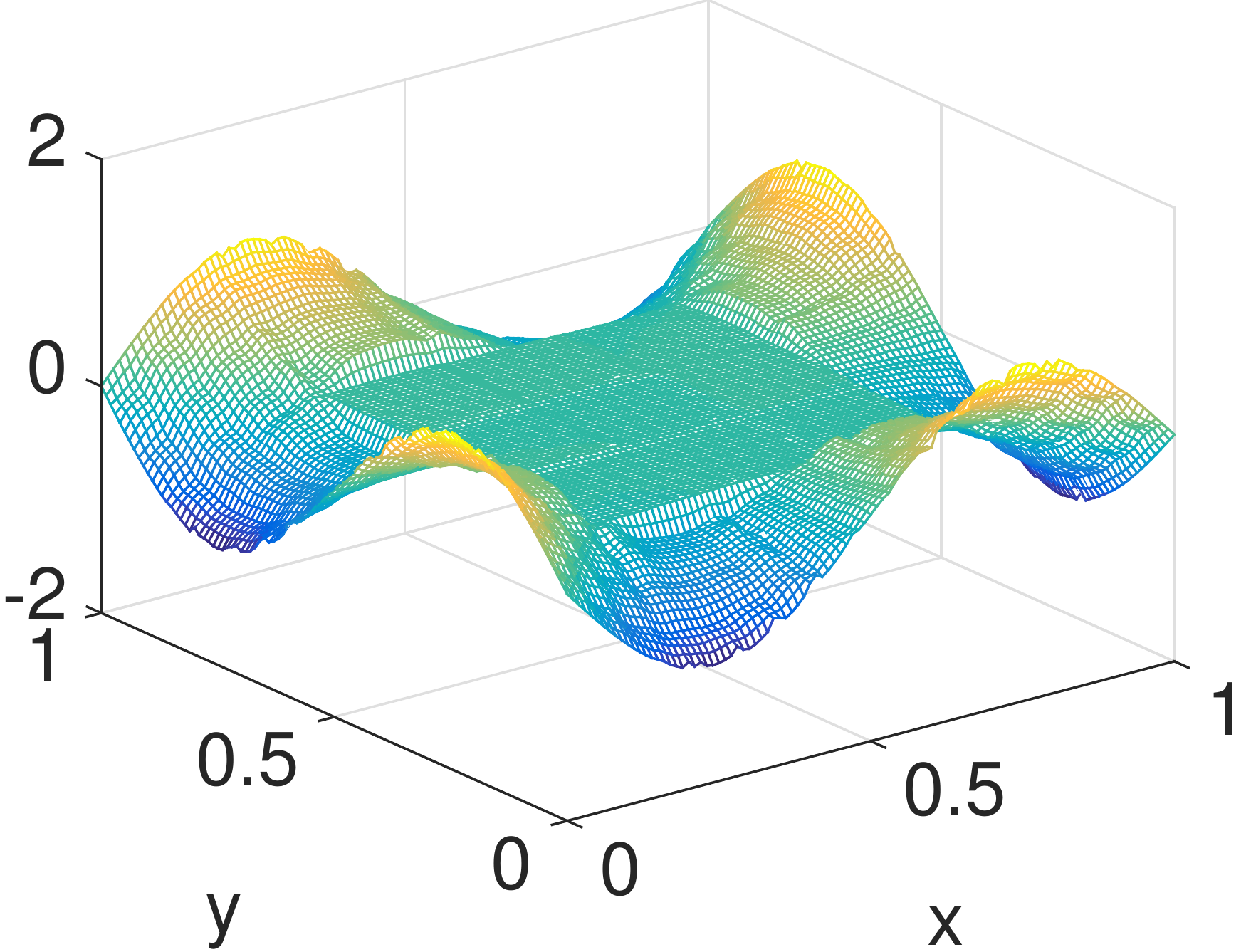}
	\includegraphics[width=0.3\textwidth]{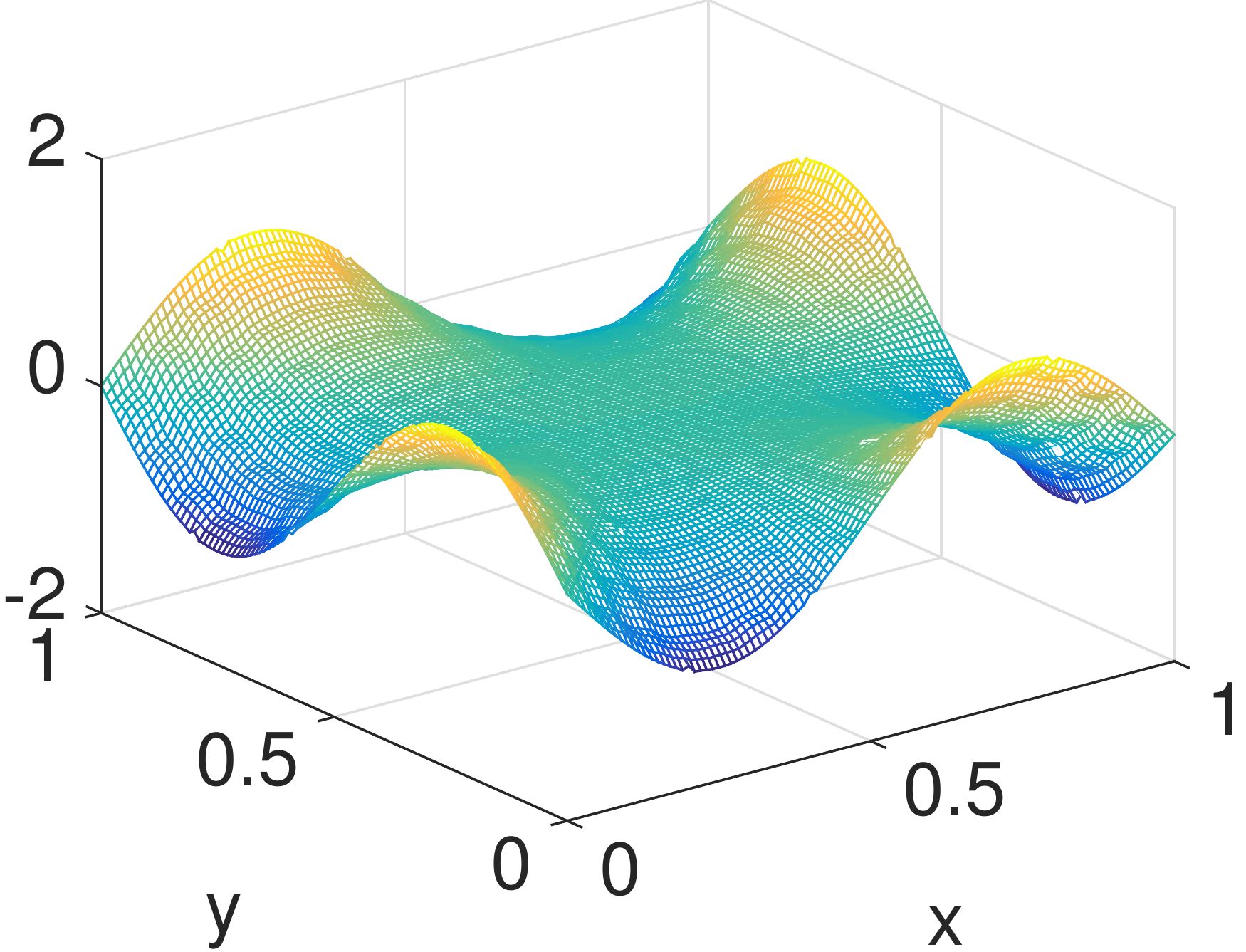}\\
	\caption{Computed solutions. Left panel shows the reference
		solution obtained with fine grids. Middle panel and right
		panel show solutions obtained from \eqref{eqn:rte_random},
		\eqref{eq:gb2}, \eqref{eq:gb3} with $k_m=10$ and $k_m=50$ (for all $m$),
		respectively.}
	\label{fig:elliptic_global}
\end{figure}
\begin{figure}
	\centering
	\includegraphics[width=0.6\textwidth]{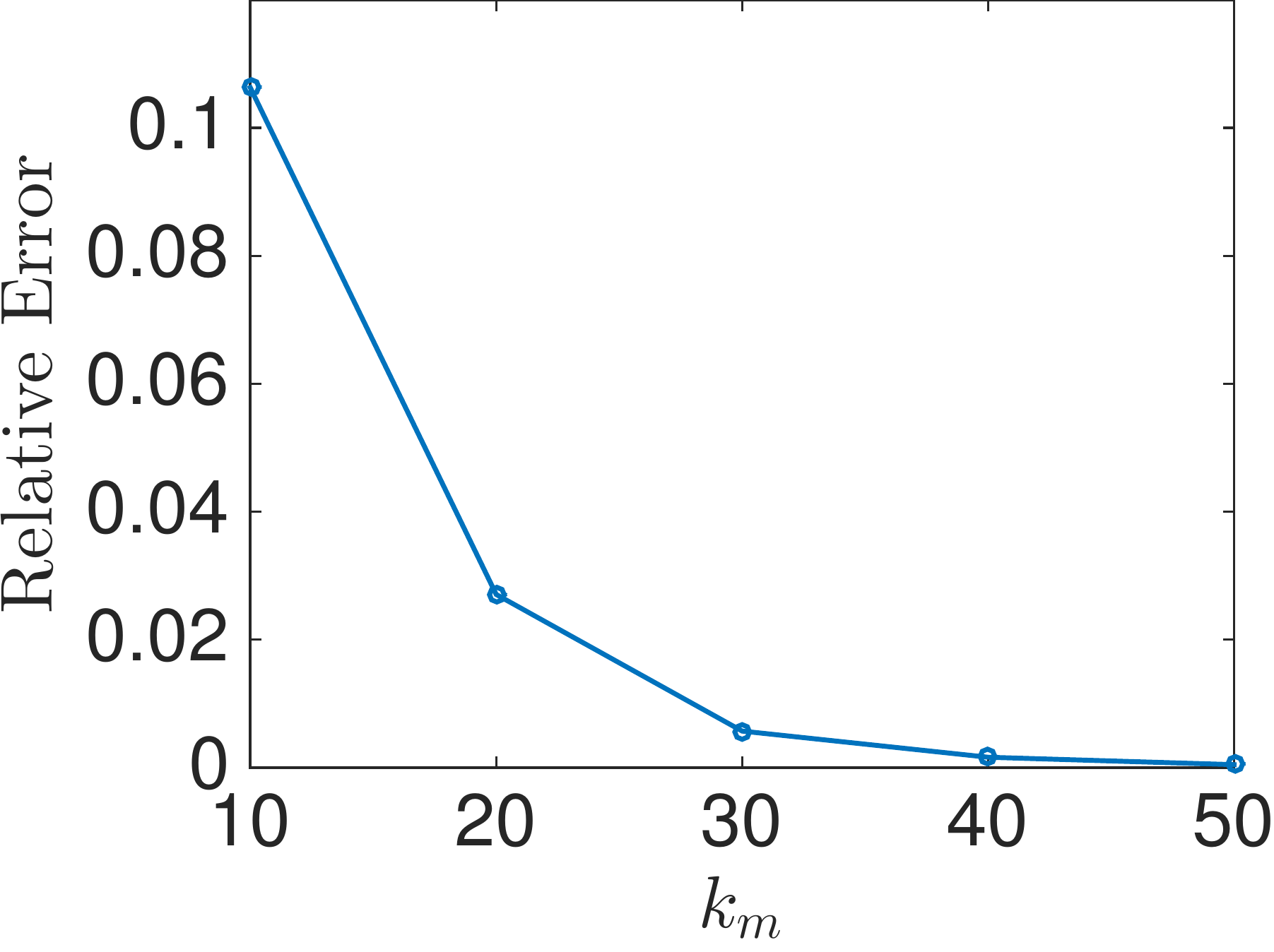}
	\caption{The global error as a function of number of random nodes per patch $k_m$. Note the rapid decay of error as $k_m$ increases.}\label{fig:elliptic_global_error}
\end{figure}

\subsection{Comparison with MsFEM and GMsFEM}
  A number of successful existing numerical
  homogenization methods share with our proposed method the property
  that that optimal basis functions are constructed offline.  MsFEM
  (Multiscale Finite Element Method)~\cite{HW97} and GMsFEM
  (Generalized MsFEM)~\cite{EFENDIEV2013116} have been used with
  success in many examples and with excellent numerical
  performance. MsFEM builds four basis functions by solving the local
  equation for $a$-harmonic functions that set $1$ at the four nodal
  points, while GMsFEM, prepares a full list of Green's functions over
  the subdomain and select the optimal ones according to a carefully
  designed spectral criterion (that translates into a generalized
  eigenvalue problem). On the theoretical level, MsFEM has been shown
  to have good convergence (see \cite{HWC99} for 
  periodic media), and the theory for GMsFEM can be found
  in~\cite{EFENDIEV2013116}. In this subsection
  we compare our methods with these two approaches, for a  challenging
  example in which the media contains both multiscale
  structures and high contrasts:
\begin{equation*}
a=1 + 1000 \, \mathbf{1}_S(x,y)\,,\quad S = \{(x,y)\in [0,1]^2:(x\cos(100\sqrt{(x-0.5)^2+(y-0.5)^2}))\leq y-0.5 \}\,.
\end{equation*} 
We plot the media in \cref{fig:HC_media}, noting that our comparison
is imperfect because the analytical result for MsFEM assumes
periodicity. Upon dividing the domain into fine
  mesh with $h=\frac{1}{100}$ and coarse mesh with $H=\frac{1}{5}$, we
  investigate the behavior of three different methods on the subdomain
  $\mathcal{K}_{2,2}$. We compute the reference optimal basis function
  by first looping over the boundary to build the
  entire Green's function list, then performing SVD. 
  \cref{fig:compare_msfem} shows that the
  random sampling method (using merely $6$ samples) can quickly
  capture the three leading basis functions and gives a higher
  accuracy, in comparison with MsFEM. In~\cref{tbl:CPUtime}, we report
  the CPU time needed for the three methods (MsFEM, GMsFEM and random
  sampling) vs the reference solution computed directly from
  performing SVD, and report the relative error in capturing the first
  three basis functions. Here the relative error is defined by:
 \begin{equation}\label{eqn:re_error}
 	\text{Error} = \frac{\| (\Imat - \Qmat_k \Qmat_k^\top) \Umat_3\|_2 }{\| \Umat_3 \|_2}\,, \quad e_i = \frac{\| (\Imat - \Qmat_k \Qmat_k^\top) \mathsf{u}_i\|_2 }{\| \mathsf{u}_i \|_2}
 \end{equation}
 where $\Qmat_k$ collects the orthonormal first $k$-basis constructed via different methods and $\Umat_n = [\mathsf{u}_1,\mathsf{u}_2,\ldots,\mathsf{u}_n]$ collects the first $n$ optimal basis functions $\mathsf{u}_i$. It is clear that GMsFEM is rather expensive while MsFEM is the cheapest of the three approaches. In terms of the error, random sampling strategies performs much better than MsFEM and similar to GMsFEM. We note that GMsFEM selects basis functions according to a spectral method reflected via a generalized eigenvalue problem. Since it has a different definition for ``optimality'', the comparison is not truly fair.

\begin{figure}
	\centering
	\includegraphics[width=0.6\textwidth]{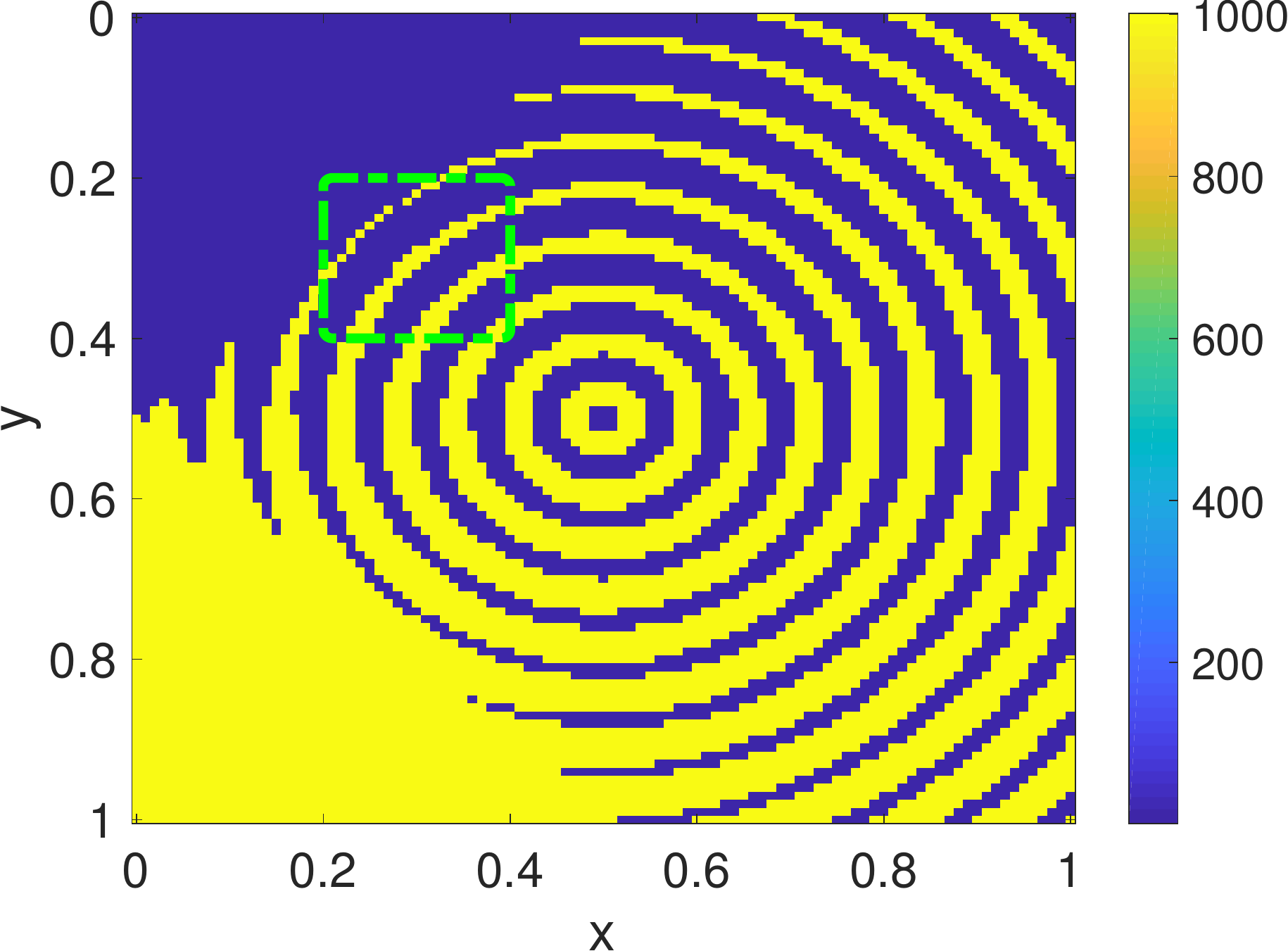}
	\caption{High contrast media with yellow part indicating $a\left(x,\frac{x}{\vep} \right)=1000$ and blue part indicating $a\left(x,\frac{x}{\vep} \right)=1$. The green box shows local patch $\Kcal_{2,2}$.} 
	\label{fig:HC_media}
\end{figure}

\begin{figure}
	\includegraphics[width=0.3\textwidth]{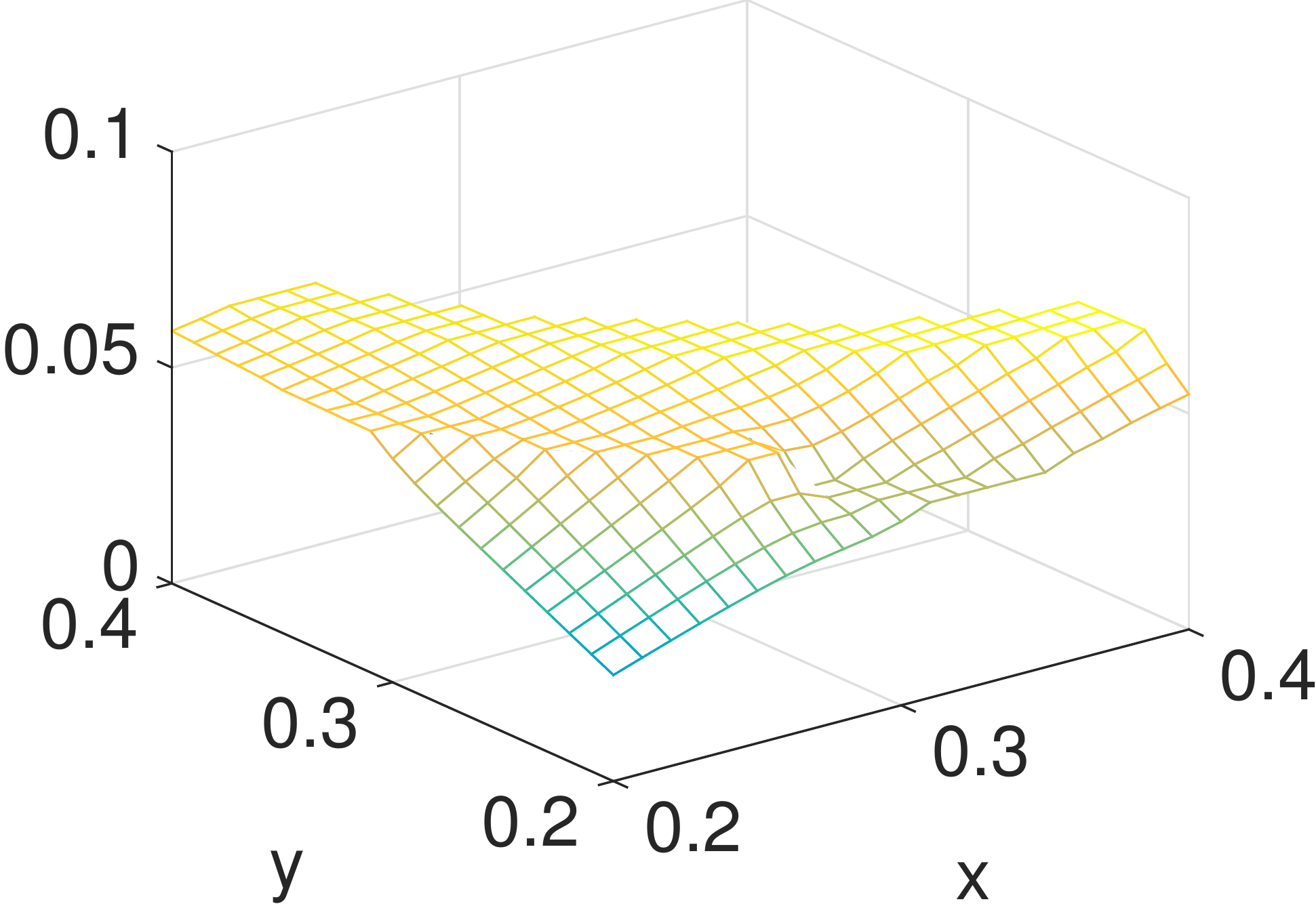}
	\includegraphics[width=0.3\textwidth]{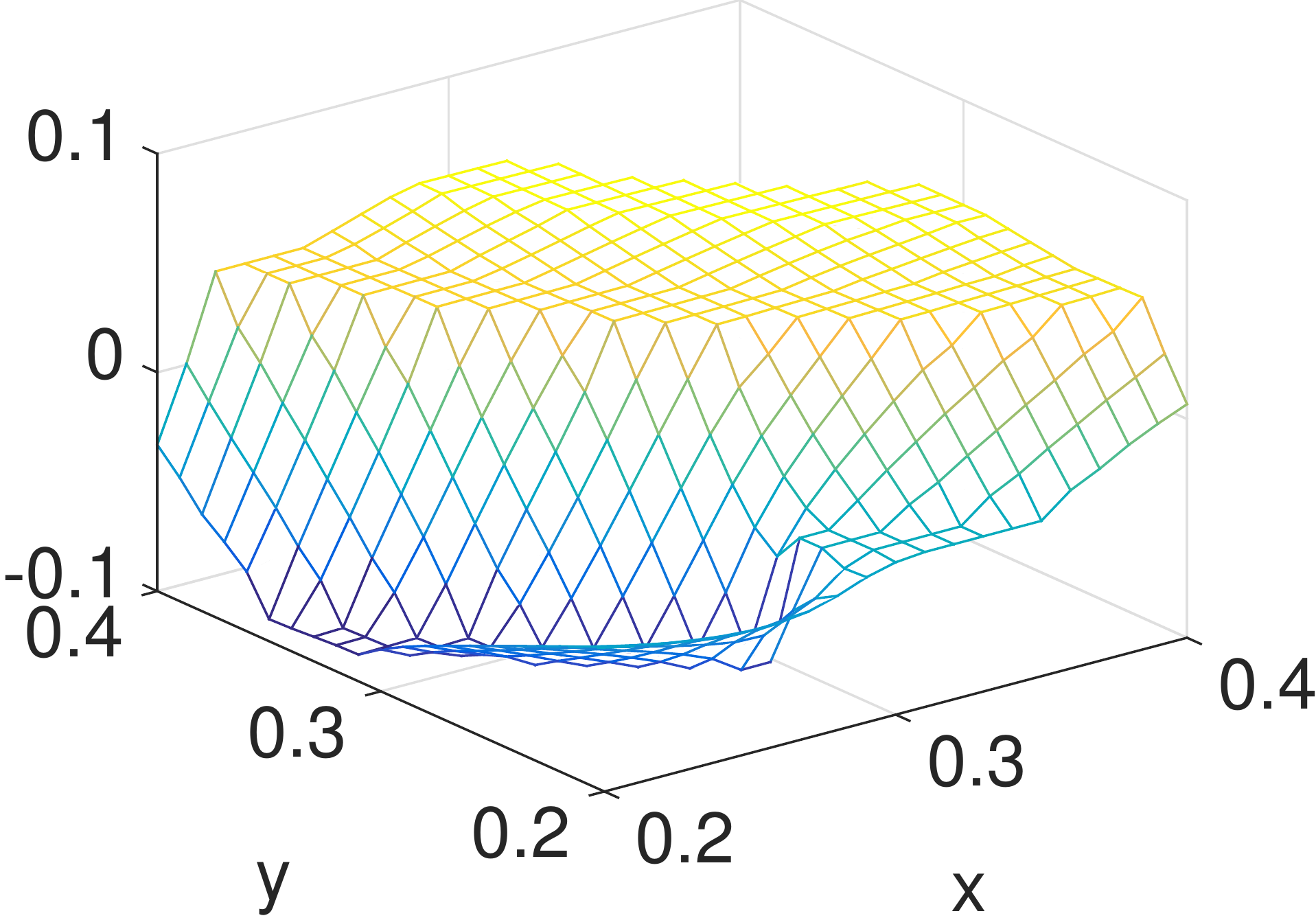}
	\includegraphics[width=0.3\textwidth]{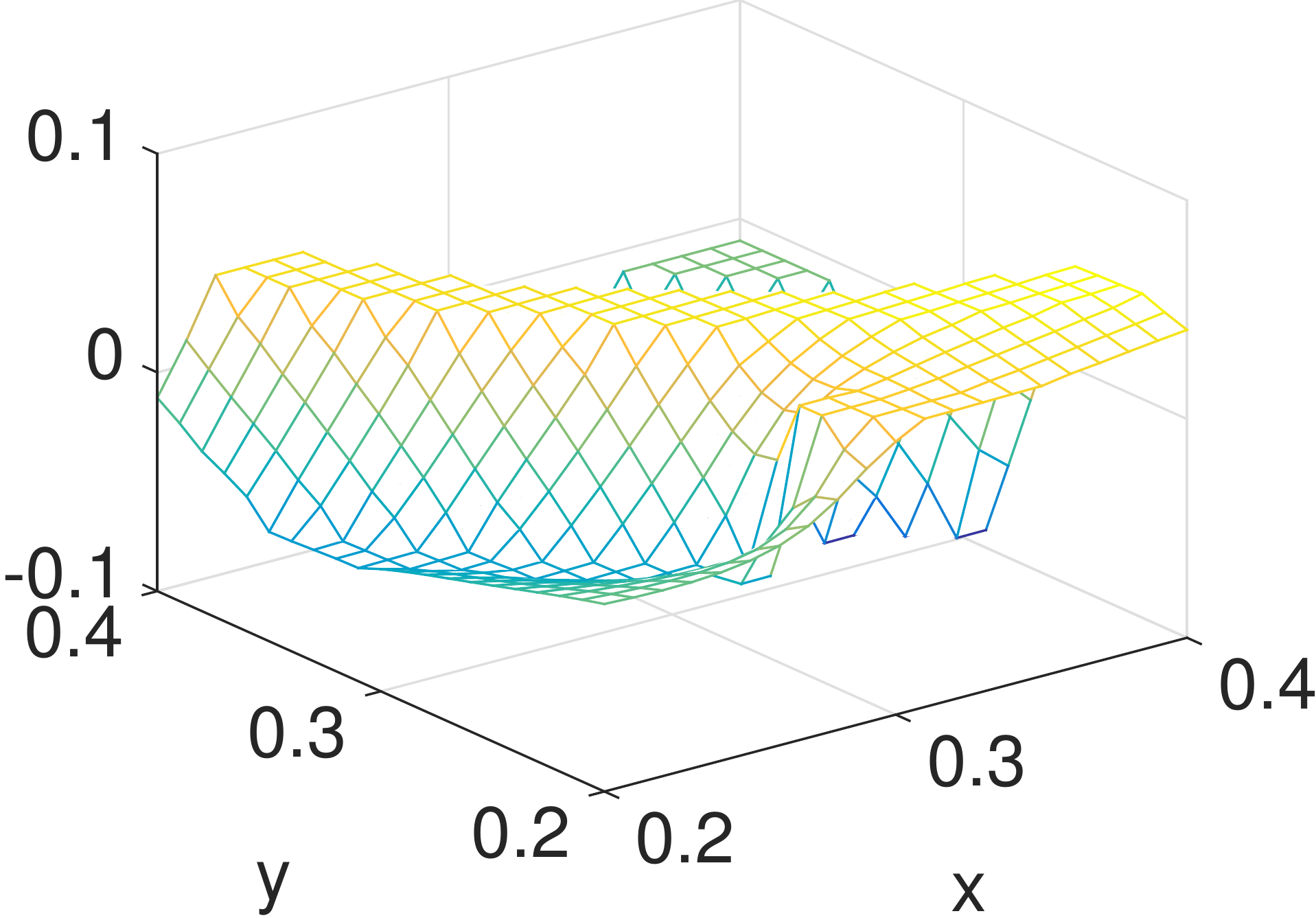}\\
	\includegraphics[width=0.3\textwidth]{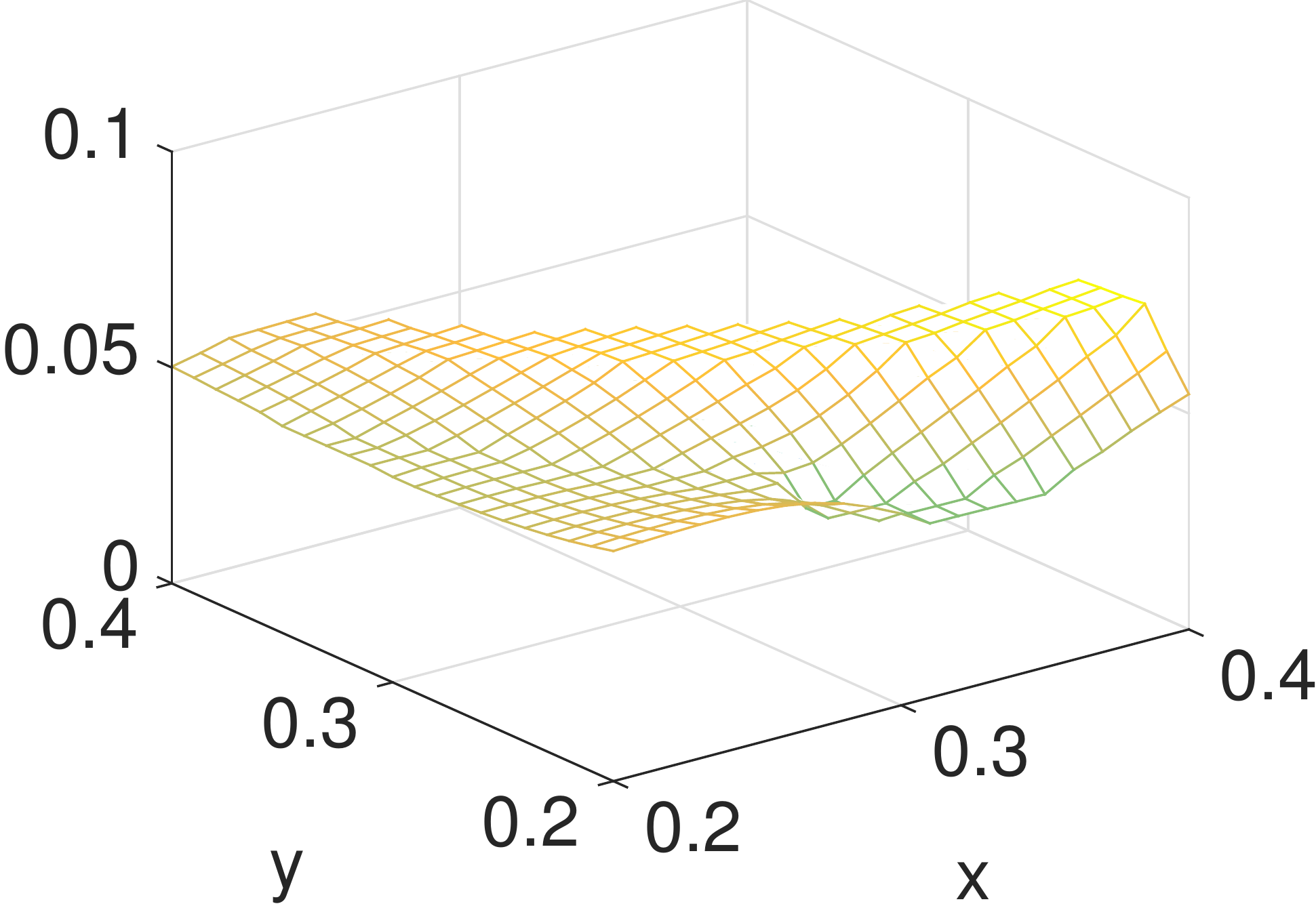}
	\includegraphics[width=0.3\textwidth]{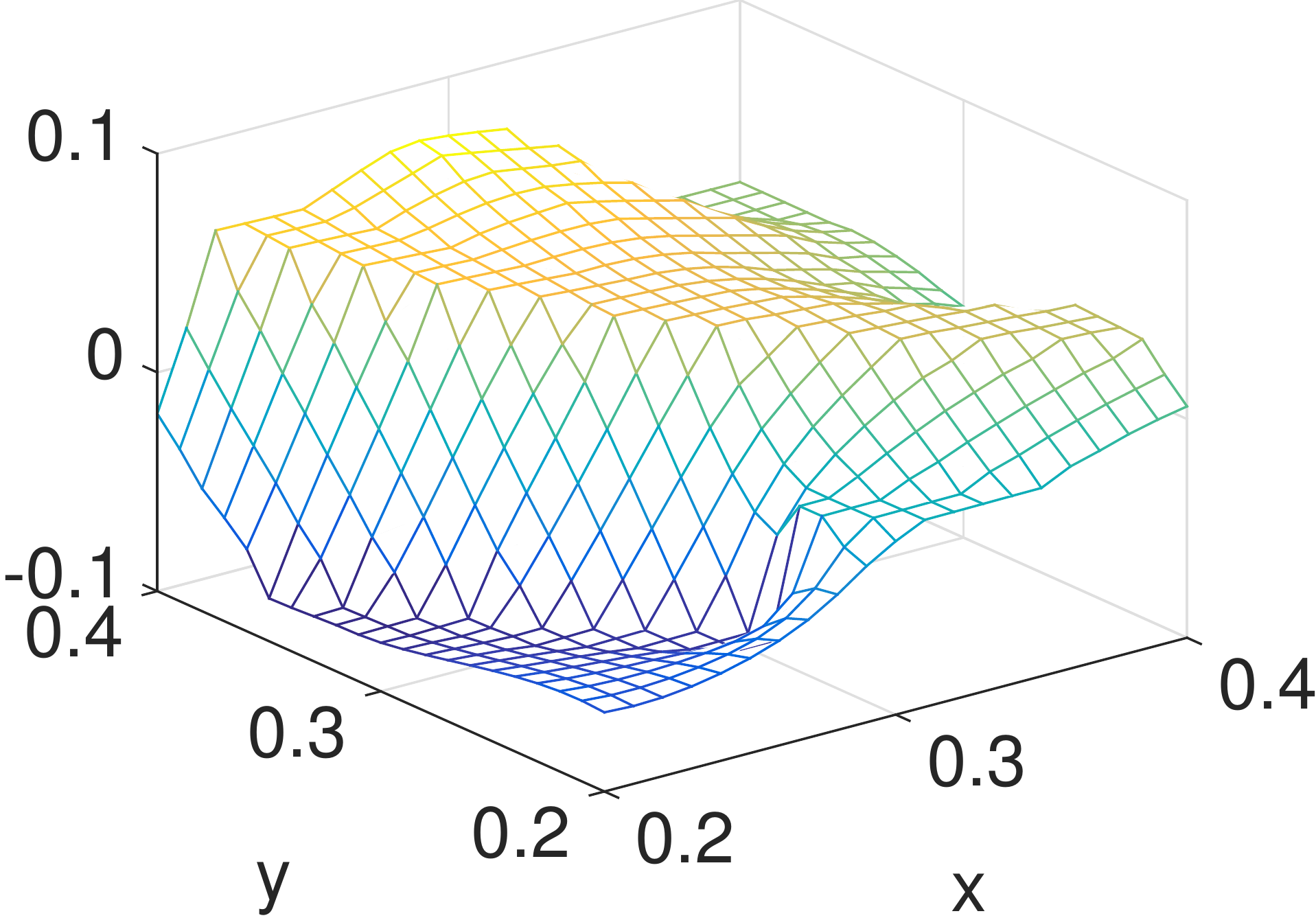}
	\includegraphics[width=0.3\textwidth]{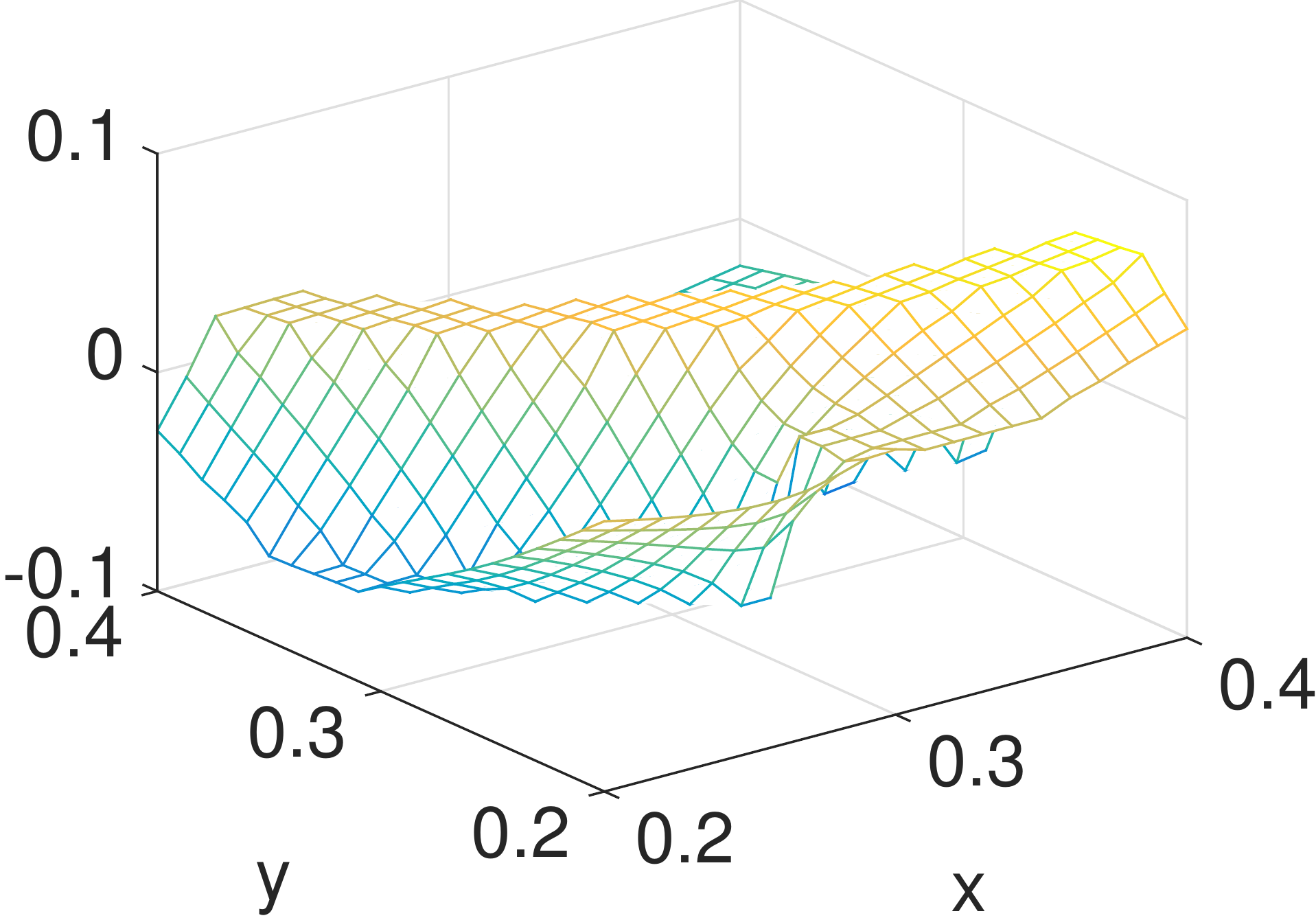}\\
	\includegraphics[width=0.3\textwidth]{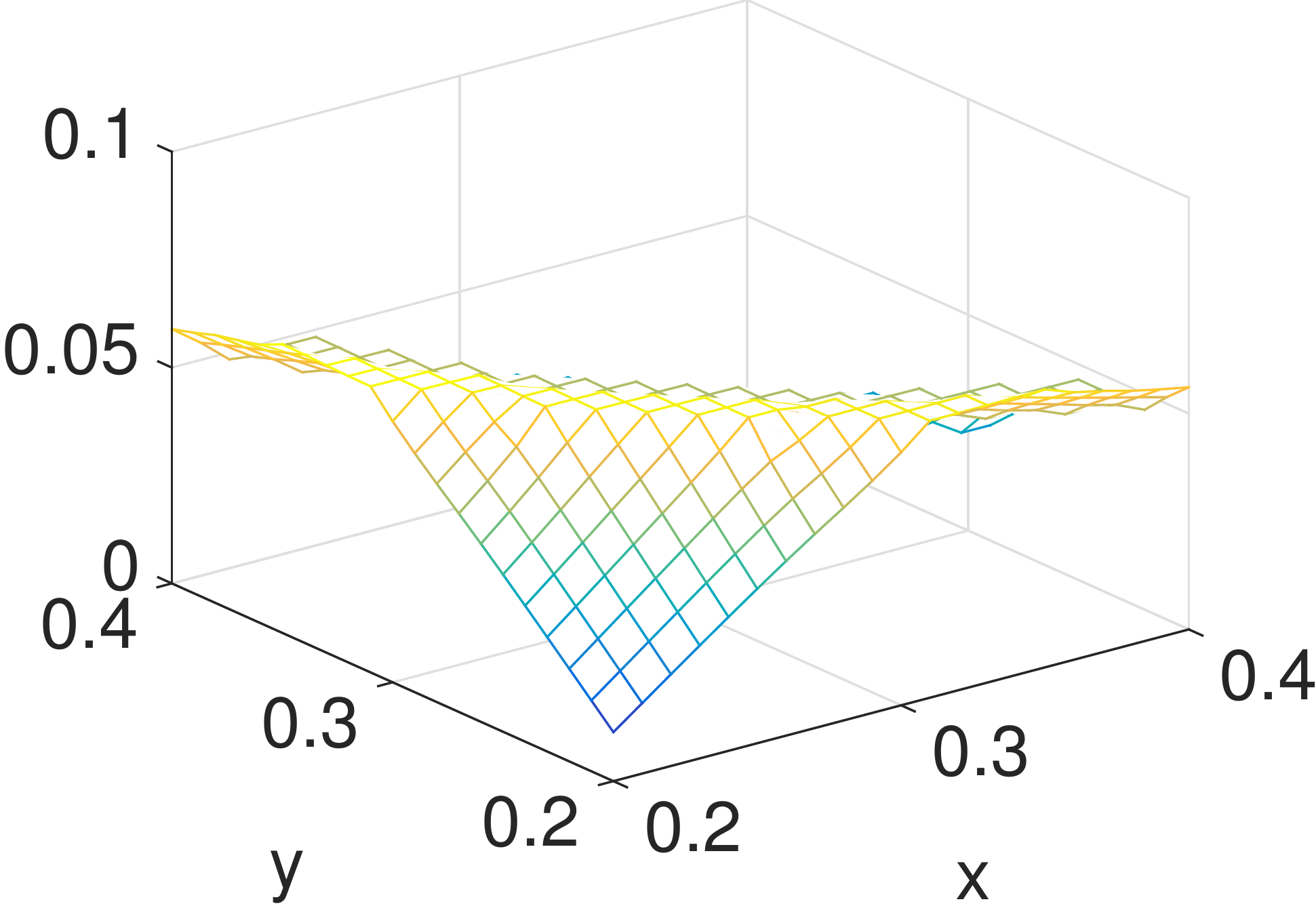}
	\includegraphics[width=0.3\textwidth]{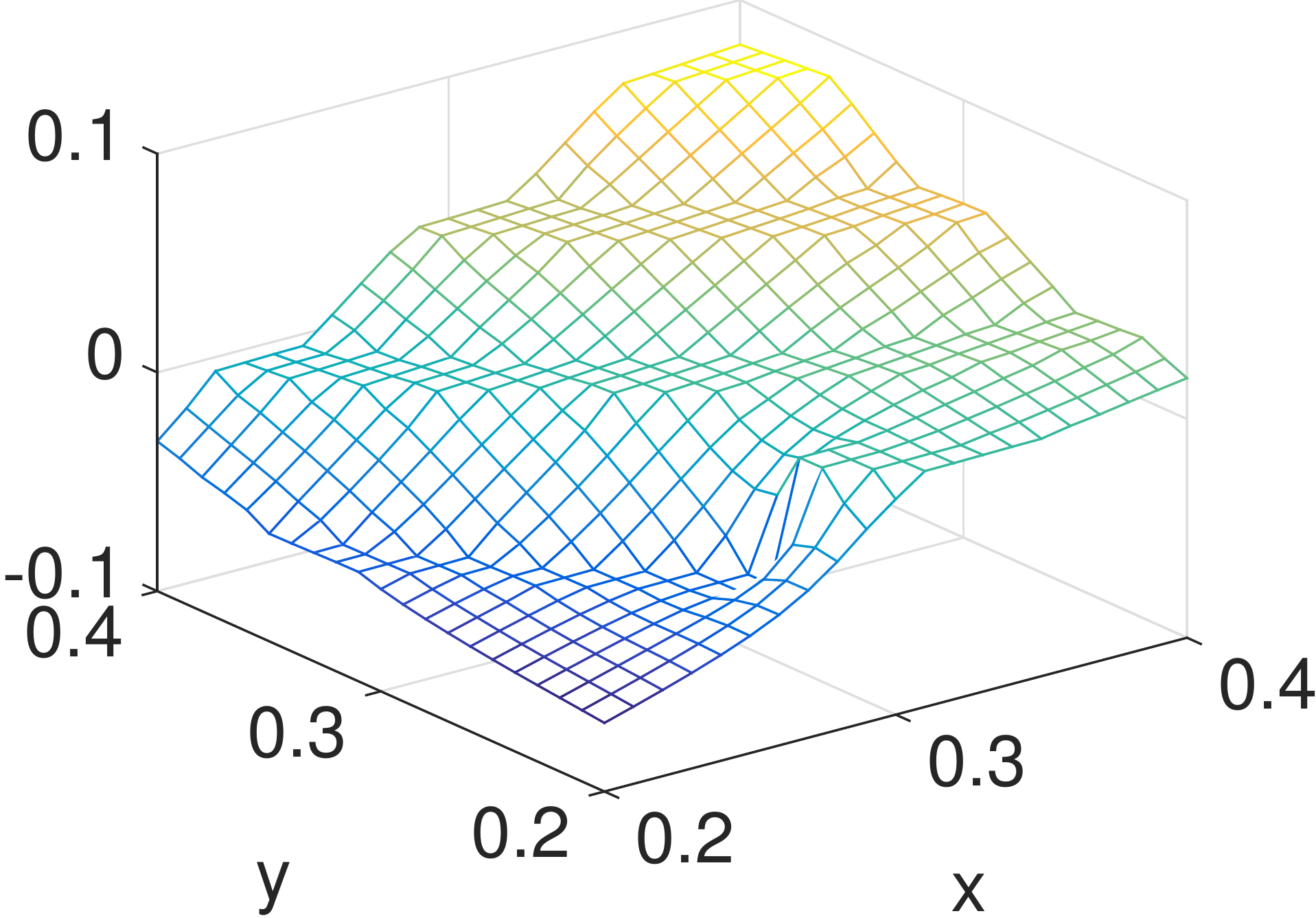}
	\includegraphics[width=0.3\textwidth]{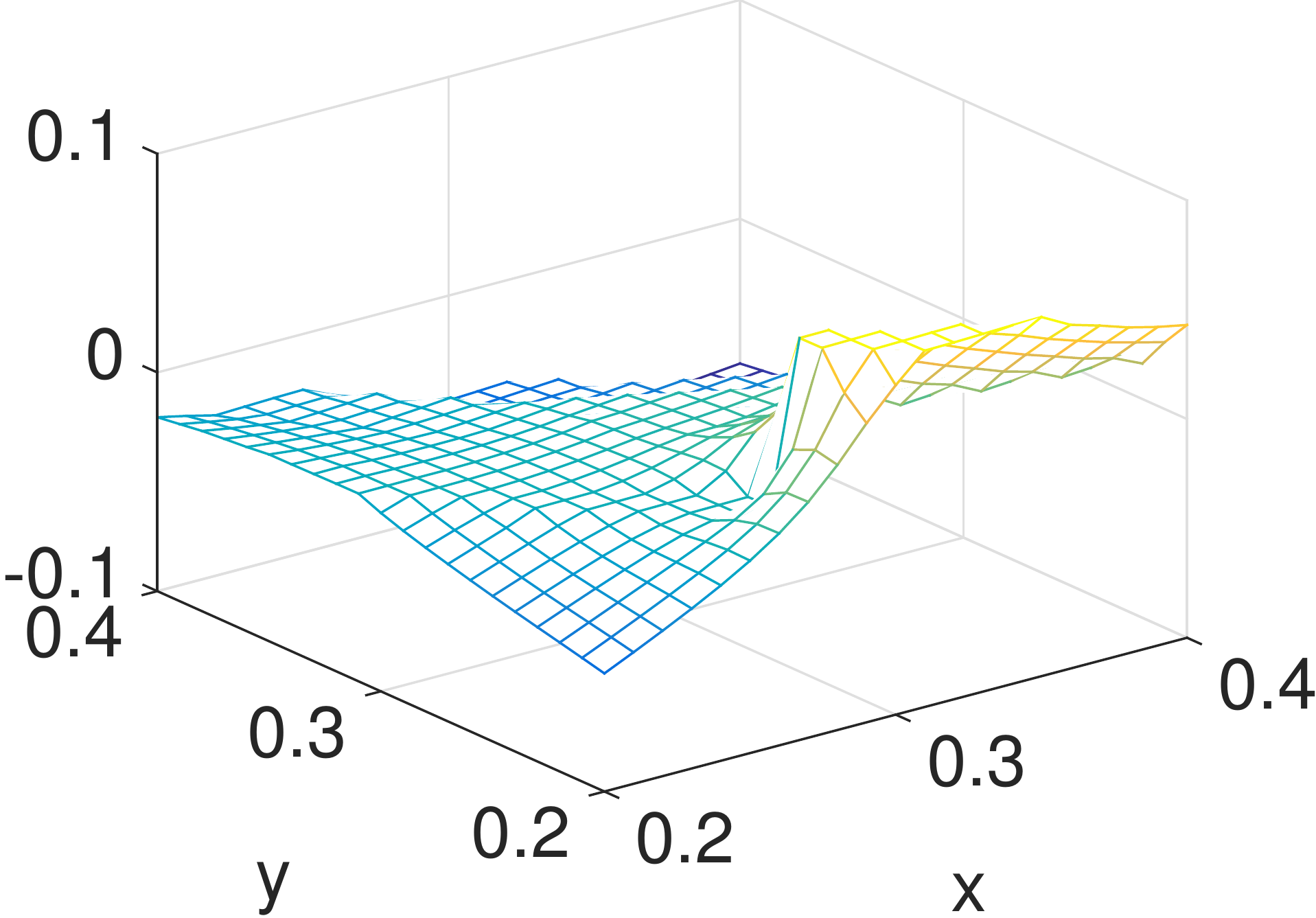}\\
	\caption{Optimal basis functions and their projections onto
		the approximate spaces. First row plots the first three singular vectors of $\wt\Gmat_{2,2}$. Second row plots their projection onto the
		space spanned by the random sampled basis with
		$k_{2,2}=6$. Third row shows projection onto MsFEM space. 
		Random sampled basis provide much better accuracy. }
	\label{fig:compare_msfem}
\end{figure}

\begin{table}[tbhp]
\centering 
		\begin{tabular}{c|c|c|c|c|c|c}
			\hline 
			\multirow{2}{*}{ }& \multirow{2}{*}{SVD (ref.)} &\multirow{2}{*}{MsFEM} & \multicolumn{3}{|c|}{GMsFEM} & \multirow{2}{*}{Random sampling} \\
			\cline{4-6}
			& & & snapshots & ensemble & spectral \\
			\hline
			CPU Time (s) & 6.6569	& 0.1663 & 7.1168 & 0.2068 & 0.0051 & 0.3164 \\
			\hline
			$e_1$ & --- &0.2043 & \multicolumn{3}{c|}{0.0867} &0.1108\\
			\hline
			$e_2$ & --- &0.5930 & \multicolumn{3}{c|}{0.1236}  &0.1101\\
			\hline
			$e_3$& --- &0.7581 &  \multicolumn{3}{c|}{0.0451} &0.0567\\
			\hline
			Error & --- &0.8206 &\multicolumn{3}{c|}{0.1557} & 0.1289\\
			\hline
		\end{tabular}
      
    \caption{CPU time and error comparison of the methods MsFEM,
      GMsFEM and random sampling (proposed method).  Error is
      defined in~\eqref{eqn:re_error}.}
\label{tbl:CPUtime}
\end{table}

\bibliographystyle{siamplain}
\bibliography{elliptic_random_lowrank}

\end{document}